\documentclass[11pt]{amsart}
\usepackage{amsmath, amssymb}
\usepackage{amsfonts}
\usepackage{mathrsfs}
\usepackage[arrow,matrix,curve,cmtip,ps]{xy}

\usepackage{amsthm}
\usepackage{mathtools} 
\usepackage{array}
\usepackage{mathdots}

\usepackage[top= 2.25cm, bottom=2cm, left = 2cm, right= 2cm]{geometry}
\usepackage{hyperref}
\usepackage{setspace}

\allowdisplaybreaks

\numberwithin{equation}{section}

\newtheorem{theorem}{Theorem}[section]
\newtheorem{lemma}[theorem]{Lemma}
\newtheorem{proposition}[theorem]{Proposition}
\newtheorem{corollary}[theorem]{Corollary}

\newtheorem{question}[theorem]{Question}
\newtheorem*{theorem*}{Theorem}
\theoremstyle{remark}
\newtheorem{remark}[theorem]{Remark}

\newtheorem{example}[theorem]{Example}

\numberwithin{equation}{section}

\newcommand{\lif}[1]{\widetilde{#1}}  

\newcommand{\com}[2][\theta]{#2^{#1}}

\newcommand{\Pbd}[1]{\Phi_{bdd}(#1)}

\newcommand{\Pdt}[1]{\Phi_{2}(#1)}

\newcommand{\cP}[1]{\bar{\Phi}(#1)}

\newcommand{\cPbd}[1]{\bar{\Phi}_{bdd}(#1)}

\newcommand{\cPdt}[1]{\bar{\Phi}_{2}(#1)}

\newcommand{\p}{\phi}

\newcommand{\Pkt}[1]{\Pi_{#1}}

\newcommand{\cPkt}[1]{\bar{\Pi}_{#1}}

\renewcommand{\r}{\pi}

\newcommand{\sH}{\bar{\mathcal{H}}}

\newcommand{\cS}[1]{\bar{S}_{#1}}

\renewcommand{\S}[1]{\mathcal{S}_{#1}}

\newcommand{\+}{\oplus}

\newcommand{\x}{\omega}

\renewcommand{\L}[1]{{}^L#1}

\newcommand{\D}[1]{\widehat{#1}}

\newcommand{\Gal}[1]{\Gamma_{#1}}

\renewcommand{\a}{\alpha}

\renewcommand{\Im}{\text{Im}\,}

\newcommand{\Ind}{\text{Ind}}

\newcommand{\Cent}{\text{Cent}}

\newcommand{\Two}{\mathbb{Z}_{2}}

\newcommand{\C}{\mathbb{C}}

\newcommand{\Int}{\text{Int}}

\newcommand{\Hom}{\text{Hom}}

\newcommand{\e}{\varepsilon}

\newcommand{\Rep}{\text{Rep}}

\newcommand{\Jac}{\text{Jac}}

\begin{document}
\title{On the cuspidal support of discrete series for p-adic quasisplit $Sp(N)$ and $SO(N)$}

\author{Bin Xu}

\address{Department of Mathematics and Statistics \\  University of Calgary, Canada }
\email{bin.xu2@ucalgary.ca}



\maketitle

\begin{abstract}
Zelevinsky's classification theory of discrete series of $p$-adic general linear groups has been well known. M{\oe}glin and Tadi{\'c} gave the same kind of theory for $p$-adic classical groups, which is more complicated due to the occurrence of nontrivial structure of L-packets. Nonetheless, their work is independent of the endoscopic classification theory of Arthur (also Mok in the unitary case), which concerns the structure of L-packets in these cases. So our goal in this paper is to make more explicit the connection between these two very different types of theories. To do so, we reprove the results of M{\oe}glin and Tadi{\'c} in the case of quasisplit symplectic groups and orthogonal groups by using Arthur's theory. 
\end{abstract}

\section{Introduction}
\label{sec: introduction}

Let $F$ be a $p$-adic field and $G$ be a quasisplit connected reductive group over $F$. 
We consider pairs $(M, \r_{cusp})$ for $G$, where $M$ is a Levi subgroup of $G$ and $\r_{cusp}$ is an irreducible supercuspidal representation of $M(F)$. Such pairs carry an action of $G(F)$ by conjugation, i.e.,
\[
(M, \r_{cusp})^{g} = (M^{g}, \r_{cusp}^{g}),
\]
where $M^{g} = g^{-1} M g$, and $\r_{cusp}^{g}(m) = \r_{cusp}(g m g^{-1})$ for $m \in M(F)^{g}$. For any pair $(M, \r_{cusp})$, let $P = MN$ be any parabolic subgroup containing $M$, and we have the normalized parabolic induction $\Ind^{G}_{P}(\r_{cusp} \otimes 1_{N})$. For simplicity we always abbreviate this to $\Ind^{G}_{P}(\r_{cusp})$, and we have the following facts about the parabolic induction.

\begin{enumerate}
\item $\Ind^{G}_{P}(\r_{cusp})$ is a smooth representation of finite length, i.e., the semi-simplification $s.s.\Ind^{G}_{P}(\r_{cusp})$ of $\Ind^{G}_{P}(\r_{cusp})$ is a direct sum of finitely many irreducible smooth representations.
\item For $g \in G(F)$, $s.s.\Ind^{G}_{P}(\r_{cusp}) = s.s.\Ind^{G}_{P'}(\r_{cusp}^{g})$ for any parabolic subgroup $P'$ containing $M^{g}$.
\end{enumerate}

It is a theorem of Bernstein and Zelevinsky \cite{BernsteinZelevinsky:1977} that all irreducible smooth representations of $G(F)$ can be constructed by parabolic induction from supercuspidal representations.

\begin{theorem}[Bernstein-Zelevinsky]
\label{thm: B-Z}
For any irreducible smooth representation $\r$ of $G(F)$, there exists a unique pair $(M, \r_{cusp})$ up to conjugation by $G(F)$ such that 
\[
\r \subseteq s.s.\Ind^{\,G}_{P}(\r_{cusp}).
\] 
Moreover, one can always find $P'$ containing $M$ such that 
\[
\r \hookrightarrow \Ind^{\,G}_{P'}(\r_{cusp})
\] 
as a subrepresentation.
\end{theorem}

\begin{remark}
The $G(F)$-conjugacy class of pairs $(M, \r_{cusp})$ in this theorem is called the cuspidal support of $\r$. For our later purposes, we would like to fix a Borel subgroup $B$ of $G$ together with a maximal torus $T \subseteq B$, and we have the standard parabolic subgroups $P = MN$, i.e., $P \supseteq B, M \supseteq T$. Then this theorem implies that for any irreducible smooth representation $\r$ of $G$, one can always find a standard parabolic subgroup $P = MN$ with a supercuspidal representation $\r_{cusp}$ of $M$ such that $\r \hookrightarrow \Ind^{G}_{P}(\r_{cusp})$ as a subrepresentation.
\end{remark}

Based on this theorem, it is natural to ask the following questions.

\begin{question}
\label{que: cuspidal support}
How to determine the cuspidal support of any irreducible smooth representation of $G(F)$?
\end{question}

\begin{question}
\label{que: unitary dual}
How can one classify the irreducible {\bf unitary} representations of $G(F)$ in terms of their cuspidal supports?
\end{question}
 
\begin{question}
\label{que: discrete series}
How can one classify the irreducible {\bf discrete series} representations of $G(F)$ in terms of their cuspidal supports?
\end{question}

Question~\ref{que: cuspidal support} is properly the most difficult one, and we are not able to say much about it here. Question~\ref{que: unitary dual} is often referred to as the {\bf unitary dual problem}, and it has been solved for $GL(n)$ \cite {Tadic:1986}. For the classical groups, Tadi{\'c} and Mui\'c have done many works (see \cite{Tadic1:2009}, \cite{MuicTadic:2011}), and again we will not say anything about it here. Our main interest is in Question~\ref{que: discrete series}, and it has the most complete theories for both $GL(n)$ (see \cite{Zelevinsky:1980}) and classical groups (see \cite{Moeglin:2002}, \cite{MoeglinTadic:2002}). Our goal is to present the results for the quasisplit symplectic groups and special orthogonal groups. To be more precise about what we want to show, we need to introduce some notations.

If $G = GL(n)$, let us take $B$ to be the group of upper-triangular matrices and $T$ to be the group of diagonal matrices, then the standard Levi subgroup $M$ can be identified with $GL(n_{1}) \times \cdots \times GL(n_{r})$ through
\[
\begin{pmatrix}
GL(n_{1})&& \\
&\ddots & \\
&& GL(n_{r})\\
\end{pmatrix}
\]
\begin{align*}
(g_{1}, \cdots, g_{r}) \longrightarrow \text{diag}\{g_{1}, \cdots, g_{r}\}
\end{align*}
with respect to any partition of $n = n_{1} + \cdots + n_{r}$.
So an irreducible supercuspidal representation $\r_{cusp}$ of $M(F)$ can be written as $\r_{cusp} = \r_{1} \otimes \cdots \otimes \r_{r}$ where $\r_{i}$ is an irreducible supercuspidal representation of $GL(n_{i}, F)$ for $1 \leqslant i \leqslant r$. For simplicity, we denote the normalized parabolic induction $\Ind_{P}^{G} (\r_{cusp})$ by $\r_{1} \times \cdots \times \r_{r}$. An irreducible supercuspidal representation $\r$ of $GL(n, F)$ can always be written in a unique way as $\rho ||^{x} : = \rho \otimes |\det(\cdot)|^{x}$ for an irreducible unitary supercuspidal representation $\rho$ and a real number $x$. To fix notations, we will always denote by $\rho$ an irreducible unitary supercuspidal representation of $GL(d_{\rho}, F)$. Now for a finite length arithmetic progression of real numbers of common length $1$ or $-1$
\[
x, \cdots, y
\]
and an irreducible unitary supercuspidal representation $\rho$ of $GL(d_{\rho})$, it is a general fact that 
\[
\rho||^{x} \times \cdots \times \rho||^{y}
\]
has a unique irreducible subrepresentation, denoted by $<\rho; x, \cdots, y>$ or $<x, \cdots, y>$. If $x \geqslant y$, it is called a Steinberg representation; if $x < y$, it is called a Speh representation. Such sequence of ordered numbers is called a {\bf segment} (cf. Appendix~\ref{sec: reducibility for GL(n)}). 
In particular, when $x = -y > 0$, we can let $a = 2x + 1 \in \mathbb{Z}$ and write 
\[
St(\rho, a) := <\frac{a-1}{2}, \cdots, -\frac{a-1}{2}>,
\]
which is an irreducible smooth representation of $GL(ad_{\rho}, F)$. In fact it is a discrete series representation by Zelevinksy's classification theorem.

\begin{theorem}[Zelevinsky \cite{Zelevinsky:1980}]
\label{thm: GL(n) classification}
All irreducible discrete series representations of $GL(n, F)$ can be obtained in a unique way as $St(\rho, a)$ for certain irreducible unitary subpercuspidal representation of $GL(d_{\rho})$ and integer $a$ so that $n = ad_{\rho}$.
\end{theorem}

If $G = Sp(2n)$, let us define it with respect to 
\[
\begin{pmatrix} 
      0 & -J_{n} \\
      J_{n} &  0 
\end{pmatrix},
\]
where 
\[
J_{n} = 
\begin{pmatrix}
        &&1\\
        &\iddots&\\
        1&&&
\end{pmatrix}.
\]
We take $B$ to be subgroup of upper-triangular matrices in $G$ and $T$ to be subgroup of diagonal matrices in $G$, then the standard Levi subgroup $M$ 
can be identified with $GL(n_{1}) \times \cdots \times GL(n_{r}) \times G_{-}$ for any partition $n = n_{1} + \cdots + n_{r} + n_{-}$ and $G_{-} = Sp(2n_{-})$ as follows
\[
\begin{pmatrix}
GL(n_{1})&&&&&&0 \\
&\ddots &&&&& \\
&& GL(n_{r})&&&&\\
&&&G_{-} &&&\\
&&&&GL(n_{r})&& \\
&&&&&\ddots&\\
0&&&&&&GL(n_{1})
\end{pmatrix}
\]
\begin{align}
\label{eq: embedding}
(g_{1}, \cdots g_{r}, g) \longrightarrow \text{diag}\{g_{1}, \cdots, g_{r}, g, {}_tg^{-1}_{r}, \cdots, {}_tg^{-1}_{1}\},
\end{align}
where ${}_tg_{i} = J_{n_{i}}{}^tg_{i}J^{-1}_{n_{i}}$ for $1 \leqslant i \leqslant r$. Note $n_{-}$ can be $0$, in which case we simply write $Sp(0) = 1$. So an irreducible supercuspidal representation $\r_{cusp}$ of $M(F)$ can be written as $\r_{cusp} = \r_{1} \otimes \cdots \otimes \r_{r} \otimes \sigma$ where $\r_{i}$ is an irreducible supercuspidal representation of $GL(n_{i}, F)$ for $1 \leqslant i \leqslant r$ and $\sigma$ is an irreducible supercuspidal representation of $G_{-}(F)$. For simplicity, we denote $\Ind_{P}^{G}(\r_{cusp})$ by $\r_{1} \times \cdots \times \r_{r} \rtimes \sigma$. Note that $\sigma$ is always unitary. The discussion here can be easily extended to special orthogonal groups. 

If $G = SO(N)$ split, we can define it with respect to $J_{N}$. When $N$ is odd, the situation is exactly the same as the symplectic case. When $N = 2n$, there are two distinctions. First, the standard Levi subgroups given through the embedding \eqref{eq: embedding} do not exhaust all standard Levi subgroups of $SO(2n)$. To get all of them, we need to take the $\theta_{0}$-conjugate of $M$ given in \eqref{eq: embedding}, where 
\begin{align}
\label{eq: permutation matrix}
\theta_{0} = 
\begin{pmatrix}
1 &&&&& \\
& \ddots &&&& \\
&&& 1 && \\
&& 1 &&& \\
&&&& \ddots & \\
&&&&& 1
\end{pmatrix}.
\end{align}
Note $M^{\theta_{0}} \neq M$ only when $n_{-} = 0$ and $n_{r} > 1$. In order to distinguish the $\theta_{0}$-conjugate standard Levi subgroups of $SO(2n)$, we will only identify those Levi subgroups $M$ in \eqref{eq: embedding} with $GL(n_{1}) \times \cdots \times GL(n_{r}) \times G_{-}$, and we denote the other simply by $M^{\theta_{0}}$. Second, if the partition $n = n_{1} + \cdots + n_{r} + n_{-}$ satisfies $n_{r} =1$ and $n_{-} = 0$, then we can rewrite it as $n = n_{1} + \cdots + n_{r - 1} + n'_{-}$ with $n'_{-} = 1$, and the corresponding Levi subgroup is the same. This is because $GL(1) \cong SO(2)$. 

In this paper, we will also consider $G = SO(2n, \eta)$, which is the outer form of the split $SO(2n)$ with respect to a quadratic extension $E/F$ and conjugation by $\theta_{0}$. Here $\eta$ is the associated quadratic character of $E/F$ by the local class field theory. Then the standard Levi subgroups of $SO(2n, \eta)$ will be the outer form of those $\theta_{0}$-stable standard Levi subgroups of $SO(2n)$. In particular, they can be identified with $GL(n_{1}) \times \cdots \times GL(n_{r}) \times SO(n_{-}, \eta)$ with $n_{-} \neq 0$. 


Finally for any irreducible discrete series representation $\r$ of a symplectic group or special orthogonal group $G(F)$, our goal is to find unitary supercuspidal representations $\rho_{i}$ of $GL(d_{\rho_{i}}, F)$ for $1 \leqslant i \leqslant r$ together with real numbers $x_{1}, \cdots, x_{r}$, and a supercuspidal representation $\sigma$ of $G_{-}(F)$ which is of the same type as $G$, such that 
\[
\r \text{ or } \r^{\theta_{0}} \hookrightarrow \rho_{1}||^{x_{1}} \times \cdots \times \rho_{r}||^{x_{r}} \rtimes \sigma
\]
as a subrepresentation. 

The approach that we are going to take will highly rely on Arthur's endoscopic classification theory for symplectic and orthogonal groups \cite{Arthur:2013}, especially the structure of tempered Arthur packets (or L-packets). It is different from the original approaches of M{\oe}glin and Tadi{\`c} (see \cite{Moeglin:2002}, \cite{MoeglinTadic:2002}), where although possibly motivated by the structure of L-packets, they do not need to use it in their arguments. There are two reasons for us to adopt the new approach. One is there are certain reducibility assumptions (see Proposition~\ref{prop: cuspidal reducibility}) taken in the works of M{\oe}glin and Tadi{\`c} that could be removed under Arthur's work (as suggested by them), so it would be very natural to start with Arthur's theory at the first place. The other reason is the endoscopic theory is ``hidden" in their works, but we want to see how it could play a role in this kind of classification theory, to be more precise, the interplay of endoscopy theory with the theory of Jacquet modules (see Section~\ref{sec: parameters of supercuspidal representations}).

{\bf Acknowledgements}: This work came out of the author's talks in the automorphic forms seminar at Fields Institute in 2014. 
The author wants to thank the organizer Chung Pang Mok for initially asking him to give those talks, otherwise this paper would not appear. The author also wants to thank all the participants in the seminar for their interests and supports. While writing up the current version of this paper, the author has enjoyed the hospitality of Institute for Advanced Study (IAS). He also gratefully acknowledges the support from University of Calgary and the Pacific Institute for the Mathematical Sciences (PIMS), when this work took its final form. At last, he wants to thank the referee for many helpful comments and suggestions. During his time at IAS, he was supported by the NSF under Grant No. DMS-1128155 and DMS-1252158.


\section{Tempered Arthur packet}
\label{sec: tempered Arthur packet}

Let $F$ be a $p$-adic field and $G$ be a quasisplit symplectic group or special orthogonal group. We define the local Langlands group as $L_{F} = W_{F} \times SL(2, \C)$, where $W_{F}$ is the usual Weil group. We write $\Gal{F} = \Gal{\bar{F}/F}$ for the absolute Galois group over $F$. Let $\D{G}$ be the complex dual group of $G$, and $\L{G}$ be the Langlands dual group of $G$. A tempered (or generic) Arthur parameter of $G$ is a $\D{G}$-conjugacy class of admissible homomorphisms
\(
\underline{\p}: L_{F} \longrightarrow \L{G},
\)
such that $\underline{\p}|_{W_{F}}$ is bounded. We denote by $\Pbd{G}$ the set of tempered Arthur parameters. Here we can simplify the Langlands dual groups as in the following table:
\begin{center}
\begin{spacing}{1.5}
\begin{tabular}{| c | c | }
     \hline
      $G$                        &      $\L{G}$              \\
     \hline
      $Sp(2n)$                &      $SO(2n+1, \C)$  \\
     \hline
      $SO(2n+1)$           &      $Sp(2n, \C)$ \\
     \hline
      $SO(2n, \eta)$       &      $SO(2n, \C) \rtimes \Gal{E/F}$     \\     
     \hline
\end{tabular}
\end{spacing}
\end{center}
In the last case, $\eta$ is a quadratic character associated with a quadratic extension $E/F$ and $\Gal{E/F}$ is the associated Galois group. If we define $O(2n, \C)$ with respect to $J_{2n}$, we can fix an isomorphism $SO(2n, \C) \rtimes \Gal{E/F} \cong O(2n, \C)$ by sending the nontrivial element of $\Gal{E/F}$ to the permutation matrix \eqref{eq: permutation matrix}.
So in either of these cases, there is a natural embedding $\xi_{N}$ of $\L{G}$ into $GL(N, \C)$ up to $GL(N, \C)$-conjugacy, where $N = 2n+1$ if $G = Sp(2n)$ or $N = 2n$ otherwise. Under such an embedding, we can view the parameter $\p$ as an equivalence class of $N$-dimensional self-dual representations of $L_{F}$, i.e., $\p^{\vee} = \p$. Let $\r_{\p}$ be the self-dual representation of $GL(N, F)$ associated with $\p$ under the local Langlands correspondence (cf. \cite{HarrisTaylor:2001}, \cite{Henniart:2000}, and \cite{Scholze:2013}). If we decompose $\p$ into equivalence classes of irreducible subrepresentations, we get 
\begin{align}
\label{eq: decomposition}
\p = \bigoplus^{q}_{i = 1} l_{i}\p_{i},
\end{align}
where $\p_{i}$ is an equivalence class of irreducible representations of $L_{F}$ and $l_{i}$ is the multiplicity. Since $L_{F}$ is a product of $W_{F}$ and $SL(2, \C)$, we can further decompose $\p_{i}$ as an tensor product
\[
\p_{i} = \p_{cusp, i} \otimes \nu_{a_{i}},
\]
where $\p_{cusp, i}$ is an equivalence class of irreducible unitary representations of $W_{F}$ and $\nu_{a_{i}}$ is the $(a_{i}-1)$-th symmetric power representation of $SL(2, \C)$. Now we have obtained all the combinatorial data needed from $\p$ in the paper. To put it in a nice way, we first identify the set of equivalence classes of irreducible unitary supercuspidal representations of $GL(d, F)$ with equivalence classes of $d$-dimensional irreducible unitary representations of $W_{F}$ through the local Langlands correspondence for $GL(d)$, and denote by $\rho_{i}$ the corresponding representation for $\p_{cusp, i}$. Also notice the representation $\nu_{a_{i}}$ is completely determined by its dimension. So altogether we can simply write $\p_{i} = \rho_{i} \otimes [a_{i}]$ formally. After this discussion we can define the multi-set of Jordan blocks for $\p$ as follows,
\[
Jord(\p) := \{(\rho_{i}, a_{i}) \text{ with multiplicity } l_{i}: 1 \leqslant i \leqslant q \},
\] 
and 
\[
Jord_{\rho}(\p) := \{ a_{i} \text{ with multiplicity } l_{i} : \rho = \rho_{i}\}.
\]

To parametrize the discrete series representations, we need to introduce a subset $\Pdt{G}$ of $\Pbd{G}$. Define
\[
\Pdt{G} := \{\p \in \Pbd{G} : \p = \bigoplus^{q}_{i=1} \p_{i}, \p_{i}^{\vee} = \p_{i}\}.
\]
It is clear that the defining condition for $\Pdt{G}$ is equivalent to requiring $Jord(\p)$ is multiplicity free and $Jord_{\rho}(\p)$ is empty unless $\rho$ is self-dual. Moreover, for certain parity reason (see Section~\ref{sec: parameters of supercuspidal representations}) the integers in $Jord_{\rho}(\p)$ must be all odd or all even when $\p \in \Pdt{G}$. Besides, there is another description of $\Pdt{G}$. For $\p \in \Pbd{G}$, we fix a representative $\underline{\p}$. Let us define
\[
S_{\underline{\p}} = \Cent(\Im \underline{\p}, \D{G}),
\]
\[
\cS{\underline{\p}} = S_{\underline{\p}} / Z(\D{G})^{\Gal{F}},
\]
\[
\S{\underline{\p}} = \cS{\underline{\p}} / \cS{\underline{\p}}^{0} = S_{\underline{\p}} / S_{\underline{\p}}^{0}Z(\D{G})^{\Gal{F}}.
\]
Then we have the following fact.
\begin{lemma}
For $\p \in \Pbd{G}$, $\p \in \Pdt{G}$ if and only if $\cS{\underline{\p}}$ is finite.
\end{lemma}
This lemma can be shown by computing the group $S_{\underline{\p}}$ explicitly (see \cite[Section 1.4]{Arthur:2013}). In particular, one can show $\S{\underline{\p}}$ is abelian.

To state Arthur's classification theory of tempered representations of quasisplit symplectic and orthogonal groups, we need to introduce some more notations. We will fix an outer automorphism $\theta_{0}$ of $G$ preserving an $F$-splitting. If $G$ is symplectic or special odd orthogonal, we let $\theta_{0} = id$. If $G$ is special even orthogonal, we let $\theta_{0}$ be induced from the conjugate action of the permutation matrix \eqref{eq: permutation matrix}. Let $\D{\theta}_{0}$ be the dual automorphism of $\theta_{0}$. We write $\Sigma_{0} = <\theta_{0}>$, $G^{\Sigma_{0}} = G \rtimes \Sigma_{0}$, and $\D{G}^{\Sigma_{0}} = \D{G} \rtimes <\D{\theta}>$. Let $\x_{0}$ be the character of $G^{\Sigma_{0}}/G$, which is nontrivial when $G$ is special even orthogonal. So in the special even orthogonal case, we can send $\theta_{0}$ to the permutation matrix \eqref{eq: permutation matrix} to get an isomorphism between $G^{\Sigma_{0}}$ and the full orthogonal group. If $G$ has $\bar{F}$-rank $n$, we write $G = G(n)$. Let $G(0) = G(0)^{\Sigma_{0}} = 1$. Also for the trivial representation of $G(0)$, we require formally $1^{\theta_{0}} \ncong 1$ if $G(0) = SO(0)$, and $1^{\theta_{0}} \cong 1$ otherwise.  

Let $\Pkt{temp}(G)$ (resp. $\Pkt{2}(G)$) be the set of equivalence classes of irreducible tempered representations (resp. discrete series representations) of $G(F)$. Note $\Sigma_{0}$ acts on these sets by conjugation, and we denote the set of $\Sigma_{0}$-orbits in $\Pkt{temp}(G)$ (resp. $\Pkt{2}(G)$) by $\cPkt{temp}(G)$ (resp. $\cPkt{2}(G)$). Also note $\Sigma_{0}$ acts on $\Pbd{G}$ (resp. $\Pdt{G}$) through $\D{\theta}_{0}$, and we denote the corresponding set of $\Sigma_{0}$-orbits by $\cPbd{G}$ (resp. $\cPdt{G}$). It is clear that for $\p \in \Pbd{G}$, $Jord(\p)$ only depends on its image in $\cPbd{G}$. It is because of this reason, we will also denote the elements in $\cPbd{G}$ by $\p$. Moreover, through the natural embedding $\xi_{N}$, we can view $\cPbd{G}$ as a subset of equivalence classes of $N$-dimensional self-dual representations of $L_{F}$.

\begin{theorem}[Arthur \cite{Arthur:2013}, Theorem 1.5.1]
\label{thm: L-packet}
\begin{enumerate}
\item For $\p \in \cPbd{G}$, one can associate a finite set $\cPkt{\p}$ of $\cPkt{temp}(G)$, determined by $\r_{\p}$ through the theory of twisted endoscopy (see Section~\ref{sec: endoscopy}). And for a fixed Whittaker datum, 
there is a canonical bijection between $\cPkt{\p}$ and characters $\D{\S{\underline{\p}}}$ of $\S{\underline{\p}}$.
\[
\xymatrix{\cPkt{\p} \ar[r] & \D{\S{\underline{\p}}} \\
                 [\r] \ar@{{|}->}[r]  & <\cdot, \r>_{\underline{\p}}}
\]

\item There are decompositions
\begin{align*}
\cPkt{temp}(G) &= \bigsqcup_{\p \in \cPbd{G}} \cPkt{\p}, \\
\cPkt{2}(G) &= \bigsqcup_{\p \in \cPdt{G}} \cPkt{\p}.
\end{align*}

\end{enumerate}
\end{theorem}

We will denote the characters of $\S{\underline{\p}}$ by $\bar{\e}$, and denote the corresponding $\Sigma_{0}$-orbit $[\r]$ of irreducible representations by $\r(\p, \bar{\e})$. 
Let us define $\Pkt{\p}^{\Sigma_{0}}$ to be set of irreducible representations of $G^{\Sigma_{0}}(F)$ whose restriction to $G(F)$ belong to $\cPkt{\p}$. We call an irreducible representation $\r^{\Sigma_{0}}$ of $G^{\Sigma_{0}}(F)$ is a discrete series if its restriction to $G(F)$ are discrete series representations. We also define $S^{\Sigma_{0}}_{\underline{\p}}$, $\cS{\underline{\p}}^{\Sigma_{0}}$ and $\S{\underline{\p}}^{\Sigma_{0}}$ as before simply by taking $\D{G}^{\Sigma_{0}}$ in place of $\D{G}$. The following theorem asserts $\Pkt{\p}^{\Sigma_{0}}$ can be parametrized by the characters of $\S{\underline{\p}}^{\Sigma_{0}}$. It is a consequence of \cite[Theorem 2.2.4]{Arthur:2013}.

\begin{theorem}[Arthur]
\label{thm: L-packet full orthogonal group}
Suppose $\p \in \cPbd{G}$, for a fixed $\Sigma_{0}$-stable Whittaker datum 
there is a canonical bijection between $\Pkt{\p}^{\Sigma_{0}}$ and characters $\D{\S{\underline{\p}}^{\Sigma_{0}}}$ of $\S{\underline{\p}}^{\Sigma_{0}}$
\[
\xymatrix{\Pkt{\p}^{\Sigma_{0}} \ar[r] & \D{\S{\underline{\p}}^{\Sigma_{0}}} \\
                 \r^{\Sigma_{0}} \ar@{{|}->}[r]  & <\cdot, \r^{\Sigma_{0}}>_{\underline{\p}},}
\]
such that 
\[
<\cdot, \r^{\Sigma_{0}}>_{\underline{\p}} |_{\S{\underline{\p}}} = <\cdot, \r>_{\underline{\p}},
\]
where $\r \subseteq \r^{\Sigma_{0}}|_{G}$.

\end{theorem}

We denote the characters of $\S{\underline{\p}}^{\Sigma_{0}}$ by $\e$, and denote the corresponding representations by $\r^{\Sigma_{0}}(\p, \e)$. We also denote the image of $\e$ in $\D{\S{\p}}$ by $\bar{\e}$. If $\e_{0}$ is the generator of the kernel of the projection $\D{\S{\underline{\p}}^{\Sigma_{0}}} \rightarrow \D{\S{\underline{\p}}}$, then this theorem implies
\begin{align}
\label{eq: full orthogonal character}
\r^{\Sigma_{0}}(\p, \e \e_{0}) \cong \r^{\Sigma_{0}}(\p, \e) \otimes \x_{0}.
\end{align}
Therefore, if $G$ is special even orthogonal and $\S{\p}^{\Sigma_{0}} \neq \S{\p}$, then $\r(\p, \bar{\e})$ is a representation of $G(F)$ satisfying $\r(\p, \bar{\e})^{\theta_{0}} \cong \r(\p, \bar{\e})$. The converse is also true, i.e., if $G$ is special even orthogonal and $\S{\p}^{\Sigma_{0}} = \S{\p}$, then $\r^{\theta_{0}} \ncong \r$ for $\r(\p, \bar{\e}) = [\r]$. In the rest of this paper, we will always fix a $\Sigma_{0}$-stable Whittaker datum 
of $G$.

\section{Parameters of supercuspidal representations}
\label{sec: parameters of supercuspidal representations}

We keep the notations from the previous section.

\begin{proposition}
\label{prop: Jord of supercuspidal}
Suppose $\r$ is a supercuspidal representation of $G(F)$, and $[\r] \in \cPkt{\p}$ for some $\p \in \cPdt{G}$. Then if $(\rho, a) \in Jord(\p)$, one must have $(\rho, a-2) \in Jord(\p)$ as long as $a - 2 > 0$.
\end{proposition}

\begin{proof}
Let $\rho$ be a unitary irreducible supercuspidal representation of $GL(d_{\rho}, F)$. We can view $GL(d_{\rho}) \times G$ as the Levi component $M_{+}$ of a standard maximal parabolic subgroup $P_{+}$ of $G_{+}$, where $G$ and $G_{+}$ are of the same type. Let $\r_{M_{+}} = \rho \otimes \r$, and $w$ is the unique non-trivial element in the relative Weyl group $W(M_{+}, G_{+})$, which acts on $GL(d_{\rho})$ as an outer automorphism. Let $\r_{M_{+}, \lambda} = \rho||^{\lambda} \otimes \r$ for $\lambda \in \C$. It is a result of Arthur (see \cite[Section 2.3]{Arthur:2013}) that for certain choice of representative $\dot{w}$ of $w$, the standard intertwining operator between $\Ind^{G_{+}}_{P_{+}}(\r_{M_{+}, \lambda})$ and $\Ind^{G_{+}}_{P_{+}}(\dot{w} \, \r_{M_{+}, \lambda})$, i.e.,
\[
J_{P_{+}}(\dot{w}, \r_{M, \lambda}) h(g)= \int_{N_{P_{+}} \cap wN_{P_{+}}w^{-1}\backslash N_{P_{+}}} h(\dot{w}^{-1}ng)dn, \,\,\,\,\,\,\, \,\, h \in \Ind^{G_{+}}_{P_{+}}(\r_{M_{+}, \lambda}),
\]
and the standard intertwining operator $J_{P_{+}}(\dot{w}^{-1}, \dot{w} \, \r_{M_{+}, \lambda})$ between $\Ind^{G_{+}}_{P_{+}}(\dot{w} \, \r_{M_{+}, \lambda})$ and $\Ind^{G_{+}}_{P_{+}}(\r_{M_{+}, \lambda})$ can be normalized by meromorphic functions $r_{P_{+}}(w, \p_{M_{+}, \lambda})$ and $r_{P_{+}}(w^{-1}, w \, \p_{M_{+}, \lambda})$ respectively, i.e.,
\begin{align*}
R_{P_{+}}(\dot{w}, \r_{M_{+}, \lambda}) &:= r_{P_{+}}(w, \p_{M_{+}, \lambda})^{-1}J_{P_{+}}(\dot{w}, \r_{M_{+}, \lambda}), \\
R_{P_{+}}(\dot{w}^{-1}, \dot{w} \, \r_{M_{+}, \lambda}) &:= r_{P_{+}}(w^{-1}, w \, \p_{M_{+}, \lambda})^{-1}J_{P_{+}}(\dot{w}^{-1}, \dot{w} \, \r_{M_{+}, \lambda}),
\end{align*}
so that 
\begin{align}
\label{eq: normalized intertwining operator}
R_{P_{+}}(\dot{w}^{-1}, \dot{w} \, \r_{M_{+}, \lambda}) R_{P_{+}}(\dot{w}, \r_{M_{+}, \lambda}) = Id.
\end{align}
Here $\p_{M_{+}, \lambda}$ denotes the Langlands parameter for $\r_{M_{+}, \lambda}$, and 
\[
r_{P_{+}}(w, \p_{M_{+}, \lambda}) \sim \frac{L(\lambda, \rho \times \r_{\p})L(2\lambda, \rho, R)}{L(1+ \lambda, \rho \times \r_{\p})L(1+2\lambda, \rho, R)}
\]
where $R$ is either a symmetric square ($S^{2}$) or a skew-symmetric square ($\wedge^{2}$) representation of $GL(d_{\rho}, \C)$ and ``$\sim$" means equal up to a nowhere vanishing holomorphic function of $\lambda$ (that is given by the $\epsilon$-factors here). Note $R = \wedge^{2}$ if $G$ is $Sp(2n), SO(2n, \eta)$ or $R = S^{2}$ if $G = SO(2n+1)$. Similarly we have
\[
r_{P_{+}}(w^{-1}, w \, \p_{M_{+}, \lambda}) \sim \frac{L(-\lambda, \rho^{\vee} \times \r_{\p})L(-2\lambda, \rho^{\vee}, R)}{L(1- \lambda, \rho^{\vee} \times \r_{\p})L(1-2\lambda, \rho^{\vee}, R)}.
\]
Then we can rewrite \eqref{eq: normalized intertwining operator} as 
\begin{align}
\label{eq: normalized intertwining operator 1}
J_{P_{+}}(\dot{w}^{-1}, \dot{w} \, \r_{M_{+}, \lambda}) J_{P_{+}}(\dot{w}, \r_{M_{+}, \lambda}) \sim \frac{L(\lambda, \rho \times \r_{\p})L(2\lambda, \rho, R)L(-\lambda, \rho^{\vee} \times \r_{\p})L(-2\lambda, \rho^{\vee}, R)}{L(1+ \lambda, \rho \times \r_{\p})L(1+2\lambda, \rho, R)L(1- \lambda, \rho^{\vee} \times \r_{\p})L(1-2\lambda, \rho^{\vee}, R)}.
\end{align}
Since $\r_{M_{+}}$ is supercuspidal, it follows from a theorem of Harish-Chandra (\cite[Theorem 5.4.2.1]{Silberger:1979} and \cite[Lemma 2.2.5]{Shahidi:1981}) that both $J_{P_{+}}(\dot{w}, \r_{M_{+}, \lambda})$ and $J_{P_{+}}(\dot{w}^{-1}, \dot{w} \, \r_{M_{+}, \lambda})$ are holomorphic for $\text{Re } \lambda \neq 0$. So now we will assume $\lambda \in \mathbb{R}$ and $\lambda > 1/2$. Since $L(s, \rho \times \r_{\p})$ does not have a pole (non-vanishing is clear) for real $s > 0$, and $L(s, \rho, R)$ does not have a pole (non-vanishing is clear) for real $s \neq 0$, we have
\begin{align}
\label{eq: normalized intertwining operator 2}
J_{P_{+}}(\dot{w}^{-1}, \dot{w} \, \r_{M_{+}, \lambda}) J_{P_{+}}(\dot{w}, \r_{M_{+}, \lambda}) \sim \frac{L(-\lambda, \rho^{\vee} \times \r_{\p})}{L(1- \lambda, \rho^{\vee} \times \r_{\p})} \,\,\,\, \,\,\,\,\,(\lambda > 1/2).
\end{align}
Finally, we learn from the definition of $L(s, \rho^{\vee} \times \r_{\p})$ that it has a pole at $s = -(a-1)/2$ if and only if $\rho \cong \rho^{\vee}$ and $(\rho, a) \in Jord(\p)$ (see Appendix~\ref{sec: local L-function}). So if we know $(\rho, a) \in Jord(\p)$ for $a > 2$, then $L(-\lambda, \rho^{\vee} \times \r_{\p})$ has a pole at $\lambda = (a-1)/2 > 1/2$. By the holomorphy of standard intertwining operators on the left hand side of \eqref{eq: normalized intertwining operator 2}, $L(1- \lambda, \rho^{\vee} \times \r_{\p})$ must also have a pole at $\lambda = (a-1)/2$, i.e., $1- \lambda = 1- (a-1)/2 = -(a-3)/2$ is a pole of $L(s, \rho^{\vee} \times \r_{\p})$. So $(\rho, a-2) \in Jord(\p)$.

\end{proof}

If $\rho$ is a self-dual unitary irreducible supercuspidal representation of $GL(d_{\rho}, F)$, we know from Appendix~\ref{sec: local L-function} that
\[
L(s, \rho \times \rho) = L(s, \rho, \wedge^{2})L(s, \rho, S^{2})
\]
has a pole at $s= 0$. We call $\rho$ is of {\bf symplectic type} if $L(s, \rho, \wedge^{2})$ has a pole at $s=0$, and we call $\rho$ is of {\bf orthogonal type} if $L(s, \rho, S^{2})$ has a pole at $s=0$. Moreover, for any positive integer $a$, the pair $(\rho, a)$ is called having orthogonal type if $\rho$ is of orthogonal type and $a$ is odd, or $\rho$ is of symplectic type and $a$ is even. Otherwise $(\rho, a)$ is called having symplectic type. Next we are going to prove a very important reducibility result, which is named ``Basic Assumption" in \cite{Moeglin:2002}, \cite{MoeglinTadic:2002}. Those careful readers may notice there is a slight difference between our statement below and the original one. The reason is they consider the group $G^{\Sigma_{0}}$ rather than $G$, nonetheless one can translate between these two statements without difficulty (see Corollary~\ref{cor: cuspidal reducibility}).

\begin{proposition}
\label{prop: cuspidal reducibility}
Suppose $\r$ is a supercuspidal representation of $G(F)$, and $[\r] \in \cPkt{\p}$ for some $\p \in \cPdt{G}$. Then for any unitary irreducible supercuspidal representation $\rho$ of $GL(d_{\rho}, F)$ and real number $a_{\rho}$, the parabolic induction 
\[
\rho||^{\pm(a_{\rho} +1)/2} \rtimes \r
\]
reduces if and only if $\rho$ is self-dual and 
\begin{align}
\label{eq: cuspidal reducibility}
a_{\rho} = \begin{cases}
                 \text{ max } Jord_{\rho}(\p), & \text{ if } Jord_{\rho}(\p) \neq \emptyset, \\
                                    0, & \text{ if $Jord_{\rho}(\p) = \emptyset$, $\rho$ is of opposite type to $\D{G}$}, \\
                                    -1, & \text{ otherwise, provided $d_{\rho}$ is even or $\r \cong \r^{\theta_{0}}$. }
                 \end{cases}
\end{align}
\end{proposition}

\begin{proof}
We will follow the proof of the previous proposition and let $\lambda \in \mathbb{R}$. Suppose $\lambda > 0$. Then the image of $J_{P_{+}}(\dot{w}, \r_{M_{+}, \lambda})$ is nonzero and irreducible by \cite[Proposition 2.6]{BorelWallach:2000}. It follows $\rho||^{\lambda} \rtimes \r$ is irreducible if and only if the kernel of $J_{P_{+}}(\dot{w}, \r_{M_{+}, \lambda})$ is trivial. Since $\Ind^{G_{+}}_{P_{+}}(\dot{w} \, \r_{M_{+}, \lambda})$ and $\Ind^{G_{+}}_{P_{+}}(\r_{M_{+}, \lambda})$ have the same irreducible constituents (see \cite[Theorem 2.9]{BernsteinZelevinsky:1977}), the kernel of $J_{P_{+}}(\dot{w}, \r_{M_{+}, \lambda})$ is trivial if and only if $J_{P_{+}}(\dot{w}, \r_{M_{+}, \lambda})$ is an isomorphism. As we have seen previously, $J_{P_{+}}(\dot{w}, \r_{M_{+}, \lambda})$ and $J_{P_{+}}(\dot{w}^{-1}, \dot{w} \, \r_{M_{+}, \lambda})$ are holomorphic. In fact, they are also nonzero (see \cite[Section 4.1]{Waldspurger:2003}). So $J_{P_{+}}(\dot{w}, \r_{M_{+}, \lambda})$ is an isomorphism if and only if
\[
J_{P_{+}}(\dot{w}^{-1}, \dot{w} \, \r_{M_{+}, \lambda}) J_{P_{+}}(\dot{w}, \r_{M_{+}, \lambda}) \neq 0.
\]
As a consequence, $\rho||^{\lambda} \rtimes \r$ is reducible if and only if
\[
J_{P_{+}}(\dot{w}^{-1}, \dot{w} \, \r_{M_{+}, \lambda}) J_{P_{+}}(\dot{w}, \r_{M_{+}, \lambda}) = 0.
\]
Let us first assume $\lambda > 1/2$, then from \eqref{eq: normalized intertwining operator 2} it is enough to see when
\begin{align}
\label{eq: reducibility condition}
\frac{L(-\lambda, \rho^{\vee} \times \r_{\p})}{L(1- \lambda, \rho^{\vee} \times \r_{\p})} = 0,
\end{align}
i.e., $L(1- \lambda, \rho^{\vee} \times \r_{\p})$ has a pole, but $L(-\lambda, \rho^{\vee} \times \r_{\p})$ does not. From our discussion in the previous proof we know this can only happen when $\rho = \rho^{\vee}$ and $\lambda = (a_{\rho} + 1)/2$, where $a_{\rho}$ is max $Jord_{\rho}(\p)$. 
Next we assume $0< \lambda \leqslant 1/2$, it follows from \eqref{eq: normalized intertwining operator 1} that 
\[
J_{P_{+}}(\dot{w}^{-1}, \dot{w} \, \r_{M_{+}, \lambda}) J_{P_{+}}(\dot{w}, \r_{M_{+}, \lambda}) \sim \frac{L( - \lambda, \rho^{\vee} \times \r_{\p})}{L(1-2\lambda, \rho^{\vee}, R)}.
\]
And the right hand side can be zero only when $L(1-2\lambda, \rho^{\vee}, R)$ has a pole, but $L( - \lambda, \rho^{\vee} \times \r_{\p})$ does not. So necessarily $\rho = \rho^{\vee}$ and $\lambda = 1/2$. By our assumption on the representation $R$, we know $L(s, \rho, R)$ has a pole at $s = 0$ if and only if $\rho$ is of opposite type to $\D{G}$. And the requirement that $L(s, \rho \times \r_{\p})$ does not have a pole at $-1/2$ implies $Jord_{\rho}(\p) = \emptyset$. 

For $\lambda < 0$, one just needs to notice 
\(
s.s.(\rho||^{s} \rtimes \r) = s.s.(\rho^{\vee}||^{-s} \rtimes \r)^{\theta}
\) 
for some $\theta \in \Sigma_{0}$, so one can apply the same argument to $\rho^{\vee}||^{-\lambda} \rtimes \r$. 


Finally, we consider $\lambda = 0$, where our previous criterion does not work. However the reducibility of $\rho \rtimes \r$ follows from the standard theory of representation theoretic $R$-groups. In Arthur's theory these groups have been shown to be isomorphic to $R$-groups defined by parameters, which can be computed explicitly (see \cite[Section 2.4 and Section 6.6]{Arthur:2013}). So our reducibility condition in this case will follow from there.

\end{proof}

Suppose $\r$ is an irreducible supercuspidal representation of $G(F)$ and $[\r] \in \cPkt{\p}$ for some $\p \in \cPdt{G}$. We know from Proposition~\ref{prop: Jord of supercuspidal} that $Jord(\p)$ should be in a certain shape, and in view of  Theorem~\ref{thm: L-packet} we would also like to know what kind of character $\bar{\e}$ of $\S{\underline{\p}}$ will parametrize $[\r]$. To give a description of such characters, we have to first make an identification between $\D{\S{\underline{\p}}^{\Sigma_{0}}}$ with $\Two$-valued functions over $Jord(\p)$. To be more precise, let us assume 
\begin{align}
\label{eq: discrete parameter}
\underline{\p} = \underline{\p}_{1} \+ \cdots \+ \underline{\p}_{r}
\end{align}
where $\underline{\p}_{i}$ are self-dual irreducible representations of dimension $n_{i}$. By Schur's Lemma, 
\[
\Cent(\underline{\p}, GL(N, \C)) \cong \underbrace{\C^{\times} \times \cdots \times \C^{\times}}_{r}
\] 
where each $\C^{\times}$ acts on the corresponding representation space of $\underline{\p}_{i}$. So 
\[
\Cent(\underline{\p}, \D{G}) \cong \{ s  = (s_{i}) \in \Two^{r} : \prod_{i}(s_{i})^{n_{i}} = 1\}.
\]
Note $\S{\underline{\p}}^{\Sigma_{0}} = \cS{\underline{\p}}^{\Sigma_{0}}$ in this case. Then 
\(
\S{\underline{\p}}^{\Sigma_{0}} \cong  \Two^{r} / <(-1, \cdots, -1)>. 
\)
Sine the right hand side does not depend on the choice of representative $\underline{\p}$, we can denote it by $\S{\p}^{\Sigma_{0}}$. If $G$ is special even orthogonal, 
\[
\S{\underline{\p}} \cong \{ s  = (s_{i}) \in \Two^{r} : \prod_{i}(s_{i})^{n_{i}} = 1\} / <(-1, \cdots, -1)>
\] 
which is a subgroup of $\S{\underline{\p}}^{\Sigma_{0}}$ of index $1$ or $2$. Similarly, we denote the right hand side by $\S{\p}$. 

Let us define the characters of $\Two^{r} / <(-1, \cdots, -1)>$ to be $\Two$-valued functions $\e = (\e_{i}) \in \Two^{r}$ such that $\prod_{i}\e_{i} = 1$. Moreover, for $s \in \Two^{r} / <(-1, \cdots, -1)>$, we define $\e(s) = \prod_{i} (\e_{i} \ast s_{i})$, where 
\[
\e_{i} \ast s_{i} = \begin{cases}
                           -1, & \text{ if } \e_{i} = s_{i} = -1 \\
                           1,  & \text{ otherwise. } 
                           \end{cases}
\]
So 
\[
\D{\S{\p}^{\Sigma_{0}}} = \{\e = (\e_{i}) \in \Two^{r} : \prod_{i}\e_{i} = 1\}.
\] 
In particular, when $G$ is special even orthogonal, we define $\e_{0} = (\e_{0,i}) \in \D{\S{\p}^{\Sigma_{0}}}$ satisfying $\e_{0,i} = 1$ if $n_{i}$ is even, and $\e_{0,i} = -1$ if $n_{i}$ is odd, then 
\[
\D{\S{\p}} = \{\e = (\e_{i}) \in \Two^{r} : \prod_{i}\e_{i} = 1\} / <\e_{0}>.
\]
In general, let $\e_{0} = 1$ if $G$ is not special even orthogonal.

Now we can formulate the theorem for parametrizing supercuspidal representations inside tempered Arthur packets.

\begin{theorem}[Moeglin \cite{Moeglin1:2011}, Theorem 1.5.1]
\label{thm: supercuspidal parametrization}
The $\Sigma_{0}$-orbits of irreducible supercuspidal representations of $G(F)$ can be parametrized by $\p \in \cPdt{G}$ and $\bar{\e} \in \D{\S{\p}}$ satisfying the following properties:

\begin{enumerate}

\item if $(\rho, a) \in Jord(\p)$, then $(\rho, a-2) \in Jord(\p)$ as long as $a - 2 > 0$;

\item if $(\rho, a), (\rho, a-2) \in Jord(\p)$, then $\e(\rho, a) \e(\rho, a - 2) = -1$;

\item if $(\rho, 2) \in Jord(\p)$, then $\e(\rho, 2) = -1$.

\end{enumerate}

\end{theorem}

The proof that we are going to give makes use of the (twisted) endoscopic character identities and explicit computation of Jacquet modules. So we will first review these two subjects in the next two sections.

\section{Endoscopy}
\label{sec: endoscopy}

The endoscopy theory can be stated for any connected reductive groups, but here we will mainly consider the case when $G$ is a quasisplit symplectic group or special orthogonal group. 

Suppose $\p \in \Pdt{G}$ and $s \in \cS{\underline{\p}} = \S{\underline{\p}}$. In our case, there is a quasisplit reductive group $H$ with the property that 
\[
\D{H} \cong \Cent(s, \D{G})^{0},
\] and the isomorphism extends to an embedding 
\[
\xi: \L{H} \rightarrow \L{G}
\] 
such that $\xi(\L{H}) \subseteq \Cent(s, \L{G})$ and $\underline{\p}$ factors through $\L{H}$. Hence we get a parameter $\p_{H} \in \Pbd{H}$. In fact it is easy to show $\cS{\underline{\p}_{H}}$ is also finite, so $\p_{H} \in \Pdt{H}$. We say $(H, \underline{\p}_{H})$ corresponds to $(\underline{\p}, s)$ through $\xi$, and denote this relation by $(H, \underline{\p}_{H}) \rightarrow (\underline{\p}, s)$. Such $H$ is called an {\bf elliptic endoscopic group} of $G$. 

\begin{example}
\label{eg: endoscopy}
\begin{enumerate}

\item If $G = Sp(2n)$, then $\L{G} = SO(2n + 1, \C)$. 
For $\p \in \cPdt{G}$, let us write 
\[
\p = \p_{1} \+ \cdots \+ \p_{r}
\] 
as in \eqref{eq: decomposition}. Then $\S{\p} = \Two^{r} / <-1, \cdots, -1>$, and for any $s = (s_{i}) \in \S{\p}$, it gives a partition on $Jord(\p)$, i.e., 
\[
\p = (\+_{s_{i} = 1}\p_{i}) \+ (\+_{s_{j} = -1}\p_{j}).
\]
Without loss of generality, let us assume 
\[
\sum_{s_{i} = 1} n_{i} = 2n_{I} + 1 = N_{I} \text{ and } \sum_{s_{j} = -1} n_{j} = 2n_{II} = N_{II}.
\] 
Define 
\[
\eta_{I} = \eta_{II} = \prod_{s_{j} = -1}\eta_{j},
\] 
where $\eta_{j}$ is a quadratic character given by the central character of $\r_{\p_{j}}$. Let 
\[
G_{I} = Sp(2n_{I}) \text{ and } G_{II} = SO(2n_{II}, \eta_{II}).
\] 
Then we have
\[
H = G_{I} \times G_{II} \text{ and } \L{H} = (\D{G}_{I} \times \D{G}_{II}) \rtimes \Gal{E_{II}/F},
\] 
where $E_{II}$ is a quadratic extension of $F$ associated with $\eta_{II}$. Let
\[
\xi_{i}: \L{G}_{i} \hookrightarrow GL(N_{i}, \C)
\]
be the natural embedding for $i = I, II$. Then 
\[
\xi := (\xi_{I} \otimes \eta_{I}) \+ \xi_{II}
\] 
factors through $\L{G}$ and defines an embedding $\L{H} \hookrightarrow \L{G}$. Let 
\[
\p_{I} := (\+_{s_{i} = 1}\p_{i}) \otimes \eta_{I} \in \cPdt{G_{I}}
\]
and
\[
\p_{II} := \+_{s_{j} = -1}\p_{j} \in \cPdt{G_{II}}.
\] 
Then 
\[
\p_{H} = \p_{I} \times \p_{II}.
\]

\item If $G = SO(2n+1)$, then $\L{G} = Sp(2n, \C)$. For $\p \in \cPdt{G}$, let us write 
\[
\p = \p_{1} \+ \cdots \+ \p_{r}
\] 
as in \eqref{eq: decomposition}. Then $\S{\p} = \Two^{r} / <-1, \cdots, -1>$, and for any $s = (s_{i}) \in \S{\p}$, it gives a partition on $Jord(\p)$, i.e., 
\[
\p = (\+_{s_{i} = 1}\p_{i}) \+ (\+_{s_{j} = -1}\p_{j}).
\]
We can assume 
\[
\sum_{s_{i} = 1} n_{i} = 2n_{I} = N_{I} \text{ and } \sum_{s_{j} = -1} n_{j} = 2n_{II} = N_{II}.
\] 
Let 
\[
G_{I} = SO(2n_{I} + 1) \text{ and } G_{II} = SO(2n_{II} + 1).
\] 
Then we have 
\[
H = G_{I} \times G_{II} \text{ and } \L{H} = \D{G}_{I} \times \D{G}_{II}
\]
Let
\[
\xi_{i}: \L{G}_{i} \hookrightarrow GL(N_{i}, \C)
\]
be the natural embedding for $i = I, II$. Then 
\[
\xi := \xi_{I} \+ \xi_{II}
\] 
factors through $\L{G}$ and defines an embedding $\L{H} \hookrightarrow \L{G}$. Let 
\[
\p_{I} := (\+_{s_{i} = 1}\p_{i}) \in \cPdt{G_{I}}
\]
and
\[
\p_{II} := \+_{s_{j} = -1}\p_{j} \in \cPdt{G_{II}}.
\] 
Then 
\[
\p_{H} = \p_{I} \times \p_{II}.
\]

\item If $G = SO(2n, \eta)$, then $\L{G} = SO(2n, \C) \rtimes \Gal{E/F}$. For $\p \in \cPdt{G}$, let us write 
\[
\p = \p_{1} \+ \cdots \+ \p_{r}
\]
as in \eqref{eq: decomposition}. Then $\S{\p}^{\Sigma_{0}} = \Two^{r} / <-1, \cdots, -1>$, and for any $s = (s_{i}) \in \S{\p} \subseteq \S{\p}^{\Sigma_{0}}$, it gives a partition on $Jord(\p)$, i.e., 
\[
\p = (\+_{s_{i} = 1}\p_{i}) \+ (\+_{s_{j} = -1}\p_{j}).
\]
By our description of $\S{\p}$, we can assume 
\[
\sum_{s_{i} = 1} n_{i} = 2n_{I} = N_{I} \text{ and } \sum_{s_{j} = -1} n_{j} = 2n_{II} = N_{II}.
\] 
Define 
\[
\eta_{I} = \eta_{II} \eta \text{ and } \eta_{II} = \prod_{s_{j} = -1}\eta_{j},
\] 
where $\eta_{j}$ is a quadratic character given by the central character of $\r_{\p_{j}}$. We also denote by $E_{i}$ the quadratic extension of $F$ associated with $\eta_{i}$ for $i = I, II$.
Let 
\[
G_{I} = SO(2n_{I}, \eta_{I}) \text{ and } G_{II} = SO(2n_{II}, \eta_{II}),
\]
Then we have 
\[
H = G_{I} \times G_{II} \text{ and } \L{H}  = (\D{G}_{I} \times \D{G}_{II}) \rtimes \Gal{L/F} 
\]
where $L = E_{I} E_{II}$. Let
\[
\xi_{i}: \L{G}_{i} \hookrightarrow GL(N_{i}, \C)
\]
be the natural embedding for $i = I, II$. Then 
\[
\xi := \xi_{I} \+ \xi_{II}
\] 
factors through $\L{G}$ and defines an embedding $\L{H} \hookrightarrow \L{G}$. Let 
\[
\p_{I} := \+_{s_{i} = 1}\p_{i} \in \cPdt{G_{I}}
\] 
and 
\[
\p_{II} := \+_{s_{j} = -1}\p_{j} \in \cPdt{G_{II}}
\]
Then
\[
\p_{H} = \p_{I} \times \p_{II}.
\]

\end{enumerate}
\end{example}

In the examples above, we can define $\cPdt{H} = \cPdt{G_{I}} \times \cPdt{G_{II}}$ (resp. $\cPbd{H} = \cPbd{G_{I}} \times \cPbd{G_{II}}$), then $\p_{H} \in \cPdt{H}$. For $s \in \S{\p}$, we still say $(H, \p_{H})$ corresponds to $(\p, s)$, and denote this relation again by $(H, \p_{H}) \rightarrow (\p, s)$.

In part (3), it is possible to also choose $s \in \S{\p}^{\Sigma_{0}}$ but not in $\S{\p}$, and then we get a partition on $Jord(\p)$, i.e., 
\[
\p = (\+_{s_{i} = 1}\p_{i}) \+ (\+_{s_{j} = -1}\p_{j}),
\]
so that 
\[
\sum_{s_{i} = 1} n_{i} = 2n_{I} + 1 \text{ and } \sum_{s_{j} = -1} n_{j} = 2n_{II} + 1.
\] 
Define 
\[
\eta_{I} = \eta_{II} \eta \text{ and } \eta_{II} = \prod_{s_{j} = -1}\eta_{j},
\] 
where $\eta_{j}$ is a quadratic character given by the central character of $\r_{\p_{j}}$.
Let 
\[
G_{I} = Sp(2n_{I}) \text{ and } G_{II} = Sp(2n_{II})
\]
Then 
\[
\p_{I} := (\+_{s_{i} = 1}\p_{i}) \otimes \eta_{I} \in \cPdt{G_{I}}
\] 
and 
\[
\p_{II} := (\+_{s_{j} = -1}\p_{j}) \otimes \eta_{II} \in \cPdt{G_{II}}.
\] 
We can take 
\[
H = G_{I} \times G_{II} \text{ and } \L{H} = \D{G}_{I} \times \D{G}_{II}.
\]
In this case, $H$ is called a {\bf twisted elliptic endoscopic group} of $G$. Let 
\[
\xi_{i}: \L{G}_{i} \hookrightarrow GL(N_{i}, \C)
\]
be the natural embedding for $i = I, II$. Then 
\[
\xi := (\xi_{I} \otimes \eta_{I}) \+ (\xi_{II} \otimes \eta_{II})
\] 
factors through $\L{G}$ and defines an embedding $\L{H} \hookrightarrow \L{G}$. Let 
\[
\p_{H} = \p_{I} \times \p_{II}.
\] 
We say $(H, \p_{H})$ corresponds to $(\p, s)$ through $\xi$, and write $(H, \p_{H}) \rightarrow (\p, s)$.

In this paper, we also want to consider the twisted elliptic endoscopic groups of $GL(N)$, but we will only need the simplest case here. Recall for $\p \in \Pbd{G}$, we can view $\underline{\p}$ as a self-dual $N$-dimensional representation through the natural embedding 
\[
\xi_{N}: \L{G} \rightarrow GL(N, \C),
\] 
and in this way we get a self-dual parameter for $GL(N)$. We fix an outer automorphism $\theta_{N}$ of $GL(N)$ preserving an $F$-splitting, and let $\D{\theta}_{N}$ be the dual automorphism on $GL(N, \C)$, then 
\[
\xi_{N}(\L{G}) \subseteq \Cent(s, GL(N, \C)) \text{ and } \D{G} = \Cent(s, GL(N, \C))^{0}
\] 
for some $s \in GL(N, \C) \rtimes \D{\theta}_{N}$. So we call $G$ a twisted elliptic endoscopic group of $GL(N)$.

What lies in the heart of the endoscopy theory is a transfer map on the spaces of smooth compactly supported functions from $G$ to its (twisted) elliptic endoscopic group $H$ (similarly from $GL(N)$ to its twisted elliptic endoscopic group $G$). The existence of the transfer map is quite deep, and it was conjectured by Langlands, Shelstad and Kottwitz. In a series of papers Waldspurger \cite{Waldspurger:1995} \cite{Waldspurger:1997} \cite{Waldspurger:2006} \cite{Waldspurger:2008} is able to reduce it to the {\bf Fundamental Lemma} for Lie algebras over function fields. Finally it is in this particular form of the fundamental lemma, Ngo \cite{Ngo:2010} gave his celebrated proof. Let us denote such transfers by 

\begin{align}
\label{eq: endoscopic transfer}
\xymatrix{C^{\infty}_{c}(G(F)) \ar[r]  & C^{\infty}_{c}(H(F)) \\
                f \ar[r] & f^{H}}
\end{align}
and similarly 
\begin{align}
\label{eq: twisted endoscopic transfer}
\xymatrix{C^{\infty}_{c}(GL(N, F)) \ar[r]  & C^{\infty}_{c}(G(F)) \\
                f \ar[r] & f^{G}}
\end{align}
We should point out these transfer maps are only well defined after we pass to the space of (twisted) {\bf orbital integrals} on the source and the space of {\bf stable orbital integrals} on the target. Note the space of (twisted) (resp. stable) orbital integrals are dual to the space of (twisted) (resp. stable) invariant distributions on $G(F)$, i.e. one can view the (twisted) (resp. stable) invariant distributions of $G(F)$ as linear functionals of the space of (twisted) (resp. stable) orbital integrals. So dual to these transfer maps, the stable invariant distributions on $H(F)$ (resp. $G(F)$) will map to the (twisted) invariant distributions on $G(F)$ (resp. $GL(N, F)$). We call this map the (twisted) {\bf spectral endoscopic transfer}.

If $\r$ is an irreducible smooth representation of $G(F)$, then it defines an invariant distribution on $G(F)$ by the trace of 
\[
\r(f) = \int_{G(F)}f(g)\r(g)dg
\] 
for $f \in C^{\infty}_{c}(G(F))$. We call this the character of $\r$ and denote it by $f_{G}(\r)$. For any irreducible representation $\r^{\Sigma_{0}}$ of $G^{\Sigma_{0}}(F)$, which contains $\r$ in its restriction to $G(F)$, we define a twisted invariant distribution on $G(F)$ by the trace of 
\[
\r^{\Sigma_{0}}(f) = \int_{G(F) \rtimes \theta_{0}} f(g)\r^{\Sigma_{0}}(g)dg
\] 
for $f \in C^{\infty}_{c}(G(F) \rtimes \theta_{0})$. We call this the twisted character of $\r$, and denote it by $f_{G}(\r^{\Sigma_{0}})$. We can also define the twisted characters for $GL(N)$ similarly, but we will write it in a slightly different way. Let $\r$ be a self-dual irreducible smooth representation of $GL(N, F)$, we can define a twisted invariant distribution on $GL(N, F)$ by taking the trace of 
\[
\r(f) \circ A_{\r}(\theta_{N})
\] 
for $f \in C^{\infty}_{c}(GL(N, F))$, where $A_{\r}(\theta_{N})$ is an intertwining operator between $\r$ and $\r^{\theta_{N}}$. We call this the twisted character of $\r$ and denote it by $f_{N^{\theta}}(\r)$. 

Since the (twisted) elliptic endoscopic groups $H$ in our case are all products of quasisplit symplectic and special orthogonal groups, we can define a group of automorphisms of $H$ by taking the product of $\Sigma_{0}$ on each factor, and we denote this group again by $\Sigma_{0}$. Let $\sH(G)$ (resp. $\sH(H)$) be the subspace of $\Sigma_{0}$-invariant functions in $C_{c}^{\infty}(G(F))$ (resp. $C_{c}^{\infty}(H(F))$). Then it follows from a simple property of the transfer map (which we will not explain here) that we can restrict both \eqref{eq: endoscopic transfer} and \eqref{eq: twisted endoscopic transfer} to $\sH(G)$ and $\sH(H)$. Now we are ready to state a more precise version of Theorem~\ref{thm: L-packet}. 

\begin{theorem}[Arthur]
\label{thm: character relation}

\begin{enumerate}

\item Suppose $\p \in \cPdt{G}$, the sum of characters in $\cPkt{\p}$ 
\[
f(\p) = \sum_{[\r] \in \cPkt{\p}}f_{G}(\r) 
\]
defines a stable invariant distribution for $f \in \sH(G)$. Moreover it is {\bf uniquely} determined by $\r_{\p}$ through 
\begin{align}
\label{eq: character relation GL(N)}
f^{G}(\p) = f_{N^{\theta}}(\r_{\p}), \,\,\,\,\,\,\,\,\,\,\,\,\,\,   f \in C^{\infty}_{c}(GL(N))
\end{align}
after we normalize the Haar measures on $G(F)$ and $GL(N, F)$ in a compatible way.

\item Suppose $\p \in \cPdt{G}$, and $(H, \p_{H}) \rightarrow (\p, s)$ for $s \in \S{\p}$. If we define a stable invariant distribution $f(\p_{H})$ for $\sH(H)$ as in (1), then after we normalize the Haar measures on $G(F)$ and $H(F)$ in a compatible way the following identity holds

\begin{align}
\label{eq: character relation}
f^{H}(\p_{H}) = \sum_{[\r] \in \cPkt{\p}} <s, \r>f_{G}(\r) \,\,\,\,\,\,\,\,\,\,\, f \in \sH(G)
\end{align}
where
\[
<\cdot, \r> := <\cdot, \r>_{\underline{\p}}
\]
under the isomorphism $\S{\p} \cong \S{\underline{\p}}$.
\end{enumerate}

\end{theorem}

\begin{remark}
\label{rk: character relation}
Although we only state the theorem for discrete parameters, these statements are also true for tempered parameters (once we extend the definition $(H, \p_{H}) \rightarrow (\p, s)$ appropriately). 
The two identities \eqref{eq: character relation GL(N)} and \eqref{eq: character relation} are the ones we call (twisted) endoscopic character identities in the end of Section~\ref{sec: parameters of supercuspidal representations}, and they are also often referred to as (twisted) character relations. There are some ambiguities that we need to clarify in such identities. On one hand, in the definition of $f_{N^{\theta}}(\r_{\p})$ we need to choose a normalization of the intertwining operator $A_{\r_{\p}}(\theta_{N})$. In this theorem, we require $A_{\r_{\p}}(\theta_{N})$ to fix some Whittaker functional for $\r_{\p}$. On the other hand, in the definition of the transfer maps there is also a normalization issue. To resolve that, we need to fix certain (resp. $\theta_{N}$-stable) Whittaker datum for $G$ (resp. $GL(N)$), and we will take the so-called Whittaker normalization on the transfer maps. Finally, the stable invariant distribution $f(\p)$ for $f \in \sH(G)$ is uniquely determined by $\r_{\p}$ for the transfer map \eqref{eq: twisted endoscopic transfer} is surjective onto the space of $\Sigma_{0}$-invariant stable orbital integrals of $G(F)$.
\end{remark}

When $G$ is special even orthogonal, we have an additional character identity. To state it, we need to identify $C^{\infty}_{c}(G(F) \rtimes \theta_{0})$ with $C^{\infty}_{c}(G(F))$ by sending $g \rtimes \theta_{0}$ to $g$, so the twisted transfer map on $C^{\infty}_{c}(G(F))$ can also be translated to $C^{\infty}_{c}(G(F) \rtimes \theta_{0})$.

\begin{theorem}[Arthur]
\label{thm: character relation full orthogonal group}
Suppose $\p \in \cPdt{G}$, and $(H, \p_{H}) \rightarrow (\p, s)$ for $s \in \S{\p}^{\Sigma_{0}}$ but not in $\S{\p}$. Then after we normalize the Haar measures on $G(F)$ and $H(F)$ in a compatible way the following identity holds

\begin{align}
\label{eq: character relation full orthogonal group}
f^{H}(\p_{H}) = \sum_{[\r] \in \cPkt{\p}} <s, \r^{\Sigma_{0}}>f_{G}(\r^{\Sigma_{0}}), \,\,\,\,\,\,\,\,\,\,\, f \in C^{\infty}_{c}(G(F) \rtimes \theta_{0})
\end{align}
where $\r^{\Sigma_{0}}|_{G} = \r$ and
\[
<\cdot, \r^{\Sigma_{0}}> := <\cdot, \r^{\Sigma_{0}}>_{\underline{\p}}
\]
under the isomorphism $\S{\p}^{\Sigma_{0}} \cong \S{\underline{\p}}^{\Sigma_{0}}$.
\end{theorem}

Again this theorem also holds for $\p \in \cP{G}$ (once we extend the definition $(H, \p_{H}) \rightarrow (\p, s)$ appropriately), and we have taken the Whittaker normalization on the transfer maps with respect to the fixed $\Sigma_{0}$-stable Whittaker datum in Theorem~\ref{thm: L-packet full orthogonal group}. We will only need this theorem in Section~\ref{sec: even orthogonal group}.

\section{Jacquet modules}
\label{sec: Jacquet modules}

First let us assume $G$ is any connected reductive group over $F$, and let $\Rep(G)$ be the category of finite-length smooth representations of $G$. If $M$ is the Levi component of a parabolic subgroup $P$ of $G$, then the normalized parabolic induction defines a functor from $\Rep(M)$ to $\Rep(G)$. The normalized Jacquet module is its left adjoint functor, i.e.,
\begin{align}
\label{eq: Frobenius reciprocity}
\Hom_{M}(\Jac_{P}\r, \sigma) \cong \Hom_{G}(\r, \Ind^{G}_{P}\sigma),
\end{align}
for $\r \in \Rep(G)$ and $\sigma \in \Rep(M)$. This relation \eqref{eq: Frobenius reciprocity} is usually referred to as {\bf Frobenius reciprocity}. One can see easily from \eqref{eq: Frobenius reciprocity} and Theorem~\ref{thm: B-Z} that $\r \in \Rep(G)$ is supercuspidal if and only if $\Jac_{P}\r = 0$ for all standard parabolic subgroups $P$ of $G$. In fact this is one of the equivalent definitions of supercuspidal representations. The next two lemmas state some general facts about Jacquet modules, and we refer the interested readers to \cite[Section 3]{MoeglinTadic:2002} for their proofs.

\begin{lemma}
\label{lemma: A}
Suppose $\r \in \Rep(G)$ is irreducible, and $\sigma$ is an irreducible supercuspidal constituent of $\Jac_{P}\r$, then there is an inclusion 
\[
\r \hookrightarrow \Ind_{P}^{\, G}\sigma.
\]
\end{lemma}

\begin{lemma}
\label{lemma: B}
Suppose $\r \in \Rep(G)$ is irreducible, and $M = M_{1} \times M_{2}$. Let $\tau_{1} \in \Rep(M_{1})$ be irreducible and $\tau_{2} \in \Rep(M_{2})$. If 
\[
\r \hookrightarrow \Ind_{P}^{\, G}(\tau_{1} \otimes \tau_{2}), 
\]
then there exists an irreducible constituent $\tau'_{2}$ in $\tau_{2}$ such that 
\[
\r \hookrightarrow \Ind_{P}^{\, G}(\tau_{1} \otimes \tau'_{2}).
\]
\end{lemma}

Now let us restrict to the case when $G$ is a quasisplit symplectic or special orthogonal group. We would like to define a modified Jacquet functor. For this we first fix a unitary irreducible supercuspidal representation $\rho$ of $GL(d_{\rho}, F)$, and we assume $M = GL(d_{\rho}) \times G_{-}$ is the Levi component of a standard maximal parabolic subgroup $P$ of $G$. In case $G_{-} = 1$ and $G$ is special even orthogonal, we require $P$ to be contained in the standard parabolic subgroup of $GL(2n)$. Then for $\r \in \Rep(G)$, 
\[
s.s. \Jac_{P}(\r) = \bigoplus_{i} \tau_{i} \otimes \sigma_{i},
\] 
where $\tau_{i} \in \Rep(GL(d_{\rho}))$ and $\sigma_{i} \in \Rep(G_{-})$, both of which are irreducible. We define $\Jac_{x}\r$ for any real number $x$ to be 
\[
\Jac_{x}(\r) = \bigoplus_{\tau_{i} = \rho||^{x}} \sigma_{i}.
\]
Note unlike $\Jac_{P}\r$, in our definition $\Jac_{x}\r$ is always semisimple. If we have an ordered sequence of real numbers $\{x_{1}, \cdots, x_{s}\}$, we can define
\[
\Jac_{x_{1}, \cdots, x_{s}}\r = \Jac_{x_{s}} \circ \cdots \circ \Jac_{x_{1}} \r.
\]
It is not hard to see $\Jac_{x}$ can be defined for $GL(n)$ in a similar way by replacing $G_{-}$ by $GL(n_{-})$. Furthermore, we can define $\Jac^{op}_{x}$ analogous to $\Jac_{x}$ but with respect to $\rho^{\vee}$ and the standard Levi subgroup $GL(n_{-}) \times GL(d_{\rho^{\vee}})$. So let us define $\Jac^{\theta}_{x} = \Jac_{x} \circ \Jac^{op}_{-x}$ for $GL(n)$. Next we want to give some properties of this modified Jacquet functor.

\begin{lemma}
\label{lemma: existence of inclusion}
If $\r \in \Rep(G)$ is irreducible, and $\Jac_{x, \cdots, y} \r =\sigma$ for $\sigma \in \Rep(G_{-})$. Then there exists an irreducible constituent $\sigma'$ in $\sigma$ so that we get an inclusion
\[
\r \hookrightarrow \rho||^{x} \times \cdots \times \rho||^{y} \rtimes \sigma'.
\] 
\end{lemma}

\begin{proof}
By Theorem~\ref{thm: B-Z}, there exists a standard parabolic subgroup $P_{-}$ of $G_{-}$ with an irreducible supercuspidal representation $\r_{M_{-}}$ on the Levi component $M_{-}$ such that there is a nontrivial equivariant homomorphism from $\sigma$ to $\Ind_{P_{-}}^{G_{-}}\r_{M_{-}}$. Then by Frobenius reciprocity, $\r_{M_{-}}$ is in $s.s.\Jac_{P_{-}} \sigma$. In particular, we can take $M = GL(d_{\rho}) \times \cdots \times GL(d_{\rho}) \times M_{-}$ with $P$ being the corresponding standard parabolic subgroup of $G$, and take $\r_{M} = \rho||^{x} \otimes \cdots \otimes\rho||^{y} \otimes \r_{M_{-}}$ to be an irreducible supercuspidal representation of $M$. Then $\r_{M}$ is in $s.s.\Jac_{P}\r$. By Lemma~\ref{lemma: A}, we know 
\[
\r \hookrightarrow \rho||^{x} \times \cdots \times \rho||^{y} \rtimes \Ind^{G_{-}}_{P_{-}}(\r_{M_{-}}). 
\]
So by Lemma~\ref{lemma: B} there exists an irreducible constituent $\sigma'$ in $\Ind^{G_{-}}_{M_{-}}(\r_{M_{-}})$ such that 
\[
\r \hookrightarrow \rho||^{x} \times \cdots \times \rho||^{y} \rtimes \sigma'.
\]
Finally by Frobenius reciprocity again, we know $\sigma'$ is in $\Jac_{x, \cdots, y}\r = \sigma$. This finishes the proof.

\end{proof}

As a special case of this lemma, we have the following corollary.

\begin{corollary}
\label{cor: existence of inclusion}
If $\r \in \Rep(G)$ is irreducible, and $\Jac_{x, \cdots, y} \r =\sigma$ for $\sigma \in \Rep(G_{-})$, which is also irreducible. Then there is an inclusion
\[
\r \hookrightarrow \rho||^{x} \times \cdots \times \rho||^{y} \rtimes \sigma.
\] 
\end{corollary}

\begin{remark}
\label{rk: existence of inclusion}
Lemma~\ref{lemma: existence of inclusion} and Corollary~\ref{cor: existence of inclusion} are also valid in the case of general linear groups, and the proofs are the same. 
\end{remark}

\begin{lemma}
\label{lemma: switching}
If $\r \in \Rep(G)$ and $|x - y| \neq 1$, then $\Jac_{x, y} \r = \Jac_{y, x} \r$.
\end{lemma}

\begin{proof}
We take the standard parabolic subgroup $P=MN$ of $G$ with $M = GL(2d_{\rho}) \times G_{-}$. If
\[
s.s.\Jac_{P} \r = \bigoplus_{i} \tau_{i} \otimes \sigma_{i},
\]
then $\sigma_{i}$ is in $\Jac_{x, y}\r$ if and only if $\Jac_{x,y}\tau_{i} \neq 0$. Let us assume $\Jac_{x, y}\tau_{i} \neq 0$, by Corollary~\ref{cor: existence of inclusion} (also see Remark~\ref{rk: existence of inclusion}) we have $\tau_{i} \hookrightarrow \rho||^{x} \times \rho||^{y}$. Since $|x - y| \neq 1$, $\rho||^{x} \times \rho||^{y} \cong \rho||^{y} \times \rho||^{x}$ is irreducible (see Appendix~\ref{sec: reducibility for GL(n)}), so we must have $\tau_{i} \cong \rho||^{x} \times \rho||^{y}$. Hence 
\[
\Jac_{x, y} \r = \bigoplus_{\tau_{i} \cong \rho||^{x} \times \rho||^{y}} (\Jac_{x, y} \tau_{i}) \otimes \sigma_{i}.
\]
By the same argument, we have 
\[
\Jac_{y, x} \r = \bigoplus_{\tau_{i} \cong \rho||^{y} \times \rho||^{x}} (\Jac_{y, x} \tau_{i}) \otimes \sigma_{i}.
\]
Therefore, $\Jac_{x, y} \r = \Jac_{y, x} \r$.

\end{proof}

\begin{lemma}
\label{lemma: subrep GL(n)}
Suppose $\r \in \Rep(GL(d_{\rho}(|a-b|+1)))$ is an irreducible constituent of 
\[
\rho||^{a} \times \cdots \times \rho||^{b}
\]
for a segment $\{a, \cdots, b\}$, and $\Jac_{x}\r = 0$ unless $x = a$, then 
\(
\r = <a, \cdots, b>.
\)
\end{lemma}

\begin{proof}
It is clear that $\Jac_{x}\r = 0$ unless $x \in \{a, \cdots, b\}$. Suppose $\{a, \cdots, y\} \subseteq \{a, \cdots, b\}$ is the longest segments such that 
\[
\Jac_{a, \cdots, y}\r \neq 0.
\]
If $y \neq b$, then we can find $z \in \{a, \cdots, b\} \backslash \{a\}$ such that $|x-z| > 1$ for all $x \in \{a, \cdots, y\}$ and $\Jac_{a, \cdots, y, z}\r \neq 0$. By Lemma~\ref{lemma: switching}, 
\[
\Jac_{z, a, \cdots, y}\r = \Jac_{a, \cdots, y, z}\r \neq 0.
\]
This means $\Jac_{z}\r \neq 0$, and we get a contradiction. So we can only have $y = b$, and by Corollary~\ref{cor: existence of inclusion} we have
\[
\r \hookrightarrow \rho||^{a} \times \cdots \times \rho||^{b}.
\]
Hence $\r = <a, \cdots, b>$.

\end{proof}

There are some explicit formulas for computing the Jacquet modules in the case of classical groups and general linear groups (see \cite[Section 1]{MoeglinTadic:2002}), and we want to recall some of them here. We will use $``\overset{s.s.}{=}"$ for equality after semisimplification.

For $GL(n)$, we know the irreducible discrete series representations are given by 
\[
St(\rho', a) = <\rho'; \frac{a-1}{2}, \cdots, -\frac{a-1}{2}>.
\] 
More generally we have irreducible representations $<\rho'; \zeta a, \cdots, \zeta b>$ attached to any decreasing segment $\{a, \cdots, b\}$ (cf. Section~\ref{sec: introduction}) for $\zeta = \pm1$. If we fix $\rho$ as before, then we have the following formulas for their Jacquet modules.
\begin{align}
\label{eq: Jac general linear group 1}
\Jac_{x} <\rho'; \zeta a, \cdots, \zeta b> = \begin{cases}
                                       <\rho'; \zeta (a-1), \cdots, \zeta b> , & \text{ if $x = \zeta a$ and $\rho' \cong \rho$,}\\
                                       0,  & \text{ otherwise; }
                                       \end{cases}
\end{align}
and
\begin{align}
\label{eq: Jac general linear group 2}
\Jac^{op}_{x} <\rho'; \zeta a, \cdots, \zeta b> = \begin{cases}
                                       <\rho'; \zeta a, \cdots, \zeta (b+1)> , & \text{ if $x = \zeta b$ and $\rho' \cong \rho^{\vee}$,}\\
                                       0,  & \text{ otherwise. }
                                       \end{cases}
\end{align}
If $\r_{i} \in \Rep(GL(n_{i}))$ for $i = 1$ or $2$, we have
\[
\Jac_{x} (\r_{1} \times \r_{2}) \overset{s.s.}{=} (\Jac_{x} \r_{1}) \times \r_{2} \+ \r_{1} \times (\Jac_{x} \r_{2}),
\] 
and 
\[
\Jac^{op}_{x} (\r_{1} \times \r_{2}) \overset{s.s.}{=} (\Jac^{op}_{x} \r_{1}) \times \r_{2} \+ \r_{1} \times (\Jac^{op}_{x} \r_{2}).
\]
Suppose $\r \in \Rep(G)$ and $\tau \in \Rep(GL(d))$. If $G$ is symplectic or special odd orthogonal, then

\begin{align*}
\label{eq: Jac semiproduct}
\Jac_{x} (\tau \rtimes \r) \overset{s.s.}{=} (\Jac_{x} \tau) \rtimes \r \+ (\Jac^{op}_{-x} \tau) \rtimes \r \+ \tau \rtimes \Jac_{x} \r.
\end{align*}
If $G = SO(2n, \eta)$, the situation is more complicated, and we would like to divide it into three cases. 

\begin{enumerate}

\item When $n \neq d_{\rho}$ or $0$,  
\[
\Jac_{x} (\tau \rtimes \r) \overset{s.s.}{=} \tau \rtimes \Jac_{x}\r \+
                                            \begin{cases}
                                           (\Jac_{x} \tau) \rtimes \r \+ (\Jac^{op}_{-x} \tau) \rtimes \r               &\text{ if $d_{\rho}$ is even } \\
                                           (\Jac_{x} \tau) \rtimes \r \+ (\Jac^{op}_{-x} \tau) \rtimes \r^{\theta_{0}}  &\text{ if $d_{\rho}$ is odd }
                                           \end{cases} 
\]

\item When $n = d_{\rho}$,
\[
\Jac_{x} (\tau \rtimes \r) \overset{s.s.}{=} \tau \rtimes \Jac_{x}\r \+ (\tau \rtimes \Jac_{x}\r^{\theta_{0}})^{\theta_{0}} \+
                                            \begin{cases}
                                           (\Jac_{x} \tau) \rtimes \r \+ (\Jac^{op}_{-x} \tau) \rtimes \r               &\text{ if $d_{\rho}$ is even } \\
                                           (\Jac_{x} \tau) \rtimes \r \+ (\Jac^{op}_{-x} \tau) \rtimes \r^{\theta_{0}}  &\text{ if $d_{\rho}$ is odd }
                                           \end{cases} 
\]

\item When $n=0$ and $d \neq d_{\rho}$,
\[
\Jac_{x} (\tau \rtimes 1) \overset{s.s.}{=} \begin{cases}
                                           (\Jac_{x} \tau) \rtimes 1 \+ (\Jac^{op}_{-x} \tau) \rtimes 1               &\text{ if $d_{\rho}$ is even } \\
                                           (\Jac_{x} \tau) \rtimes 1 \+ (\Jac^{op}_{-x} \tau \rtimes 1)^{\theta_{0}}  &\text{ if $d_{\rho}$ is odd }
                                           \end{cases} 
\]           

\item When $n=0$ and $d = d_{\rho}$,
\[
\Jac_{x} (\tau \rtimes 1) \overset{s.s.}{=} \begin{cases}
                                           (\Jac_{x} \tau) \rtimes 1 \+ (\Jac^{op}_{-x} \tau) \rtimes 1               &\text{ if $d_{\rho}$ is even } \\
                                           (\Jac_{x} \tau) \rtimes 1   &\text{ if $d_{\rho}$ is odd }
                                           \end{cases} 
\]                                            

\end{enumerate}
The formulas for special even orthogonal groups here are deduced from \cite[Theorem 3.4]{Jantzen:2006}. At last we define 
\[
\bar{\Jac}_{x} = \begin{cases}
                         \Jac_{x} + \Jac_{x} \circ \theta_{0},  & \text{ if $G = SO(2n)$ and $n = d_{\rho} \neq 1$, } \\
                         \Jac_{x},         & \text{ otherwise. }
                         \end{cases}
\]

Let $\bar{\Rep}(G)$ be the category of finite-length representations of $G(F)$ viewed as $\sH(G)$-modules. We denote the elements in $\bar{\Rep}(G)$ by $[\r]$ for $\r \in \Rep(G)$, and we call $[\r]$ is irreducible if $\r$ is irreducible. For $[\r] \in \bar{\Rep}(G)$, let us define 
\[
\tau \rtimes [\r] := [\tau \rtimes \r] \text{ and } \bar{\Jac}_{x} [\r] := [\bar{\Jac}_{x} \r].
\]
Then we can combine all cases into the following single formula
\begin{align}
\label{eq: Jac semiproduct}
\bar{\Jac}_{x} (\tau \rtimes [\r]) \overset{s.s.}{=} (\Jac_{x} \tau) \rtimes [\r] \+ (\Jac^{op}_{-x} \tau) \rtimes [\r] \+ \tau \rtimes \bar{\Jac}_{x} [\r].
\end{align}

Finally, we would like to extend the discussion of this section to the category $\Rep(G^{\Sigma_{0}})$ of finite-length representations of $G^{\Sigma_{0}}(F)$. Let $P = MN$ be a standard parabolic subgroup of $G$. If $M$ is $\theta_{0}$-stable, we write $M^{\Sigma_{0}} := M \rtimes \Sigma_{0}$ and $P^{\Sigma_{0}} := P \rtimes \Sigma_{0}$. Otherwise, we let $M^{\Sigma_{0}} = M$ and $P^{\Sigma_{0}} = P$. Note when $G^{\Sigma_{0}}$ is even orthogonal group, one can define the normalized parabolic induction and normalized Jacquet module in a similar way (see \cite[Section 6]{Ban:1999}).
Suppose $\sigma^{\Sigma_{0}} \in \Rep(M^{\Sigma_{0}})$, $\r^{\Sigma_{0}} \in \Rep(G^{\Sigma_{0}})$. It follows from the definition that
\[
(\Jac_{P^{\Sigma_{0}}} \r^{\Sigma_{0}})|_{M} = \Jac_{P}(\r^{\Sigma_{0}}|_{G}).
\]
And 
\[
(\Ind^{G^{\Sigma_{0}}}_{P^{\Sigma_{0}}} \sigma^{\Sigma_{0}})|_{G} = \Ind^{G}_{P}(\sigma^{\Sigma_{0}}|_{M}),
\] 
unless $G$ is special even orthogonal and $M^{\Sigma_{0}} = M$, in which case 
\[
(\Ind^{G^{\Sigma_{0}}}_{P^{\Sigma_{0}}} \sigma^{\Sigma_{0}})|_{G} = \Ind^{G}_{P}(\sigma^{\Sigma_{0}}|_{M}) \+ (\Ind^{G}_{P}(\sigma^{\Sigma_{0}}|_{M}))^{\theta_{0}}.
\] 
Let us define
\[
\bar{\Jac}_{P} = \begin{cases}
                         \Jac_{P} + \Jac_{P} \circ \theta_{0},  & \text{ if $G = SO(2n)$ and $M^{\theta_{0}} \neq M$,} \\
                         \Jac_{P},         & \text{ otherwise. }
                         \end{cases}
\]
And
\[
\Ind^{G}_{P} [\sigma] := [\Ind^{G}_{P} \sigma] \text{ and } \bar{\Jac}_{P} [\r] := [\bar{\Jac}_{P} \r].
\]
Then we have
\[
(\Jac_{P^{\Sigma_{0}}} \, \Ind^{G^{\Sigma_{0}}}_{P^{\Sigma_{0}}} \sigma^{\Sigma_{0}})|_{M} =  \bar{\Jac}_{P} \, \Ind^{G}_{P} (\sigma^{\Sigma_{0}} |_{M}), 
\]
and
\[
[(\Ind^{G^{\Sigma_{0}}}_{P^{\Sigma_{0}}} \, \Jac_{P^{\Sigma_{0}}} \r^{\Sigma_{0}})|_{G}] = \Ind^{G}_{P} \, \bar{\Jac}_{P} [\r^{\Sigma_{0}}|_{G}].
\]
The Frobenius reciprocity still holds in this case, i.e.,
\[
\Hom_{M^{\Sigma_{0}}}(\Jac_{P^{\Sigma_{0}}} \r^{\Sigma_{0}}, \sigma^{\Sigma_{0}}) \cong \Hom_{G^{\Sigma_{0}}}(\r^{\Sigma_{0}}, \Ind^{G^{\Sigma_{0}}}_{P^{\Sigma_{0}}} \sigma^{\Sigma_{0}}).
\]
Moreover, the results of this section can be stated similarly for representations of $G^{\Sigma_{0}}(F)$. In particular, for $\tau \in \Rep(GL(d))$, we have 
\begin{align}
\label{eq: Jac semiproduct full orthogonal group}
\Jac_{x} (\tau \rtimes \r^{\Sigma_{0}}) \overset{s.s.}{=} (\Jac_{x} \tau) \rtimes \r^{\Sigma_{0}} \+ (\Jac^{op}_{-x} \tau) \rtimes \r^{\Sigma_{0}} \+ \tau \rtimes \Jac_{x} \r^{\Sigma_{0}}.
\end{align}

\section{Compatibility of Jacquet modules with endoscopic transfer}
\label{sec: compatibility of Jacquet modules with endoscopic transfer}

As normalized parabolic induction is compatible with endoscopic transfer, the normalized Jacquet module is also compatible with endoscopic transfer. Since the Jacquet module is originally defined on representations, we need to first extend it to the space of finite linear combinations of (twisted) characters, and then to the space of (twisted) invariant distributions (see \eqref{eq: Jacquet on distribution}). In particular, this extended Jacquet functor will preserve stability (see \eqref{eq: Jacquet preserve stability}). If $G$ is any quasisplit connected reductive group over $F$ and $\theta$ is an $F$-automorphism of $G$ preserving an $F$-splitting, we will denote the space of finite linear combinations of twisted characters of $G(F)$ by $R(G^{\theta})$ and denote the space of stable finite linear combinations of characters on $G(F)$ by $R(G)^{st}$. Moreover, let $\D{I}(G^{\theta})$) be the space of twisted invariant distributions on $G(F)$ and $\D{SI}(G)$ be the space of stable invariant distributions on $G(F)$. When $G = GL(N)$, we will simply write $R(N^{\theta})$ and $\D{I}(N^{\theta})$ for the corresponding spaces. In the following discussion we will assume $G$ is a quasisplit symplectic or special orthogonal group.

Suppose $H$ is an elliptic endoscopic group of $G$, we know from Section~\ref{sec: endoscopy} that $H = G_{I} \times G_{II}$, and there is an embedding 
\[
\xi: \L{H} \hookrightarrow \L{G},
\]
where $\L{H} = \D{H} \rtimes \Gal{L/F}$ for $L = F, E_{II}$ or $E_{I}E_{II}$ accordingly. We fix $\Gal{}$-splittings $(\mathcal{B}_{H}, \mathcal{T}_{H}, \{\mathcal{X}_{\alpha_{H}}\})$ and $(\mathcal{B}_{G}, \mathcal{T}_{G}, \{\mathcal{X}_{\alpha_{G}}\})$ for $\D{H}$ and $\D{G}$ respectively. By taking certain $\D{G}$-conjugate of $\xi$, we can assume  $\xi(\mathcal{T}_{H}) = \mathcal{T}_{G}$ and $\xi(\mathcal{B}_{H}) \subseteq \mathcal{B}_{G}$. So we can view the Weyl group $W_{H} = W(\D{H}, \mathcal{T}_{H})$ as a subgroup of $W_{G} = W(\D{G}, \mathcal{T}_{G})$. We also view $\L{H}$ as a subgroup of $\L{G}$ through $\xi$. 

We fix a standard parabolic subgroup $P = MN$ of $G$ with standard embedding $\L{P} \hookrightarrow \L{G}$. Then there exists a torus $S \subseteq \mathcal{T}_{G}$ such that $\L{M} = \Cent(S, \L{G})$. Let $W_{M} =  W(\D{M}, \mathcal{T}_{G})$. We define
\[
W_{G}(H, M) := \{w \in W_{G} | \, \Cent(w(S), \L{H}) \rightarrow \Gal{L/F} \text{ surjective }\}.
\]
For any $w \in W_{G}(H, M)$, let us take $g \in \D{G}$ such that $\Int(g)$ induces $w$. Since $\Cent(w(S), \L{H}) \rightarrow \Gal{L/F}$ is surjective, $g \L{P} g^{-1} \cap \L{H}$ defines a parabolic subgroup of $\L{H}$ with Levi component $g \L{M} g^{-1} \cap \L{H}$. So we can choose a standard parabolic subgroup $P'_{w} = M'_{w}N'_{w}$ of $H$ with standard embedding $\L{P'_{w}} \hookrightarrow \L{H}$ such that $\L{P'_{w}}$ (resp. $\L{M'_{w}}$) is $\D{H}$-conjugate to $g \L{P} g^{-1} \cap \L{H}$ (resp. $g \L{M} g^{-1} \cap \L{H}$). In fact, $M'_{w}$ is an endoscopic group of $M$ (not necessarily elliptic, see \cite{KottwitzShelstad:1999}), and we have an embedding $\xi_{M'_{w}}: \L{M'_{w}} \rightarrow \L{M}$ given by the following diagram:
\[
\xymatrix{\L{P'_{w}} \ar@{^{(}->}[d] & \L{M'_{w}} \ar[l] \ar[r]^{\xi_{M'_{w}}} & \L{M} \ar[r] & \L{P} \ar@{^{(}->}[d] \\
\L{H} \ar[r]^{\Int(h)} & \L{H} \ar[r]^{\xi} & \L{G} & \L{G} \ar[l]_{\Int(g)} 
}
\]
where $h \in \D{H}$ induces an element in $W_{H}$. Note the choice of $h$ is unique up to $\D{M}'_{w}$-conjugation, and so is $\xi_{M'_{w}}$. If we change $g$ to $h'gm$, where $h' \in \D{H}$ induces an element in $W_{H}$ and $m \in \D{M}$ induces an element in $W_{M}$, then we still get $P'_{w}$, but $\xi_{M'_{w}}$ changes to $\Int(m^{-1}) \circ \xi_{M'_{w}}$ up to $\D{M}'_{w}$-conjugation. Therefore, for any element $w$ in 
\[
W_{H} \backslash W_{G}(H, M)/W_{M},
\]
we can associate a standard parabolic subgroup $P'_{w} = M'_{w} N'_{w}$ of $H$ and a $\D{M}$-conjugacy class of embedding $\xi_{M'_{w}}: \L{M'_{w}} \rightarrow \L{M}$. Moreover, we have the following commutative diagram:
\begin{align}
\label{eq: compatible with endoscopic transfer 0} 
\xymatrix{\D{SI}(H) \ar[d]_{\+_{w} \Jac_{P'_{w}}}  \ar[r] & \D{I}(G) \ar[d]^{\Jac_{P}} \\
                \bigoplus_{w} \D{SI}(M'_{w}) \ar[r]  & \D{I}(M), }
\end{align}
where the sum is over $W_{H} \backslash W_{G}(H, M)/W_{M}$, and the horizontal maps correspond to the spectral endoscopic transfers with respect to $\xi$ on the top and $\xi_{M'_{w}}$ on the bottom.

Suppose $M = GL(m) \times G_{-}$, then the Levi subgroups $M'_{w}$ of $H$ appearing in \eqref{eq: compatible with endoscopic transfer 0} are of the form $M_{I} \times M_{II}$, where $M_{I} \cong GL(m_{I}) \times G_{I -}$ is a Levi subgroup of $G_{I}$ and $M_{II} \cong GL(m_{II}) \times G_{II -}$ is a Levi subgroup of $G_{II}$ with $m = m_{I} + m_{II}$. The spectral endoscopic transfer maps $\D{SI}(G_{I -} \times G_{II -})$ to $\D{I}(G_{-})$, and it also maps $\D{SI}(GL(m_{I}) \times GL(m_{II}))$ to $\D{I}(GL(m))$, which is given by parabolic induction. Now we fix a unitary irreducible supercuspidal representation $\rho$ of $GL(d_{\rho}, F)$, and let $m = d_{\rho}$. We would like to restrict \eqref{eq: compatible with endoscopic transfer 0} to invariant distributions of $M(F)$ such that on $GL(d_{\rho}, F)$ they are given by the character of $\rho||^{x}$, then the relevant Levi subgroups of $H$ will satisfy $m_{I} = 0$ or $m_{II} = 0$. Since there is no canonical projection from $\D{I}(M)$ to such distributions, we will have to restrict \eqref{eq: compatible with endoscopic transfer 0} to spaces of finite linear combinations of characters first.

Let us write
\[
M'_{w_{I}} = GL(d_{\rho}) \times H_{I -} := GL(d_{\rho}) \times G_{I -} \times G_{II}
\]
and 
\[
M'_{w_{II}} = GL(d_{\rho}) \times H_{II -} := GL(d_{\rho}) \times G_{I} \times G_{II-}.
\]
We also keep the notations in Example~\ref{eg: endoscopy}, in particular when $G$ is symplectic, $G_{I}$ is symplectic and $G_{II}$ is special even orthogonal. Let $\theta_{i} = \theta_{0}$ with respect to $G_{i}$ for $i = I, II$. Then we have the following cases.

\begin{enumerate}

\item If $G$ is symplectic, then $M'_{w_{I}}$, $M'_{w_{II}}$ and $(M'_{w_{II}})^{\theta_{II}}$ are the relevant standard Levi subgroups of $H$. Note $M'_{w_{II}} = (M'_{w_{II}})^{\theta_{II}}$ if and only if $G_{II-} \neq 1$. We get a modified diagram of \eqref{eq: compatible with endoscopic transfer 0} as follows.
\begin{align}
\label{eq: compatible with endoscopic transfer 1} 
\xymatrix{R(H)^{st} \ar[d]_{\Jac^{I}_{x} \+ \bar{\Jac}_{x}}  \ar[r] & R(G) \ar[d]^{\Jac_{x}} \\
                R(H_{I -})^{st} \+ R(H_{II-})^{st} \ar[r]  & R(G_{-}),  }
\end{align}
where $\Jac^{I}_{x}$ is with respect to $\rho \otimes \eta_{I}$.

\item If $G$ is special odd orthogonal, then $M'_{w_{I}}$, $M'_{w_{II}}$ are the only relevant standard Levi subgroups of $H$, and we get a modified diagram of \eqref{eq: compatible with endoscopic transfer 0} as follows.
\begin{align}
\label{eq: compatible with endoscopic transfer 2} 
\xymatrix{R(H)^{st} \ar[d]_{\Jac_{x} \+ \Jac_{x}}  \ar[r] & R(G) \ar[d]^{\Jac_{x}} \\
                R(H_{I -})^{st} \+ R(H_{II-})^{st} \ar[r]  & R(G_{-})  }
\end{align}

\item If $G$ is special even orthogonal, then $M'_{w_{I}}$, $(M'_{w_{I}})^{\theta_{I}}$, $M'_{w_{II}}$ and $(M'_{w_{II}})^{\theta_{II}}$ are the relevant standard Levi subgroups of $H$. Note $M'_{w_{i}} = (M'_{w_{i}})^{\theta_{i}}$ if and only if $G_{i-} \neq 1$ for $i = I, II$. We get a modified diagram of \eqref{eq: compatible with endoscopic transfer 0} as follows.
\begin{align}
\label{eq: compatible with endoscopic transfer 3} 
\xymatrix{R(H)^{st} \ar[d]_{\bar{\Jac}_{x} \+ \bar{\Jac}_{x}}  \ar[r] & R(G) \ar[d]^{\Jac_{x}} \\
                R(H_{I -})^{st} \+ R(H_{II-})^{st} \ar[r]  & R(G_{-})  
                }
\end{align}
\end{enumerate}

Next we view $G$ as a twisted elliptic endoscopic group of $GL(N)$, and there is an embedding 
\[
\xi_{N}: \L{G} \hookrightarrow GL(N, \C)
\]
where $\L{G} = \D{G} \rtimes \Gal{L/F}$ for $L = F$ or $E$ accordingly. We also fix a $\D{\theta}_{N}$-stable $\Gal{}$-splitting $(\mathcal{B}_{N}, \mathcal{T}_{N}, \{\mathcal{X}_{\alpha_{N}}\})$ of $GL(N, \C)$. By taking certain $GL(N, \C)$-conjugate of $\xi_{N}$, we can assume $\xi_{N}(\mathcal{T}_{G}) = (\mathcal{T}_{N}^{\D{\theta}_{N}})^{0}, \xi_{N}(\mathcal{B}_{G}) \subseteq \mathcal{B}_{N}$. So we can view the Weyl group $W_{G} = W(\D{G}, \mathcal{T}_{G})$ as a subgroup of $W_{N^{\theta}} := W(GL(N, \C), \mathcal{T}_{N})^{\D{\theta}_{N}}$. We also view $\L{G}$ as a subgroup of $GL(N, \C)$ through $\xi_{N}$. 

We fix a standard $\theta_{N}$-stable parabolic subgroup $P = MN$ of $GL(N)$ with standard embedding $\L{P} \hookrightarrow GL(N, \C)$. Then there exists a torus $S \subseteq (\mathcal{T}_{N}^{\D{\theta}_{N}})^{0}$ such that $\L{M} = \Cent(S, GL(N, \C))$. Let $W_{M^{\theta}} = W(\D{M}, \mathcal{T}_{N})^{\D{\theta}_{N}}$. We define
\[
W_{N^{\theta}}(G, M) := \{w \in W_{N^{\theta}} | \, \Cent(w(S), \L{G}) \rightarrow \Gal{L/F} \text{ surjective }\}.
\]
For any $w \in W_{N^{\theta}}(G, M)$, let us take $g_{N} \in GL(N, \C)$ such that $\Int(g_{N})$ induces $w$. Since $\Cent(w(S), \L{G}) \rightarrow \Gal{L/F}$ is surjective, $g_{N} \L{P} g_{N}^{-1} \cap \L{G}$ defines a parabolic subgroup of $\L{G}$ with Levi component $g_{N} \L{M} g_{N}^{-1} \cap \L{G}$. So we can choose a standard parabolic subgroup $P'_{w} = M'_{w}N'_{w}$ of $H$ with standard embedding $\L{P'_{w}} \hookrightarrow \L{G}$ such that $\L{P'_{w}}$ (resp. $\L{M'_{w}}$) is $\D{G}$-conjugate to $g_{N} \L{P} g_{N}^{-1} \cap \L{G}$ (resp. $g_{N} \L{M} g_{N}^{-1} \cap \L{G}$). As before, $M'_{w}$ can be viewed as a twisted endoscopic groups of $M$ (not necessarily elliptic, see \cite{KottwitzShelstad:1999}), and we have an embedding $\xi_{M'_{w}}: \L{M'_{w}} \rightarrow \L{M}$ given by the following diagram:
\[
\xymatrix{\L{P'_{w}} \ar@{^{(}->}[d]  & \L{M'_{w}} \ar[l] \ar[r]^{\xi_{M'_{w}}} & \L{M} \ar[r] & \L{P} \ar@{^{(}->}[d]  \\
\L{G} \ar[r]^{\Int(g)} & \L{G} \ar[r]^{\xi_{N} \quad}  & GL(N, \C) & GL(N, \C) \ar[l]_{\Int(g_{N})} \\
}
\]
where $g \in \D{G}$ induces an element in $W_{G}$. Note the choice of $g$ is unique up to $\D{M}'_{w}$-conjugation, and so is $\xi_{M'_{w}}$. If we change $g_{N}$ to $g'g_{N}m$, where $g' \in \D{G}$ induces an element in $W_{G}$ and $m \in \D{M}$ induces an element in $W_{M^{\theta}}$, then we still get $P'_{w}$, but $\xi_{M'_{w}}$ changes to $\Int(m^{-1}) \circ \xi_{M'_{w}}$ up to $\D{M}'_{w}$-conjugation. Therefore, for any element $w$ in 
\[
W_{G} \backslash W_{N^{\theta}}(G, M)/W_{M^{\theta}}
\]
we can associate a standard parabolic subgroup $P'_{w} = M'_{w} N'_{w}$ of $G$ and a $\D{M}$-conjugacy class of embedding $\xi_{M'_{w}}: \L{M'_{w}} \rightarrow \L{M}$. Moreover, we have the following commutative diagram:
\begin{align}
\label{eq: compatible with twisted endoscopic transfer 0} 
\xymatrix{\D{SI}(G) \ar[d]_{\+_{w} \Jac_{P'_{w}}}  \ar[r] & \D{I}(N^{\theta}) \ar[d]^{\Jac_{P}} \\
                \bigoplus_{w} \D{SI}(M'_{w}) \ar[r]  & \D{I}(M^{\theta}). }
\end{align}
where the sum is over $W_{G} \backslash W_{N^{\theta}}(G, M)/W_{M^{\theta}}$, and the horizontal maps correspond to the twisted spectral endoscopic transfers with respect to $\xi$ on the top and $\xi_{M'_{w}}$ on the bottom.

We again fix a unitary irreducible supercuspidal representation $\rho$ of $GL(d_{\rho}, F)$, and suppose
\[
M = GL(d_{\rho}) \times GL(N_{-}) \times GL(d_{\rho^{\vee}}).
\]
Then the standard Levi subgroups of $G$ appearing in \eqref{eq: compatible with twisted endoscopic transfer 0} are
\[
M'_{w} = GL(d_{\rho}) \times G_{-}
\]
and $(M'_{w})^{\theta_{0}}$. Note $M'_{w} = (M'_{w})^{\theta_{0}}$ unless $G$ is special even orthogonal and $N_{-} = 0$. For the purpose of restricting \eqref{eq: compatible with twisted endoscopic transfer 0} to the twisted invariant distributions of $M(F)$ such that on $GL(d_{\rho}, F) \times GL(d_{\rho^{\vee}}, F)$ they are given by the twisted character of $\rho||^{x} \otimes \rho^{\vee}||^{-x}$, we will have to first restrict the diagram to spaces of finite linear combinations of (twisted) characters. Then we can get a modified diagram as follows
\begin{align}
\label{eq: compatible with twisted endoscopic transfer} 
\xymatrix{R(G)^{st} \ar[d]_{\bar{\Jac}_{x}}  \ar[r] & R(N^{\theta}) \ar[d]^{\Jac^{\theta}_{x}} \\
                R(G_{-})^{st} \ar[r]  & R(N_{-}^{\theta}).  }
\end{align}

At last, when $G$ is special even orthogonal, there is a twisted version of the diagram \eqref{eq: compatible with endoscopic transfer 3}, which can be derived as in the case of $GL(N)$ (also see Appendix~\ref{sec: Casselman's formula} for the general case). Here we will only state the result using the notations from Section~\ref{sec: endoscopy} and \eqref{eq: compatible with endoscopic transfer 3}. We assume $G_{-} \neq 1$.
\begin{align}
\label{eq: compatible with endoscopic transfer full orthogonal group} 
\xymatrix{R(H)^{st} \ar[d]_{\Jac^{I}_{x} \+ \Jac^{II}_{x}}  \ar[r] & R(G^{\theta_{0}}) \ar[d]^{\Jac_{x}} \\
                R(H_{I -})^{st} \+ R(H_{II-})^{st} \ar[r]  & R(G^{\theta_{0}}_{-}),  }
\end{align}
where $\Jac^{i}_{x}$ is with respect to $\rho \otimes \eta_{i}$ for $i = I, II$.

Both diagrams~\eqref{eq: compatible with endoscopic transfer 0} and \eqref{eq: compatible with twisted endoscopic transfer 0} can be established by using Casselman's formula \cite{Casselman:1977} and its twisted version for relating the (twisted) characters of representations with that of their {\bf unnormalized} Jacquet modules (see \cite{MW:2006} and \cite{Hiraga:2004}). For the convenience of the reader, we will give the proof of the general case in Appendix~\ref{sec: Casselman's formula}. In the next section, we are going to prove Theorem~\ref{thm: supercuspidal parametrization} by applying \eqref{eq: compatible with twisted endoscopic transfer}  (resp. \eqref{eq: compatible with endoscopic transfer 1}, \eqref{eq: compatible with endoscopic transfer 2} and \eqref{eq: compatible with endoscopic transfer 3}) to the (twisted) endosopic character identity \eqref{eq: character relation GL(N)} (resp. \eqref{eq: character relation}). We will only need \eqref{eq: compatible with endoscopic transfer full orthogonal group} in Section~\ref{sec: even orthogonal group}.

\section{Proof of Theorem~\ref{thm: supercuspidal parametrization}}
\label{sec: proof}

In the following sections we will always assume $G$ is a quasisplit symplectic group or special orthogonal group. Before we start the proof, we would like to make explicit the effects of Jacquet modules on the (twisted) endoscopic character identities \eqref{eq: character relation GL(N)} and \eqref{eq: character relation}. So let us fix a self-dual irreducible unitary supercuspidal representation $\rho$ of $GL(d_{\rho}, F)$ and a real number $x$. Let $\p \in \cPdt{G}$ and we define $\p_{-} \in \cPbd{G_{-}}$ by its $Jord(\p_{-})$ as follows. 
\[
Jord(\p_{-}) = Jord(\p) \cup \{(\rho, 2x-1)\} \backslash \{(\rho, 2x+1)\},
\] 
if $(\rho, 2x+1) \in Jord(\p)$ and $x > 0$, or $\emptyset$ otherwise. We also set $\r_{\p_{-}} = 0$ if $Jord(\p_{-}) = \emptyset$. Note $\p_{-}$ depends on both $\rho$ and $x$. The following lemma is clear by our explicit formulas \eqref{eq: Jac general linear group 1} and \eqref{eq: Jac general linear group 2}.

\begin{lemma}
\(
\r_{\p_{-}} = \Jac^{\theta}_{x} \r_{\p}.
\) 
\end{lemma}

So after applying \eqref{eq: compatible with twisted endoscopic transfer} to the twisted endoscopic identity \eqref{eq: character relation GL(N)}, we have 
\begin{align}
\label{eq: Jac character relation GL(N)}
f^{G_{-}}(\sum_{[\r] \in \cPkt{\p}} \bar{\Jac}_{x}\r) = f_{N_{-}^{\theta}}(\r_{\p_{-}})
\end{align}
for $f \in C^{\infty}_{c}(GL(N_{-}), F)$.
Since Theorem~\ref{thm: character relation} is also valid for all tempered parameters (see Remark~\ref{rk: character relation}), the left hand side of \eqref{eq: Jac character relation GL(N)} has to be $f^{G_{-}}(\p_{-})$. Then we have the following result.

\begin{lemma}

\begin{align}
\label{eq: Jac packet}
\cPkt{\p_{-}} = \bar{\Jac}_{x} \cPkt{\p}.
\end{align}
In particular,
\begin{align}
\label{eq: Jac vanishing}
\bar{\Jac}_{x} \cPkt{\p} = 0 \text{ if $(\rho, 2x+1) \notin Jord(\p)$}.
\end{align}

\end{lemma}

\begin{proof}
Since the transfer map \eqref{eq: twisted endoscopic transfer} is surjective onto the space of $\Sigma_{0}$-invariant stable orbital integrals of $G(F)$, we have
\[
f(\p_{-}) = f(\sum_{[\r] \in \cPkt{\p}} \bar{\Jac}_{x}\r) = \sum_{[\r] \in \cPkt{\p}} f_{G}(\bar{\Jac}_{x}\r)
\]
for $f \in \sH(G)$. Then the lemma follows from the linear independence of characters of irreducible smooth representations of $G(F)$.
\end{proof}

Suppose $Jord(\p_{-}) \neq \emptyset$, then $x > 0$ and there are two possibilites, i.e., $\p_{-} \in \cPdt{G_{-}}$, or
\[
\p_{-} = 2\p_{1} \+ \p_{2} \+ \cdots \+ \p_{r} \in \cPbd{G_{-}}
\]  
as in \eqref{eq: decomposition}, where $\p_{1} = \rho \otimes [2x-1]$. In the first case we have $(\rho, 2x-1) \notin Jord(\p)$. If $x \neq 1/2$, then there is a canonical isomorphism $\S{\p} \cong \S{\p_{-}}$ 
after identifying $Jord(\p)$ with $Jord(\p_{-})$ by sending $(\rho, 2x+1)$ to $(\rho, 2x-1)$. If $x =1/2$, we have a projection from $\S{\p}$ to $\S{\p_{-}}$ by restricting $\Two$-valued functions on $Jord(\p)$ to $Jord(\p_{-})$. Hence we get an exact sequence

\begin{align}
\xymatrix{1 \ar[r] & <s> \ar[r] & \S{\p} \ar[r] & \S{\p_{-}} \ar[r] & 1},
\end{align}
where $s(\cdot) = 1$ over $Jord(\p)$ except for $s(\rho, 1/2) = -1$.

In the second case, let $\underline{\p}_{-}: L_{F} \rightarrow \L{G}_{-}$ be a representative of $\p_{-}$. We can also identify $\S{\underline{\p}_{-}}$ and its characters $\D{\S{\underline{\p}}}_{-}$ with certain quotient spaces of $\Two$-valued functions on $Jord(\p_{-})$ (forgetting multiplicities), as in the case of discrete parameters. Note
\[
\Cent(\underline{\p}_{-}, GL(N_{-}, \C)) \cong \underbrace{GL(2, \C) \times \C^{\times} \times \cdots \times \C^{\times}}_{r},
\]
and then
\[
\Cent(\underline{\p}_{-}, \D{G}_{-}) \cong \{s = (s_{i}) \in O(2, \C) \times \Two^{r-1}: \det(s_{1})^{n_{1}} \cdot \prod_{i \neq 1} (s_{i})^{n_{i}} = 1\}.
\]
We write $z$ for the nontrivial central element of $O(2, \C)$. 
Then $\cS{\underline{\p}_{-}}^{\Sigma_{0}} \cong O(2, \C) \times \Two^{r-1} / <z, -1, \cdots, -1>$, and hence $\S{\underline{\p}_{-}}^{\Sigma_{0}} \cong \Two^{r} / <1, -1, \cdots, -1>$. If $G$ is special even orthogonal, 
\[
\S{\underline{\p}_{-}} \cong \{s = (s_{i}) \in \Two^{r} : \prod_{i} (s_{i})^{n_{i}} = 1\} / <1, -1, \cdots, -1>
\]
which is a subgroup of $\S{\underline{\p}_{-}}^{\Sigma_{0}}$ of index 1 or 2. Let us denote by $\S{\p_{-}}$ (resp. $\S{\p_{-}}^{\Sigma_{0}}$) the corresponding quotient space of $\Two$-valued functions on $Jord(\p_{-})$  (forgetting multiplicities) such that $\S{\p_{-}} \cong \S{\underline{\p}_{-}}$ (resp. $\S{\p_{-}}^{\Sigma_{0}} \cong \S{\underline{\p}_{-}}^{\Sigma_{0}}$) under these isomorphisms. There is a projection from $\S{\p} \rightarrow \S{\p_{-}}$ (resp. $\S{\p}^{\Sigma_{0}} \rightarrow \S{\p_{-}}^{\Sigma_{0}}$) by sending $s$ to $s_{-}$ such that $s_{-}(\cdot) = s(\cdot)$ over $Jord(\p) \backslash \{ (\rho, 2x+1), (\rho, 2x-1) \}$ and $s_{-}(\rho, 2x-1) = s(\rho, 2x+1)s(\rho, 2x-1)$. Hence there is a short exact sequence 

\begin{align}
\xymatrix{1 \ar[r] & <s> \ar[r] & \S{\p} \ar[r] & \S{\p_{-}} \ar[r] & 1}
\end{align}

\begin{align*}
\xymatrix{(\text{ resp. }  1 \ar[r] & <s> \ar[r] & \S{\p}^{\Sigma_{0}} \ar[r] & \S{\p_{-}}^{\Sigma_{0}} \ar[r] & 1 )}
\end{align*}
where $s(\cdot) = 1$ over $Jord(\p)$ except for $s(\rho, 2x+1) = s(\rho, 2x-1) = -1$. For the characters of $\S{\p_{-}}$, we have
\[
\D{\S{\p_{-}}^{\Sigma_{0}}} = \{\e = (\e_{i} \in \Two^{r}) : \prod_{i \neq 1} \e_{i} = 1\},
\]
and if $G$ is special even orthogonal,
\[
\D{\S{\p}}_{-} = \{\e = (\e_{i} \in \Two^{r}) : \prod_{i \neq 1} \e_{i} = 1\} / <\e_{0}>,
\]
where $\e_{0} = (\e_{0,i}) \in \D{\S{\p_{-}}^{\Sigma_{0}}}$ satisfies $\e_{0,i} = 1$ if $n_{i}$ is even, and $\e_{0,i} = -1$ if $n_{i}$ is odd. So $\e_{0}$ is trivial when restricted to $\S{\p_{-}}$. In general, let $\e_{0} = 1$ if $G$ is not special even orthogonal. 

At last we want to point out in this case $\p_{-}$ factors through $\p_{M_{-}} \in \cPdt{M_{-}}$, for a Levi subgroup $M_{-} \cong GL(n_{1}) \times G'$ and $\p_{M_{-}} = \p_{1} \times \p'$ such that 
\[
\p' = \p_{2} \+ \cdots \+ \p_{r}.
\] 
Since 
\[
\S{\underline{\p}_{M_{-}}}^{\Sigma_{0}} \hookrightarrow \S{\underline{\p}_{-}}^{\Sigma_{0}} \text{ and } \S{\underline{\p}_{M_{-}}}^{\Sigma_{0}} \cong \S{\underline{\p}'}^{\Sigma_{0}},
\] 
we can get an inclusion $\S{\p'}^{\Sigma_{0}} \hookrightarrow \S{\p_{-}}^{\Sigma_{0}}$, which in fact just extends $s'(\cdot) \in \Two^{Jord(\p')}$ trivially to $Jord(\p_{-})$ (forgetting multiplicities). So on the dual side, there is a projection 
\[
\D{\S{\p_{-}}^{\Sigma_{0}}} \rightarrow \D{\S{\p'}^{\Sigma_{0}}}
\] 
given by restricting $\e(\cdot)$ to $Jord(\p')$. Taking quotient by $<\e_{0}>$, we get $\D{\S{\p_{-}}} \rightarrow \D{\S{\p'}}$. It follows from Arthur's theory (i.e., Theorem~\ref{thm: L-packet}, \ref{thm: L-packet full orthogonal group}, \ref{thm: character relation} and \ref{thm: character relation full orthogonal group}) that
\begin{align}
\label{eq: tempered L-packet}
\cPkt{\p_{-}} = \r_{\p_{1}} \rtimes \cPkt{\p'} \,\,\,\,\,\,\,\,\,\, ( \text{ resp. } \Pkt{\p_{-}}^{\Sigma_{0}} = \r_{\p_{1}} \rtimes \Pkt{\p'}^{\Sigma_{0}} )
\end{align}
Moreover, 
\begin{align}
\label{eq: elliptic representaion}
\r_{\p_{1}} \rtimes \r(\p', \bar{\e}') = \+_{\bar{\e}' \leftarrow \bar{\e} \in \D{\S{\p}}_{-}} \, \r(\p_{-}, \bar{\e})
\end{align}
\[
\Big( \text{ resp. }  \r_{\p_{1}} \rtimes \r^{\Sigma_{0}}(\p', \e') = \+_{\e' \leftarrow \e \in \D{\S{\p_{-}}^{\Sigma_{0}}}} \, \r^{\Sigma_{0}}(\p_{-}, \e) \Big).
\]
We will need this description of $\cPkt{\p_{-}}$ (resp. $\Pkt{\p_{-}}^{\Sigma_{0}}$) in Section~\ref{sec: cuspidal support of discrete series} (resp. Section~\ref{sec: even orthogonal group}).

In all the above cases, we can canonically identify $\D{\S{\p}}_{-}$ (resp. $\D{\S{\p_{-}}^{\Sigma_{0}}}$) with a subgroup of $\D{\S{\p}}$ (resp. $\D{\S{\p}^{\Sigma_{0}}}$) of index 1 or 2, so later on we will always view $\e \in \D{\S{\p_{-}}^{\Sigma_{0}}}$ as functions on $Jord(\p)$.

\begin{lemma}
\label{lemma: Jac packet}
Suppose $\p \in \cPdt{G}$, and $(\rho, 2x+1) \in Jord(\p)$.

\begin{enumerate}

\item If $x > 1/2$ and $(\rho, 2x-1) \notin Jord(\p)$, then $\r(\p_{-}, \bar{\e}) = \bar{\Jac}_{x} \r(\p, \bar{\e})$ for all $\bar{\e} \in \D{\S{\p}}_{-} \cong \D{\S{\p}}$.

\item If $x > 1/2$ and $(\rho, 2x-1) \in Jord(\p)$, then $\bar{\Jac}_{x} \r(\p, \bar{\e}) = 0$ unless $\bar{\e} \in \D{\S{\p}}_{-}$, i.e., 
\[
\e(\rho, 2x+1) \e(\rho, 2x-1) = 1,
\]
in which case $\r(\p_{-}, \bar{\e}) = \bar{\Jac}_{x} \r(\p, \bar{\e})$.   

\item If $x = 1/2$, then $\bar{\Jac}_{1/2} \r(\p, \bar{\e}) = 0$ unless $\bar{\e} \in \D{\S{\p}}_{-}$, i.e.,
\[
\e(\rho, 2) = 1,
\]
in which case $\r(\p_{-}, \bar{\e}) = \bar{\Jac}_{x} \r(\p, \bar{\e})$.

\end{enumerate}

\end{lemma}

\begin{proof}
First we know from \eqref{eq: Jac packet} that $\cPkt{\p_{-}} = \Jac_{x}\cPkt{\p}$, so in particular $\bar{\Jac}_{x}\r$ do not have common irreducible constituents with each other for $[\r] \in \cPkt{\p}$. Next for $s \in \S{\p}$, suppose $(H, \p_{H}) \rightarrow (\p, s)$, then we have 
\begin{align}
\label{eq: Jac packet 0}
f^{H}(\p_{H}) = \sum_{[\r] \in \cPkt{\p}} <s, \r> f_{G}(\r) = \sum_{\e \in \D{\S{\p}}} \e(s) f_{G}(\r(\p, \e)), 
\end{align}
for $f \in \sH(G)$. In the notation of Section~\ref{sec: endoscopy}, we can write 
\[
H = G_{I} \times G_{II} \text{ and } \p_{H} = \p_{I} \times \p_{II}
\] 
Let us first assume $(\rho, 2x+1) \notin Jord(\p_{II})$. Then $(\rho', 2x+1) \in Jord(\p_{I})$ for $\rho' = \rho \otimes \eta_{I}$ if $G$ is symplectic, and $\rho' = \rho$ otherwise. By \eqref{eq: Jac vanishing}, 
\[
\bar{\Jac}_{x}\cPkt{\p_{II}} = 0.
\] 
So we let $H_{-} = H_{I -}$ (see Section~\ref{sec: compatibility of Jacquet modules with endoscopic transfer}), and define $\p_{H_{-}} = \p_{I-} \times \p_{II}$, where 
\[
Jord(\p_{I-}) = Jord(\p_{I}) \cup \{(\rho', 2x-1)\} \backslash \{(\rho', 2x+1)\}.
\] 
After applying \eqref{eq: compatible with endoscopic transfer 1}, \eqref{eq: compatible with endoscopic transfer 2} and \eqref{eq: compatible with endoscopic transfer 3} accordingly to \eqref{eq: Jac packet 0}, we get 

\begin{align}
\label{eq: Jac character relation}
f^{H_{-}}(\p_{H_{-}}) = \sum_{[\r] \in \cPkt{\p}} <s, \r> f_{G_{-}}(\bar{\Jac}_{x}\r) = \sum_{\bar{\e} \in \D{\S{\p}}} \bar{\e}(s) f_{G_{-}}(\bar{\Jac}_{x}\r(\p, \bar{\e})), 
\end{align}
for $f \in \sH(G_{-})$. On the other hand note $(H_{-}, \p_{H_{-}}) \rightarrow (\p_{-}, s_{-})$, where $s_{-}$ is the image of $s$ under the projection $\S{\p} \rightarrow \S{\p_{-}}$, so we have 
\[
f^{H_{-}}(\p_{H_{-}}) = \sum_{[\r_{-}] \in \cPkt{\p_{-}}} <s_{-}, \r_{-}> f_{G_{-}}(\r_{-}) = \sum_{\bar{\e}' \in \D{\S{\p}}_{-}} \bar{\e}'(s_{-}) f_{G_{-}}(\r(\p_{-}, \bar{\e}')),
\]
for $f \in \sH(G_{-})$. Combining this identity with \eqref{eq: Jac character relation}, we get

\begin{align}
\label{eq: Jac packet 1}
\sum_{\bar{\e} \in \D{\S{\p}}} \bar{\e}(s) f_{G_{-}}(\bar{\Jac}_{x}\r(\p, \bar{\e})) = \sum_{\bar{\e}' \in \D{\S{\p}}_{-}} \bar{\e}'(s_{-}) f_{G_{-}}(\r(\p_{-}, \bar{\e}')). 
\end{align}
In fact \eqref{eq: Jac packet 1} also holds when $(\rho, 2x+1) \in Jord(\p_{II})$, and the argument is similar. 

By the linear independence of characters of irreducible smooth representations of $G(F)$, $\r(\p_{-}, \bar{\e}')$ is in $\bar{\Jac}_{x}\r(\p, \bar{\e})$ only when 
\[
\bar{\e}(s) = \bar{\e}'(s_{-})
\] 
for all $s \in \S{\p}$, i.e., $\bar{\e}' = \bar{\e}$. This implies $\bar{\Jac}_{x}\r(\p, \bar{\e}) = 0$ for $\bar{\e} \notin \D{\S{\p}}_{-}$. Then after a little thought, one can see 
\[
\r(\p_{-}, \bar{\e}) = \bar{\Jac}_{x} \r(\p, \bar{\e})
\] 
for all $\bar{\e} \in \D{\S{\p}}_{-}$.

\end{proof}

Now we are in the position to prove Theorem~\ref{thm: supercuspidal parametrization}. For the convenience of readers we will restate the theorem here.

\begin{theorem}[M{\oe}glin]
The $\Sigma_{0}$-orbits of irreducible supercuspidal representations of $G(F)$ can be parametrized by $\p \in \cPdt{G}$ and $\bar{\e} \in \S{\p}$ satisfying the following properties:

\begin{enumerate}

\item if $(\rho, a) \in Jord(\p)$, then $(\rho, a-2) \in Jord(\p)$ as long as $a - 2 > 0$;

\item if $(\rho, a), (\rho, a-2) \in Jord(\p)$, then $\e(\rho, a) \e(\rho, a - 2) = -1$;

\item if $(\rho, 2) \in Jord(\p)$, then $\e(\rho, 2) = -1$.

\end{enumerate}

\end{theorem}

\begin{proof}

Let $\r$ be an irreducible discrete series representation of $G(F)$, and we can assume $[\r] = \r(\p, \bar{\e})$ for some $\p \in \cPdt{G}$ and $\bar{\e} \in \D{\S{\p}}$. It is not hard to see that $\r$ is supercuspidal if and only if $\bar{\Jac}_{x}\r(\p, \bar{\e}) = 0$ for any unitary irreducible supercuspidal representation $\rho$ of $GL(d_{\rho}, F)$ and any real number $x$. Then by \eqref{eq: Jac vanishing}, it is enough to consider the cases when $(\rho, 2x+1) \in Jord(\p)$. Note each of the conditions in this theorem excludes exactly one situation in Lemma~\ref{lemma: Jac packet} for $\bar{\Jac}_{x}\r(\p, \bar{\e}) \neq 0$. So it is clear that these conditions are both necessary and sufficient for $\r$ being supercuspidal.

\end{proof}

\begin{remark}
The necessity of condition (1) has already been established by Proposition~\ref{prop: Jord of supercuspidal}, but in this proof we do not need to know that result.
\end{remark}

\section{Cuspidal support of discrete series}
\label{sec: cuspidal support of discrete series}

In this section we are going to characterize the cuspidal supports of discrete series representations of 
$G(F)$. Let $\p \in \cPdt{G}$, for any $(\rho, a) \in Jord(\p)$, we denote by $a_{-}$ the biggest positive integer smaller than $a$ in $Jord_{\rho}(\p)$. And we would also like to write $a_{min}$ for the minimum of $Jord_{\rho}(\p)$. If $a = a_{min}$, we let $a_{-} = 0$ if $a$ is even, and $-1$ otherwise. In this case, we always assume $\e(\rho, a)\e(\rho, a_{-}) = -1$.

\begin{proposition}
\label{prop: parabolic reduction}
Suppose $\p \in \cPdt{G}$, and $\e \in \D{\S{\p}^{\Sigma_{0}}}$.

\begin{enumerate}

\item If $\e(\rho, a)\e(\rho, a_{-}) = -1$ and $a_{-} < a - 2$, then 
\begin{align}
\label{eq: parabolic reduction 1}
\r(\p, \bar{\e}) \hookrightarrow <(a-1)/2, \cdots, (a_{-} + 3)/2> \rtimes \r(\p', \bar{\e}')
\end{align}
is the unique irreducible element in $\bar{\Rep}(G)$ as an $\sH(G)$-submodule, where 
\[
Jord(\p') = Jord(\p) \cup \{(\rho, a_{-} + 2)\} \backslash \{(\rho, a)\},
\]
and 
\[
\e'(\cdot)= \e(\cdot) \text{ over } Jord(\p) \backslash \{(\rho, a)\}, \quad \quad \e'(\rho, a_{-}+2) = \e(\rho, a).
\]

\item If $\e(\rho, a)\e(\rho, a_{-}) = 1$, then 
\begin{align}
\label{eq: parabolic reduction 2}
\r(\p, \bar{\e}) \hookrightarrow <(a-1)/2, \cdots, -(a_{-} -1)/2> \rtimes \r(\p', \bar{\e}'),
\end{align}
where 
\[
Jord(\p') = Jord(\p) \backslash \{(\rho, a), (\rho, a_{-})\},
\]
and $\e'(\cdot)$ is the restriction of $\e(\cdot)$. In particular, suppose $\e_{1} \in \D{\S{\p}^{\Sigma_{0}}}$ satisfying $\e_{1}(\cdot) = \e(\cdot)$ over $Jord(\p')$ and
\[
\e_{1}(\rho, a) = -\e(\rho, a), \quad \quad \e_{1}(\rho, a_{-}) = -\e(\rho, a_{-}).
\]
If $\bar{\e}_{1} = \bar{\e}$, then the induced $\sH(G)$-module in \eqref{eq: parabolic reduction 2} has a unique irreducible element in $\bar{\Rep}(G)$ as an $\sH(G)$-submodule. Otherwise, it has two irreducible elements in $\bar{\Rep}(G)$ as $\sH(G)$-submodules, namely
\[
\r(\p, \bar{\e}) \+ \r(\p, \bar{\e}_{1}).
\]

\item If $\e(\rho, a_{min}) = 1$ and $a_{min}$ is even, then 
\begin{align}
\label{eq: parabolic reduction 3}
\r(\p, \bar{\e}) \hookrightarrow <(a_{min}-1)/2, \cdots, 1/2> \rtimes \r(\p', \bar{\e}')
\end{align}
is the unique irreducible element in $\bar{\Rep}(G)$ as an $\sH(G)$-submodule, where 
\[
Jord(\p') = Jord(\p) \backslash \{(\rho, a_{min})\},
\] 
and $\e'(\cdot)$ is the restriction of $\e(\cdot)$.

\end{enumerate}

\end{proposition}

\begin{proof}
The proofs of part (1) and part (3) are almost the same, so here we will only give the proof of part (1). We start by considering the Jacquet module $\bar{\Jac}_{(a-1)/2, \cdots, (a_{-} + 3)/2} \r(\p, \bar{\e})$, and by applying Lemma~\ref{lemma: Jac packet} multiple times we have 
\[
\bar{\Jac}_{(a-1)/2, \cdots, (a_{-} + 3)/2} \r(\p, \bar\e) = \r(\p', \bar{\e}').
\]
It follows from Corollary~\ref{cor: existence of inclusion} that
\[
\r(\p, \bar{\e}) \hookrightarrow \rho||^{\frac{a-1}{2}} \times \cdots \times \rho||^{\frac{a_{-} + 3}{2}} \rtimes \r(\p', \bar{\e}').
\]
By Lemma~\ref{lemma: B}, we can take an irreducible constituent $\tau$ in $\rho||^{\frac{a-1}{2}} \times \cdots \times \rho||^{\frac{a_{-} + 3}{2}}$, such that 
\[
\r(\p, \bar{\e}) \hookrightarrow \tau \rtimes \r(\p', \bar{\e}').
\]
So it is enough to show $\tau = <\frac{a-1}{2}, \cdots, \frac{a_{-} + 3}{2}>$. If this is not the case, we know from Lemma~\ref{lemma: subrep GL(n)} that $\Jac_{x}\tau \neq 0$ for some $(a_{-} + 3)/2 \leqslant x < (a-1)/2$. So $\tau \hookrightarrow \rho||^{x} \times \tau'$ for some irreducible representation $\tau'$, and 
\[
\r(\p, \bar{\e}) \hookrightarrow \rho||^{x} \times \tau' \rtimes \r(\p', \bar{\e}').
\]
By Frobenius reciprocity, $\bar{\Jac}_{x}\r(\p, \bar{\e}) \neq 0$. However, $(\rho, 2x+1) \notin Jord(\p)$ under our assumption, so we get a contradiction (see \eqref{eq: Jac vanishing}). 

To see the induced $\sH(G)$-module in \eqref{eq: parabolic reduction 1} has a unique irreducible element in $\bar{\Rep}(G)$ as an $\sH(G)$-submodule, we can compute its Jacquet module under $\bar{\Jac}_{(a-1)/2, \cdots, (a_{-} + 3)/2}$. By applying the formula~\eqref{eq: Jac semiproduct}, we find the Jacquet module consists of 
\[
\Jac_{X_{1}} <(a-1)/2, \cdots, (a_{-} + 3)/2> \times \Jac^{op}_{-X_{2}} <(a-1)/2, \cdots, (a_{-} + 3)/2> \rtimes \bar{\Jac}_{X_{3}} \r(\p', \bar{\e}'),
\]
where 
\[
\{ (a-1)/2, \cdots, (a_{-} + 3)/2 \} = X_{1} \sqcup X_{2} \sqcup X_{3},
\]
and $X_{i}$ inherits the order from $\{ (a-1)/2, \cdots, (a_{-} + 3)/2 \}$. Note $\Jac_{X_{1}} <(a-1)/2, \cdots, (a_{-} + 3)/2> \neq 0$ only if $X_{1}$ is a segment $\{(a-1)/2, \cdots, x_{1}\}$. Similarly, $\Jac^{op}_{-X_{2}} <(a-1)/2, \cdots, (a_{-} + 3)/2> \neq 0$ only if $X_{2}$ is a segment $\{-(a_{-} + 3)/2, \cdots, x_{2}\}$. Since $-(a_{-} + 3)/2 \notin \{ (a-1)/2, \cdots, (a_{-} + 3)/2 \}$, $X_{2}$ has to be empty. Therefore the Jacquet module can only contain terms like
\[
\Jac_{(a-1)/2, \cdots, x_{1}} <(a-1)/2, \cdots, (a_{-} + 3)/2> \rtimes \bar{\Jac}_{x_{1} - 1, \cdots, (a_{-} + 3)/2} \r(\p', \bar{\e}').
\]
But from our definition of $Jord(\p')$, we see $\{a, \cdots, (a_{-}+4)\}$ has no intersection with $Jord_{\rho}(\p')$, so 
\[
\bar{\Jac}_{x_{1} - 1, \cdots, (a_{-} + 3)/2} \r(\p', \bar{\e}') = 0
\] 
by \eqref{eq: Jac vanishing}. Hence we can only have
\[
\bar{\Jac}_{(a-1)/2, \cdots, (a_{-} + 3)/2} \Big( <(a-1)/2, \cdots, (a_{-} + 3)/2> \rtimes \r(\p', \bar{\e}') \Big) = \r(\p', \bar{\e}').
\]
Note this implies $<(a-1)/2, \cdots, (a_{-} + 3)/2> \rtimes \r(\p', \bar{\e}')$ has a unique irreducible element in $\bar{\Rep}(G)$ as an $\sH(G)$-submodule. 
 
For part (2), we will first consider $\bar{\Jac}_{(a-1)/2, \cdots, (a_{-} + 1)/2} \r(\p, \bar{\e})$, and again by applying Lemma~\ref{lemma: Jac packet} multiple times we have 
\[
\bar{\Jac}_{(a-1)/2, \cdots, (a_{-} + 1)/2} \r(\p, \bar{\e}) = \r(\p_{-}, \bar{\e}_{-}),
\]
where $Jord(\p_{-}) = Jord(\p) \cup \{(\rho, a_{-})\} \backslash \{(\rho, a)\}$, and $\e_{-}(\cdot)$ is the restriction of $\e(\cdot)$ to $Jord(\p_{-})$ (forgetting multiplicities). As in part (1), we can show from here that

\begin{align}
\label{eq: parabolic reduction 2-1}
\r(\p, \bar{\e}) \hookrightarrow <(a-1)/2, \cdots, (a_{-} + 1)/2> \rtimes \r(\p_{-}, \bar{\e}_{-}).
\end{align}
Note $\cPkt{\p_{-}} = St(\rho, a_{-}) \rtimes \cPkt{\p'}$ (see \eqref{eq: tempered L-packet}), so
\[
\r(\p_{-}, \bar{\e}_{-}) \hookrightarrow St(\rho, a_{-}) \rtimes \r(\p', \bar{\e}') = <(a_{-} -1)2, \cdots, -(a_{-} -1)/2> \rtimes \r(\p', \bar{\e}'),
\]
and hence

\begin{align}
\label{eq: parabolic reduction 2-2}
\r(\p, \bar{\e}) \hookrightarrow <(a-1)/2, \cdots, (a_{-} + 1)/2> \times <(a_{-} -1)/2, \cdots, -(a_{-} -1)/2> \rtimes \r(\p', \bar{\e}').
\end{align}
By Lemma~\ref{lemma: B}, we can take an irreducible constituent $\tau$ in 
\[
<(a -1)/2, \cdots, (a_{-} +1)/2> \times <(a_{-} -1)/2, \cdots, -(a_{-} -1)/2> ,
\]
such that 
\[
\r(\p, \bar{\e}) \hookrightarrow \tau \rtimes \r(\p', \bar{\e}').
\]
Therefore it suffices to show $\tau = <(a -1)/2, \cdots, -(a_{-} -1)/2>$. If this is not the case, then by Theorem~\ref{thm: irreducibility for induced representation of GL(n)}
\[
\tau \hookrightarrow <(a_{-} -1)/2, \cdots, -(a_{-} -1)/2> \times <(a -1)/2, \cdots, (a_{-} +1)/2>.
\]
And by Frobenius reciprocity, we have 
\[
\bar{\Jac}_{(a_{-} -1)/2, \cdots, -(a_{-} -1)/2} \r \neq 0.
\]
But this is impossible, because one can check 
\[
\Jac^{\theta}_{(a_{-} -1)/2, \cdots, -(a_{-} -1)/2} \r_{\p} = 0.
\]
At last, we still need to show the irreducible elements in $\bar{\Rep}(G)$ as submdules of the induced $\sH(G)$-module in \eqref{eq: parabolic reduction 2} are either $\r(\p, \bar{\e})$ or $\r(\p, \bar{\e}) \+ \r(\p, \bar{\e}_{1})$ depending on whether $\bar{\e}$ and $\bar{\e}_{1}$ are equal or not. Note we can show in the same way as in part (1) that $\r(\p, \bar{\e})$ is the unique irreducible element in $\bar{\Rep}(G)$ as an submodule of the induced $\sH(G)$-module in \eqref{eq: parabolic reduction 2-1}. And the same is true for $\r(\p, \bar{\e}_{1})$. Since $\bar{\e} = \bar{\e}_{1}$ if and only if $\bar{\e}_{-} = \bar{\e}_{1, -}$, where $\e_{1, -}(\cdot)$ is again the restriction of $\e_{1}(\cdot)$ to $Jord(\p_{-})$ (forgetting multiplicities), let us assume $\bar{\e} \neq \bar{\e}_{1}$ first. Then by \eqref{eq: elliptic representaion}
\[
\r(\p_{-}, \bar{\e}_{-}) \+ \r(\p_{-}, \bar{\e}_{1, -}) = St(\rho, a_{-}) \rtimes \r(\p', \bar{\e}'), 
\]
and hence the irreducible elements in $\bar{\Rep}(G)$ as submodules of the induced $\sH(G)$-module in \eqref{eq: parabolic reduction 2-2} are exactly $\r(\p, \bar{\e}) \+ \r(\p, \bar{\e}_{1})$. So we only need to show the induced $\sH(G)$-modules in \eqref{eq: parabolic reduction 2} and \eqref{eq: parabolic reduction 2-2} have the same irreducible elements in $\bar{\Rep}(G)$ as submodules. One direction is clear, i.e., the irreducible elements in $\bar{\Rep}(G)$ as submodules of 
\[
<(a-1)/2, \cdots, (a_{-} + 1)/2> \times <(a_{-} -1)/2, \cdots, -(a_{-} -1)/2> \rtimes \r(\p', \bar{\e}')
\]
contain that of 
\[
<(a-1)/2, \cdots, -(a_{-} -1)/2> \rtimes \r(\p', \bar{\e}')
\]
And from what we have shown, it is clear that $\r(\p, \bar{\e}) \+ \r(\p, \bar{\e}_{1})$ are in $<(a-1)/2, \cdots, -(a_{-} -1)/2> \rtimes \r(\p', \bar{\e}')$, so they have to contain the same irreducible elements in $\bar{\Rep}(G)$ as $\sH(G)$-submodules. Now if $\bar{\e} = \bar{\e}_{1}$, we have by \eqref{eq: elliptic representaion}
\[
\r(\p_{-}, \bar{\e}_{-}) = St(\rho, a_{-}) \rtimes \r(\p', \bar{\e}'),
\] 
and the rest of the argument is the same.

\end{proof}

\section{Remarks on even orthogonal groups}
\label{sec: even orthogonal group}

The previous results of this paper can also be extended to representations of $G^{\Sigma_{0}}(F)$. Note the only nontrivial case here is when $G$ is special even orthogonal. First, we will extend Proposition~\ref{prop: cuspidal reducibility}.

\begin{corollary}
\label{cor: cuspidal reducibility}
Suppose $\r$ is a supercuspidal representation of $G(F)$ and $[\r] \in \cPkt{\p}$ for some $\p \in \cPdt{G}$. Let $\r^{\Sigma_{0}}$ be any irreducible representation of $G^{\Sigma_{0}}(F)$, whose restriction to $G$ contains $\r$. Then for any unitary irreducible supercuspidal representation $\rho$ of $GL(d_{\rho}, F)$ and real number $a_{\rho}$, the parabolic induction 
\[
\rho||^{\pm(a_{\rho} +1)/2} \rtimes \r^{\Sigma_{0}}
\]
reduces if and only if $\rho$ is self-dual and 
\begin{align}
\label{eq: cuspidal reducibility full orthogonal group}
a_{\rho} = \begin{cases}
                 \text{ max } Jord_{\rho}(\p), & \text{ if } Jord_{\rho}(\p) \neq \emptyset, \\
                                    0, & \text{ if $Jord_{\rho}(\p) = \emptyset$, $\rho$ is of opposite type to $\D{G}$}, \\
                                    -1, & \text{ otherwise. }
                 \end{cases}
\end{align}

\end{corollary}

\begin{proof}
We can assume $G$ is special even orthogonal. First we would like to give the relation of irreducibility between an irreducible representation $\r$ of $G(F)$ and an irreducible representation $\r^{\Sigma_{0}}$ of $G^{\Sigma_{0}}(F)$ which contains $\r$ in its restriction to $G(F)$. For any irreducible representation $\tau$ of $GL(d, F)$, it is easy to show the following fact:

\begin{itemize}

\item If $\r \ncong \r^{\theta_{0}}$, $\tau \rtimes \r^{\Sigma_{0}}$ is irreducible if and only if $\tau \rtimes \r$ is irreducible and $(\tau \rtimes \r)^{\theta_{0}} \ncong \tau \rtimes \r$.

\item If $\r \cong \r^{\theta_{0}}$, $\tau \rtimes \r$ is irreducible if and only if $\tau \rtimes \r^{\Sigma_{0}}$ is irreducible and $\tau \rtimes \r^{\Sigma_{0}} \ncong (\tau \rtimes \r^{\Sigma_{0}}) \otimes \x_{0}$.

\end{itemize}

Let $\tau = \rho||^{(a_{\rho} +1)/2}$ and $\r$ be supercuspidal. We assume $\tau \rtimes \r$ is a representation of $G_{+}(F)$. Note the condition \eqref{eq: cuspidal reducibility} implies \eqref{eq: cuspidal reducibility full orthogonal group}. To see the necessity of the condition \eqref{eq: cuspidal reducibility full orthogonal group}, we need to show if it is not satisfied, then $\tau \rtimes \r^{\Sigma_{0}}$ is irreducible. Since $\tau \rtimes \r$ is irreducible in this case, it suffices to consider $\r \ncong \r^{\theta_{0}}$, and we would like to show $(\tau \rtimes \r)^{\theta_{0}} \ncong \tau \rtimes \r$. Since $\tau$ and $\r$ are both supercuspidal, this is also equivalent to show there does not exist a Weyl group element of $G(F)$ sending $\tau \rtimes \r^{\theta_{0}}$ to $\tau \rtimes \r$, i.e., $\tau \ncong \tau^{\vee}$ or $d$ is even. Suppose $\tau \cong \tau^{\vee}$ and $d$ is odd, then $a_{\rho} = -1$ and $\rho$ is necessarily of orthogonal type, hence one can only have $Jord_{\rho}(\p) \neq \emptyset$ in view of \eqref{eq: cuspidal reducibility full orthogonal group}. This implies $\S{\p}^{\Sigma_{0}} \neq \S{\p}$, so $\r \cong \r^{\theta_{0}}$ (see the remarks after Theorem~\ref{thm: L-packet full orthogonal group}) and we get a contradiction. 

To see the reducibility condition \eqref{eq: cuspidal reducibility full orthogonal group} is also sufficient, we first consider the case $\r \cong \r^{\theta_{0}}$, then the condition \eqref{eq: cuspidal reducibility full orthogonal group} becomes the same as \eqref{eq: cuspidal reducibility}. If \eqref{eq: cuspidal reducibility full orthogonal group} is satisfied, then $\tau \rtimes \r$ reduces. 
Suppose $\tau \rtimes \r^{\Sigma_{0}}$ is irreducible, then $\tau \rtimes \r^{\Sigma_{0}} \cong (\tau \rtimes \r^{\Sigma_{0}}) \otimes \x_{0}$, and hence 
\[
(\tau \rtimes \r^{\Sigma_{0}})|_{G_{+}} \cong \tau \rtimes (\r^{\Sigma_{0}}|_{G}) \cong \tau \rtimes \r \cong \r_{+} \+ \r_{+}^{\theta_{0}},
\] 
where $\r_{+} \ncong \r_{+}^{\theta_{0}}$. By the theory of Langlands quotient, one must have $\tau \cong \rho$. Define $\p_{+}$ by 
\[
Jord(\p_{+}) := Jord(\p) \cup \{(\rho, 1) \text{ with multiplicity $2$ }\}.
\] 
Then $[\rho \rtimes \r] \subseteq \cPkt{\p_{+}}$. Since $\r \cong \r^{\theta_{0}}$, we have $\S{\p}^{\Sigma_{0}} \neq \S{\p}$ , and it follows $\S{\p_{+}}^{\Sigma_{0}} \neq \S{\p_{+}}$. So $\r_{+}^{\theta_{0}} \cong \r_{+}$. This is a contradiction.

At last, we can assume $\r \ncong \r^{\theta_{0}}$, and it suffices for us to show if $\tau \rtimes \r^{\Sigma_{0}}$ is irreducible, then \eqref{eq: cuspidal reducibility full orthogonal group} is not satisfied. In this case $\tau \rtimes \r$ is irreducible and $(\tau \rtimes \r)^{\theta_{0}} \ncong \tau \rtimes \r$. In particular, \eqref{eq: cuspidal reducibility} is not satisfied. So we only need to exclude the case that $\rho$ is of orthogonal type, $Jord_{\rho}(\p) = \emptyset$, $a_{\rho} = -1$ and $d$ is odd. In this case $[\tau \rtimes \r] = [\rho \rtimes \r] \in \cPkt{\p_{+}}$ and $\S{\p_{+}}^{\Sigma_{0}} \neq \S{\p_{+}}$, so $(\tau \rtimes \r)^{\theta_{0}} \cong \tau \rtimes \r$, which again leads to a contradiction. This finishes the proof.

\end{proof}

Next, we would like to extend Lemma~\ref{lemma: Jac packet}.

\begin{lemma}
\label{lemma: Jac packet full orthogonal group}
Suppose $\p \in \cPdt{G}$, and $(\rho, 2x+1) \in Jord(\p)$ with $x > 0$. Let $\p_{-} \in \cPbd{G_{-}}$ such that 
\[
Jord(\p_{-}) = Jord(\p) \cup \{(\rho, 2x-1)\} \backslash \{(\rho, 2x+1)\}.
\]
Then we have the following facts:
\begin{enumerate}

\item If $x > 1/2$ and $(\rho, 2x-1) \notin Jord(\p)$, then $\r^{\Sigma_{0}}(\p_{-}, \e) = \Jac_{x} \r^{\Sigma_{0}}(\p, \e)$ for all $\e \in \D{\S{\p_{-}}^{\Sigma_{0}}} \cong \D{\S{\p}^{\Sigma_{0}}}$.

\item If $x > 1/2$ and $(\rho, 2x-1) \in Jord(\p)$, then $\Jac_{x} \r^{\Sigma_{0}}(\p, \e) = 0$ unless $\e \in \D{\S{\p_{-}}^{\Sigma_{0}}}$, i.e., 
\[
\e(\rho, 2x+1) \e(\rho, 2x-1) = 1,
\]
in which case $\r^{\Sigma_{0}}(\p_{-}, \e) = \Jac_{x} \r^{\Sigma_{0}}(\p, \e)$.   

\item If $x = 1/2$, then $\Jac_{1/2} \r^{\Sigma_{0}}(\p, \e) = 0$ unless $\e \in \D{\S{\p_{-}}^{\Sigma_{0}}}$, i.e.,
\[
\e(\rho, 2) = 1,
\]
in which case $\r^{\Sigma_{0}}(\p_{-}, \e) = \Jac_{x} \r^{\Sigma_{0}}(\p, \e)$.

\end{enumerate}

\end{lemma}

\begin{proof}
Let $G = G(n)$ be special even orthogonal. If $n = d_{\rho}$, then it suffices to assume $Jord(\p) = \{(\rho, 2)\}$. In this case, we necessarily have $\e = 1$ and
\[
[(\Jac_{1/2} \r^{\Sigma_{0}}(\p, \e))|_{G}] = [\Jac_{1/2} (\r^{\Sigma_{0}}(\p, \e) |_{G})] =  \bar{\Jac}_{1/2} (\r(\p, \bar{\e})) = 1.
\]
So now we can assume $n \neq d_{\rho}$. We claim $\Jac_{x} \r^{\Sigma_{0}}(\p, \e) = 0$ unless $\e \in \D{\S{\p_{-}}^{\Sigma_{0}}}$, in which case,
\[
\Jac_{x} \r^{\Sigma_{0}}(\p, \e) = \r^{\Sigma_{0}}(\p_{-}, \e) \text{ or } \r^{\Sigma_{0}}(\p_{-}, \e \e_{0}). 
\]
Suppose $\S{\p}^{\Sigma_{0}} \neq \S{\p}$, then 
\[
[(\Jac_{x} \r^{\Sigma_{0}}(\p, \e))|_{G}] = [\Jac_{x} (\r^{\Sigma_{0}}(\p, \e) |_{G})] =  \bar{\Jac}_{x} \r(\p, \bar{\e}) = 0 \text{ or } \r(\p_{-} ,\bar{\e}),
\]
and it is nonzero only when $\e \in \D{\S{\p_{-}}^{\Sigma_{0}}}$. 
Suppose $\S{\p}^{\Sigma_{0}} = \S{\p}$, then
\[
[(\Jac_{x} \r^{\Sigma_{0}}(\p, \e))|_{G}] = [\Jac_{x} (\r^{\Sigma_{0}}(\p, \e) |_{G})] =  2\bar{\Jac}_{x} (\r(\p, \bar{\e})) = 0 \text{ or } 2\r(\p_{-} ,\bar{\e}),
\]
and it is again nonzero only when $\e \in \D{\S{\p_{-}}^{\Sigma_{0}}}$. So the claim is clear and it also suffices to show the lemma when $\S{\p}^{\Sigma_{0}} \neq \S{\p}$, i.e., $\e_{0} \neq 1$. Let us choose $s^{*} \in \S{\p}^{\Sigma_{0}}$ such that $\e_{0}(s^{*}) = -1$. Then $s^{*} \notin \S{\p}$. Suppose $(H, \p_{H}) \rightarrow (\p, s^{*})$, then we have from \eqref{eq: character relation full orthogonal group}

\begin{align}
\label{eq: Jac packet full orthogonal group}
f^{H}(\p_{H}) = \sum_{[\r] \in \cPkt{\p}} <s^{*}, \r^{\Sigma_{0}}> f_{G}(\r^{\Sigma_{0}}) = \sum_{\bar{\e} \in \D{\S{\p}}} \e(s^{*}) f_{G}(\r^{\Sigma_{0}}(\p, \e)), 
\end{align}
for $f \in C^{\infty}_{c}(G(F) \rtimes \theta_{0})$. In the notation of Section~\ref{sec: endoscopy}, we can write 
\[
H = G_{I} \times G_{II} \text{ and } \p_{H} = \p_{I} \times \p_{II}
\] 
Without loss of generality we can assume $(\rho \otimes \eta_{II}, 2x+1) \notin Jord(\p_{II})$ and $(\rho \otimes \eta_{I}, 2x+1) \in Jord(\p_{I})$. Then by \eqref{eq: Jac vanishing}, 
\[
\bar{\Jac}^{II}_{x}\cPkt{\p_{II}} = 0.
\] 
We let $H_{-} = H_{I -}$ (see Section~\ref{sec: compatibility of Jacquet modules with endoscopic transfer}), and define $\p_{H_{-}} = \p_{I-} \times \p_{II}$, where 
\[
Jord(\p_{I-}) = Jord(\p_{I}) \cup \{(\rho \otimes \eta_{I}, 2x-1)\} \backslash \{(\rho \otimes \eta_{I}, 2x+1)\}.
\] 
So after applying \eqref{eq: compatible with endoscopic transfer full orthogonal group} to \eqref{eq: Jac packet full orthogonal group}, we get 

\begin{align}
\label{eq: Jac character relation full orthogonal group}
f^{H_{-}}(\p_{H_{-}}) 
= \sum_{\bar{\e} \in \D{\S{\p}}} \e(s^{*}) f_{G_{-}}(\Jac_{x}\r^{\Sigma_{0}}(\p, \e)) = \sum_{\bar{\e} \in \D{\S{\p}}_{-}} \e(s^{*}_{-}) f_{G_{-}}(\r^{\Sigma_{0}}(\p_{-}, \e')), 
\end{align}
for $f \in  C^{\infty}_{c}(G_{-}(F) \rtimes \theta_{0})$, where $\e' = \e$ or $\e' = \e \e_{0}$. Since $(H_{-}, \p_{H_{-}}) \rightarrow (\p_{-}, s^{*}_{-})$, where $s^{*}_{-}$ is the image of $s^{*}$ under the projection $\S{\p}^{\Sigma_{0}} \rightarrow \S{\p_{-}}^{\Sigma_{0}}$ and $s^{*}_{-} \notin \S{\p_{-}}$, we also have 
\[
f^{H_{-}}(\p_{H_{-}}) 
= \sum_{\bar{\e} \in \D{\S{\p}}_{-}} \e(s^{*}_{-}) f_{G_{-}}(\r^{\Sigma_{0}}(\p_{-}, \e)),
\]
for $f \in C^{\infty}_{c}(G_{-}(F) \rtimes \theta_{0})$. Combining this identity with \eqref{eq: Jac character relation full orthogonal group}, we get

\begin{align}
\label{eq: Jac packet full orthogonal group}
\sum_{\bar{\e} \in \D{\S{\p}}_{-}} \e(s^{*}_{-}) f_{G_{-}}(\r^{\Sigma_{0}}(\p_{-}, \e')) = \sum_{\bar{\e} \in \D{\S{\p}}_{-}} \e(s^{*}_{-}) f_{G_{-}}(\r^{\Sigma_{0}}(\p_{-}, \e)). 
\end{align}
By the linear independence of twisted characters of irreducible smooth representations of $G(F)$, we have 
\[
\e(s^{*}_{-}) f_{G_{-}}(\r^{\Sigma_{0}}(\p_{-}, \e')) =  \e(s^{*}_{-}) f_{G_{-}}(\r^{\Sigma_{0}}(\p_{-}, \e))
\] 
and hence 
\[
f_{G_{-}}(\r^{\Sigma_{0}}(\p_{-}, \e')) = f_{G_{-}}(\r^{\Sigma_{0}}(\p_{-}, \e)).
\]
This implies $\r^{\Sigma_{0}}(\p_{-}, \e') = \r^{\Sigma_{0}}(\p_{-}, \e)$, so $\e = \e'$.


\end{proof}

As a consequence of this lemma, we can extend Proposition~\ref{prop: parabolic reduction}. We will follow the same setup in the beginning of Section~\ref{sec: cuspidal support of discrete series}.

\begin{proposition}
\label{prop: parabolic reduction full orthogonal group}
Suppose $\p \in \cPdt{G}$, and $\e \in \D{\S{\p}^{\Sigma_{0}}}$.

\begin{enumerate}

\item If $\e(\rho, a)\e(\rho, a_{-}) = -1$ and $a_{-} < a - 2$, then 
\begin{align}
\label{eq: parabolic reduction 1 full orthogonal group}
\r^{\Sigma_{0}}(\p, \e) \hookrightarrow <(a-1)/2, \cdots, (a_{-} + 3)/2> \rtimes \r^{\Sigma_{0}}(\p', \e')
\end{align}
as the unique irreducible subrepresentation, where 
\[
Jord(\p') = Jord(\p) \cup \{(\rho, a_{-} + 2)\} \backslash \{(\rho, a)\},
\]
and 
\[
\e'(\cdot)= \e(\cdot) \text{ over } Jord(\p) \backslash \{(\rho, a)\}, \quad \quad \e'(\rho, a_{-}+2) = \e(\rho, a).
\]

\item If $\e(\rho, a)\e(\rho, a_{-}) = 1$, then 
\begin{align}
\label{eq: parabolic reduction 2 full orthogonal group}
\r^{\Sigma_{0}}(\p, \e) \hookrightarrow <(a-1)/2, \cdots, -(a_{-} -1)/2> \rtimes \r^{\Sigma_{0}}(\p', \e'),
\end{align}
where 
\[
Jord(\p') = Jord(\p) \backslash \{(\rho, a), (\rho, a_{-})\},
\]
and $\e'(\cdot)$ is the restriction of $\e(\cdot)$. In particular, suppose $\e_{1} \in \D{\S{\p}^{\Sigma_{0}}}$ satisfying $\e_{1}(\cdot) = \e(\cdot)$ over $Jord(\p')$ and
\[
\e_{1}(\rho, a) = -\e(\rho, a), \quad \quad \e_{1}(\rho, a_{-}) = -\e(\rho, a_{-}).
\]
Then the induced representation in \eqref{eq: parabolic reduction 2 full orthogonal group} has two irreducible subrepresentations, namely
\[
\r^{\Sigma_{0}}(\p, \e) \+ \r^{\Sigma_{0}}(\p, \e_{1}).
\]

\item If $\e(\rho, a_{min}) = 1$ and $a_{min}$ is even, then 
\begin{align}
\label{eq: parabolic reduction 3 full orthogonal group}
\r^{\Sigma_{0}}(\p, \e) \hookrightarrow <(a_{min}-1)/2, \cdots, 1/2> \rtimes \r^{\Sigma_{0}}(\p', \e')
\end{align}
as the unique irreducible subrepresentation, where 
\[
Jord(\p') = Jord(\p) \backslash \{(\rho, a_{min})\},
\] 
and $\e'(\cdot)$ is the restriction of $\e(\cdot)$.

\end{enumerate}

\end{proposition}

The proof of this proposition is almost the same as Proposition~\ref{prop: parabolic reduction}, so we omit it here.


\section{Classification of discrete series}
\label{sec: classification of discrete series}

Now we want to characterize the irreducible discrete series representations of $G^{\Sigma_{0}}(F)$ in terms of their cuspidal supports. For any irreducible discrete series representation $\r^{\Sigma_{0}}(\p, \e)$ of $G^{\Sigma_{0}}(F)$, we can associate a triple $(Jord, \r_{cusp}^{\Sigma_{0}}, \Delta)$. Here $Jord = Jord(\p)$ and $\r_{cusp}^{\Sigma_{0}}$ is a supercuspidal representation of $G_{-}^{\Sigma_{0}}(F)$ which is part of the cuspidal support of $\r^{\Sigma_{0}}$. Let us assume $\r_{cusp}^{\Sigma_{0}} = \r^{\Sigma_{0}}(\p_{cusp}, \e_{cusp})$. Finally, $\Delta$ is a $\Two$-valued function defined on a subset of
\[
Jord \sqcup (Jord \times Jord),
\]
i.e., $\Delta$ is not defined on $(\rho, a) \in Jord$ with $a$ being odd and $Jord_{\rho}(\p_{cusp}) \neq \emptyset$; $\Delta$ is not defined on pairs $(\rho, a), (\rho', a') \in Jord$ with $\rho \neq \rho'$. Moreover, we require $\Delta$ to satisfy the following properties:
\begin{enumerate}

\item $\Delta(\rho, a) \Delta(\rho, a')^{-1} = \Delta(\rho, a; \rho, a')$,

\item $\Delta(\rho, a; \rho, a') \Delta(\rho, a'; \rho, a'') = \Delta(\rho, a; \rho, a'')$,

\item $\Delta(\rho, a; \rho, a') = \Delta(\rho, a'; \rho, a)$.

\end{enumerate}
In our case, we can define 
\[
\Delta(\rho, a) = \e(\rho, a)
\]
for $(\rho, a) \in Jord$ with $a$ being even or $Jord_{\rho}(\p_{cusp}) = \emptyset$; and 
\[
\Delta(\rho, a; \rho', a') = \e(\rho, a) \e(\rho', a')^{-1}
\] 
for $(\rho, a), (\rho', a') \in Jord$ with $\rho = \rho'$; otherwise $\Delta$ is not defined. 

In view of Theorem~\ref{thm: supercuspidal parametrization} and Proposition~\ref{prop: parabolic reduction full orthogonal group}, $(\p_{cusp}, \e_{cusp})$ can be constructed from $(\p, \e)$ as follows. First we take a maximal sequence of parameters $\p_{i}$ for $1 \leqslant i \leqslant k$ such that $\p_{1} = \p$ and $\p_{i+1}$ is obtained from $\p_{i}$ by removing $(\rho, a)$ and $(\rho, a_{-})$, where $a_{-}$ is the biggest positive integer smaller than $a$ in $Jord_{\rho}(\p_{i})$ and $\e(\rho ,a) = \e(\rho, a_{-})$. Secondly, we remove all $(\rho, a) \in Jord(\p_{k})$, where $a = min \, Jord_{\rho}(\p_{k})$ is even and $\e(\rho, a) = 1$. We denote the resulting parameter by $\p_{k+1}$ and index $Jord_{\rho}(\p_{k+1}) = \{a_{j}\}$ for $j \geqslant 1$ such that $a_{j+1} > a_{j}$. Then we can identify $Jord(\p_{k+1})$ with $Jord(\p_{cusp})$ by sending $(\rho, a_{j})$ to $(\rho, 2j - 1)$ if $a_{j}$ is odd, or $(\rho, 2j)$ if $a_{j}$ is even. Let $\e_{cusp}$ be the restriction of $\e$.


In general, we can consider all triples $(Jord, \r_{cusp}^{\Sigma_{0}}, \Delta)$ such that $Jord = Jord(\p)$ for some $\p \in \cPdt{G}$, $\r_{cusp}^{\Sigma_{0}}$ is some supercuspidal representation of $G_{-}^{\Sigma_{0}}(F)$ which is of the same type as $G^{\Sigma_{0}}(F)$, and $\Delta$ satisfies the property that we have mentioned above. Let $Jord_{\rho} = Jord_{\rho}(\p)$. Next we will introduce the concept of {\bf admissibility} for such pairs. Let 
\[
Jord^{+}_{\rho}(\p_{cusp}) = \begin{cases}
                                              Jord_{\rho}(\p_{cusp}) \cup \{0\}, & \text{ if $a_{min} = min \, Jord_{\rho}$ is even and $\Delta(\rho, a_{min}) = 1$}, \\
                                              Jord_{\rho}(\p_{cusp}), & \text{ otherwise }.
                                              \end{cases}
\]
Then $(Jord, \r_{cusp}^{\Sigma_{0}}, \Delta)$ is called an {\bf admissible triple of alternated type} if 
\begin{enumerate}

\item $\Delta(\rho, a; \rho, a_{-}) = -1$, if $a_{-}$ is the biggest positive integer smaller than $a$ in $Jord_{\rho}$.

\item $|Jord^{+}_{\rho}(\p_{cusp})| = |Jord_{\rho}|$.

\end{enumerate}
We say $(Jord', \r_{cusp}^{\Sigma_{0}}, \Delta')$ is subordinated to $(Jord, \r_{cusp}^{\Sigma_{0}}, \Delta)$ if $Jord_{\rho}' = Jord_{\rho} \backslash \{a, a_{-}\}$, where $\Delta(\rho, a; \rho, a_{-}) = 1$, and $\Delta'$ is the restriction of $\Delta$. Then $(Jord, \r_{cusp}^{\Sigma_{0}}, \Delta)$ is called an {\bf admissible triple} if there exists a sequence of triples $(Jord_{i}, \r_{cusp}^{\Sigma_{0}}, \Delta_{i})$ for $1 \leqslant i \leqslant k$ such that 
\begin{enumerate}

\item $(Jord, \r_{cusp}^{\Sigma_{0}}, \Delta) = (Jord_{1}, \r_{cusp}^{\Sigma_{0}}, \Delta_{1})$,

\item $(Jord_{i+1}, \r_{cusp}^{\Sigma_{0}}, \Delta_{i+1})$ is subordinated to $(Jord_{i}, \r_{cusp}^{\Sigma_{0}}, \Delta_{i})$ for $1 \leqslant i \leqslant k-1$,

\item $(Jord_{k}, \r_{cusp}^{\Sigma_{0}}, \Delta_{k})$ is an admissible triple of alternated type.

\end{enumerate}
Comparing this definition with our construction of $(\p_{cusp}, \e_{cusp})$ from $(\p, \e)$, it is easy to see that the triples we associate with irreducible discrete series representations are admissible. On the other hand, from any admissible triple $(Jord, \r_{cusp}^{\Sigma_{0}}, \Delta)$ with $Jord = Jord(\p)$ for some $\p \in \cPdt{G}$ and $\r_{cusp}^{\Sigma_{0}} = \r^{\Sigma_{0}}(\p_{cusp}, \e_{cusp})$, we can always extend $\e_{cusp}(\cdot)$ in a unique way to $\e(\cdot) \in \D{\S{\p}^{\Sigma_{0}}}$ such that the triple is associated with $\r^{\Sigma_{0}}(\p, \e)$. Therefore we have shown the following theorem.

\begin{theorem}[M{\oe}glin-Tadi{\'c}]
\label{thm: classification of discrete series}
There is a one to one correspondence between irreducible discrete series representations of $G^{\Sigma_{0}}(F)$ and admissible triples $(Jord, \r_{cusp}^{\Sigma_{0}}, \Delta)$.
\end{theorem}


One can also see how to construct irreducible discrete series representations from admissible triples according to Proposition~\ref{prop: parabolic reduction full orthogonal group}. If $(Jord, \r_{cusp}^{\Sigma_{0}}, \Delta)$ is an admissible triple of alternated type, let 
\[
l_{\rho}:  Jord_{\rho} \longrightarrow Jord^{+}_{\rho}(\p_{cusp})
\]
be the monotone bijection. Then the corresponding irreducible discrete series representation $\r^{\Sigma_{0}}$ can be viewed as the unique irreducible subrepresentation of 
\[
\Big(\prod_{\rho} (\prod_{a \in Jord_{\rho}} <(a -1)/2, \cdots, (l_{\rho}(a) + 1)/2>)\Big) \rtimes \r_{cusp}^{\Sigma_{0}},
\]
where the product over $Jord_{\rho}$ is in the increasing order. If $(Jord, \r_{cusp}^{\Sigma_{0}}, \Delta)$ is an admissible triple, we can assume $(Jord', \r_{cusp}^{\Sigma_{0}}, \Delta')$ is subordinated to $(Jord, \r_{cusp}^{\Sigma_{0}}, \Delta)$, where $Jord'_{\rho} = Jord_{\rho} \backslash \{a, a_{-}\}$. Suppose $\r'^{\Sigma_{0}}$ corresponds to $(Jord', \r_{cusp}^{\Sigma_{0}}, \Delta')$, then 
\[
<(a-1)/2, \cdots, -(a_{-}-1)/2> \rtimes \r'^{\Sigma_{0}}
\]
has two irreducible subrepresentations, and one will correspond to $\r^{\Sigma_{0}}$ while the other corresponds to the other extension of $\Delta'$ to $Jord$.


\section{Remarks on the original approach of M{\oe}glin and Tadi{\'c}}
\label{sec: original approach}

The original approach of M{\oe}glin and Tadi{\'c} to Theorem~\ref{thm: classification of discrete series} does not depend on Arthur's theory, i.e., Theorem~\ref{thm: L-packet}, ~\ref{thm: L-packet full orthogonal group}, ~\ref{thm: character relation} and ~\ref{thm: character relation full orthogonal group}. So the first immediate question becomes how to associate a set of Jordan blocks to every irreducible discrete series representation of $G^{\Sigma_{0}}(F)$ without assuming Arthur's theory. The answer can be motivated by the following result due to Arthur. It follows from the computation of the $R$-group defined by parameters and the fact that they are isomorphic to the representation theoretic $R$-group (see \cite[Section 2.4 and Section 6.6]{Arthur:2013}).

\begin{theorem}
\label{thm: tempered reducibility}
Suppose $\r^{\Sigma_{0}}$ is an irreducible discrete series representation of $G^{\Sigma_{0}}(F)$, and $\r^{\Sigma_{0}} \in \Pkt{\p}^{\Sigma_{0}}$ for some $\p \in \cPdt{G}$. 
Then for any self-dual irreducible supercuspidal representation $\rho$ of $GL(d_{\rho}, F)$ and positive integer $a$, $(\rho, a) \in Jord(\p)$ if and only if $(\rho, a)$ is of the same type as $\D{G}$, and
\begin{align}
\label{eq: tempered reducibility}
St(\rho, a) \rtimes \r^{\Sigma_{0}}
\end{align}
is irreducible.

\end{theorem}
It is clear from this theorem that we can associate every irreducible discrete series representation $\r^{\Sigma_{0}}$ of $G^{\Sigma_{0}}(F)$ with a set $Jord(\r^{\Sigma_{0}})$ of Jordan blocks as follows,
\[
Jord(\r^{\Sigma_{0}}) := \{(\rho, a) \text{ of the same type as $\D{G}$ }: \rho \text{ is self-dual supercuspidal, $a \in \mathbb{Z}_{>0}$ and \eqref{eq: tempered reducibility} is irreducible} \}.
\]
The next question is about the construction of $\Two$-valued function $\Delta$ (see Section~\ref{sec: classification of discrete series}). In \cite{Moeglin:2002}, M{\oe}glin defines $\Delta$ over a subset of 
\[
Jord(\r^{\Sigma_{0}}) \sqcup (Jord(\r^{\Sigma_{0}}) \times Jord(\r^{\Sigma_{0}})),
\]
i.e., $\Delta$ is not defined on $(\rho, a) \in Jord(\r^{\Sigma_{0}})$ with $a$ being odd and $Jord_{\rho}(\r^{\Sigma_{0}}_{cusp}) \neq \emptyset$; $\Delta$ is not defined on pairs $(\rho, a), (\rho', a') \in Jord(\r^{\Sigma_{0}})$ with $\rho \neq \rho'$. Moreover, $\Delta$ satisfies those properties that we have described in Section~\ref{sec: classification of discrete series}. Here we will only mention how to define $\Delta$ for pairs $(\rho, a), (\rho, a_{-}) \in Jord(\r^{\Sigma_{0}})$, where $a_{-}$ is the biggest positive integer in $Jord_{\rho}(\r^{\Sigma_{0}})$ that is smaller than $a$, and also for $(\rho, a_{min}) \in Jord(\r^{\Sigma_{0}})$ with $a_{min} = min \, Jord_{\rho}(\r^{\Sigma_{0}})$ being even. In view of Proposition~\ref{prop: parabolic reduction full orthogonal group}, this definition is given in the reversed way, i.e.,
\begin{enumerate}

\item $\Delta(\rho, a; \rho, a_{-}) = 1$ if and only if 
\[
\r^{\Sigma_{0}} \hookrightarrow <(a-1)/2, \cdots, (a_{-}+1)/2> \rtimes \r_{-}^{\Sigma_{0}}
\]
for some irreducible representation $\r_{-}^{\Sigma_{0}}$ of $G^{\Sigma_{0}}_{-}(F)$.

\item When $a_{min}$ is even, $\Delta(\rho, a_{min}) = 1$ if and only if
\[
\r^{\Sigma_{0}} \hookrightarrow <(a_{min}-1)/2, \cdots, 1/2> \rtimes \r_{-}^{\Sigma_{0}}
\]
for some irreducible representation $\r_{-}^{\Sigma_{0}}$ of $G^{\Sigma_{0}}_{-}(F)$.

\end{enumerate}

At last, for $G(n)$ we let $N = 2n + 1$ if $G$ is symplectic, and $N = 2n$ otherwise. Then M{\oe}glin proved the following dimension equality.


\begin{theorem}[M{\oe}glin \cite{Moeglin:2014}]
\label{conj: dimension equality}
Suppose $\r^{\Sigma_{0}}$ is a discrete series representation of $G^{\Sigma_{0}}(n, F)$, then 
\[
\sum_{(\rho, a) \in Jord(\r^{\Sigma_{0}})} a d_{\rho} = N.
\]
\end{theorem}
This theorem becomes trivial if we know Theorem~\ref{thm: L-packet} and identify $Jord(\r^{\Sigma_{0}}) = Jord(\p)$ under Theorem~\ref{thm: tempered reducibility}. But without assuming all these results of Arthur, this theorem is far from being obvious. 


\appendix

\section{Local L-function}
\label{sec: local L-function}

In this appendix, we give explicit formulas for three different types of local L-functions, i.e., Rankin-Selberg L-function, symmetric square L-function and skew symmetric square L-function. Let $F$ be a $p$-adic field, and $q$ be the number of elements in the residue field of $F$. 

\subsection{Rankin-Selberg L-function}
\label{subsec: Rankin-Selberg L-function}

We follow \cite{JPSS:1983} here. Let $\r$ be an irreducible smooth representation of $GL(n, F)$ and $\sigma$ be an irreducible smooth representation of $GL(m, F)$, the local Rankin-Selberg L-function is denoted by $L(s, \r \times \sigma)$ for $s \in \C$. It satisfies $L(s, \r \times \sigma) = L(s, \sigma \times \r)$ and $L(s, \r||^{t} \times \sigma) = L(s + t, \r \times \sigma)$. 

{\bf Cuspidal case:}

Suppose both $\r$ and $\sigma$ are unitary supercuspidal representations. 

\begin{enumerate}

\item If $n \neq m$, then $L(s, \r \times \sigma) = 1$;

\item If $n = m$, then 
\[
L(s, \r \times \sigma) = \prod_{t} (1 - q^{-(s + it)})^{-1}
\] 
where the product is over all real numbers $t$ such that $\r||^{it} \cong \sigma^{\vee}$.

\end{enumerate}

{\bf Discrete series case:}

We assume $\r$ is $St(\rho, a)$ for an irreducible unitary supercuspidal representations $\rho$ and integer $a$. Similarly we assume $\sigma$ is $St(\rho', b)$. If $n \geqslant m$, then
\[
L(s, \r \times \sigma) = \prod^{b}_{i=1} L(s + \frac{a+b}{2} - i, \rho \times \rho').
\]

{\bf Tempered case:}

Suppose $\r = \r_{1} \times \cdots \times \r_{l}$ and $\sigma = \sigma_{1} \times \cdots \times \sigma_{k}$, where $\r_{i}$, $\sigma_{j}$ are discrete series representations. Then
\[
L(s, \r \times \sigma) = \prod_{i,j} L(s, \r_{i} \times \sigma_{j}).
\]

{\bf Non-tempered case:}

Let $\r$ be the Langlands quotient of the induced representation $\Pi = \r_{1}||^{u_{1}} \times \cdots \times \r_{l}||^{u_{l}}$ for tempered representation $\r_{i}$ and real numbers $u_{1} > \cdots > u_{l}$. Let $\sigma$ be the Langlands quotient of the induced representation $\Sigma = \sigma_{1}||^{v_{1}} \times \cdots \times \sigma_{k}||^{v_{k}}$ for tempered representation $\sigma_{j}$ and real numbers $v_{1} > \cdots > v_{k}$. Then
\[
L(s, \r \times \sigma) = L(s, \Pi \times \Sigma) = \prod_{i,j} L(s + u_{i} + v_{j}, \r_{i} \times \sigma_{j}).
\]

\subsection{Symmetric square and skew-symmetric square L-functions}
\label{subsec: symmetric square L-function}

We follow \cite{Shahidi:1992} here. Let $\r$ be an irreducible smooth representation of $GL(n, F)$. The symmetric square (resp. skew-symmetric square) L-function is denoted by $L(s, \r, S^{2})$ (resp. $L(s, \r, \wedge^{2})$). We have $L(s, \r \times \r) = L(s, \r, S^{2})L(s, \r, \wedge^{2})$, and $L(s, \r||^{t}, R) = L(s + 2t, \r, R)$ for $R = S^{2}$ or $\wedge^{2}$.

{\bf Cuspidal case}

Suppose $\r$ is a unitary supercuspidal representation of $GL(n, F)$.

\begin{enumerate}

\item $L(s, \r, \wedge^{2}) = 1$ unless $n$ is even and some unramified twist of $\r$ is self-dual. So let us suppose $n$ is even and $\r$ is self-dual. Let $S$ be the set of real numbers $t$ modulo $\frac{\pi}{\ln q} \mathbb{Z}$, such that
\[
\int_{Sp_{n}(F) \backslash GL_{n}(F)} f({}^{t}gw^{-1}gw)dg \neq 0,
\]
for some $f \in C^{\infty}_{c}(GL_{n}(F))$ defining a matrix coefficient of $\r||^{it}$. Here ${}^{t}g$ is the transpose of $g$ and
\begin{align}
\label{eq: anti-diagonal}
w = \begin{pmatrix}
0 &&&&& 1\\
& &&& -1& \\
&&& \iddots && \\
&1&&& & \\
-1&&&&& 0
\end{pmatrix}.
\end{align}
Then 
\[
L(s, \r, \wedge^{2}) = \prod_{t \in S} (1 - q^{-(s + 2it)})^{-1}.
\]

\item $L(s, \r, S^{2}) = 1$ unless some unramified twist of $\r$ is self-dual. So let us suppose $\r$ is self-dual.

\begin{enumerate}

\item If $n$ is odd, then 
\[
L(s, \r, S^{2}) = (1 - q^{-rs})^{-1},
\]
where $r$ is the maximal integer such that $\r \cong \r||^{2\pi i/(r \ln q)}$.

\item If $n$ is even, 
\[
L(s, \r, S^{2}) = \prod_{t \in S'} (1 - q^{-(s + 2it)})^{-1},
\]
where $S'$ is the set of real numbers $t$ modulo $\frac{\pi}{\ln q} \mathbb{Z}$ such that $\r||^{2it} \cong \r$ and for any $f \in C^{\infty}_{c}(GL_{n}(F))$ defining a matrix coefficient of $\r||^{it}$, one has 
\[
\int_{Sp_{n}(F) \backslash GL_{n}(F)} f({}^{t}gw^{-1}gw)dg = 0.
\]
Here $w$ is again given by $\eqref{eq: anti-diagonal}$ and ${}^{t}g$ is the transpose of $g$.

\end{enumerate}

\end{enumerate}

{\bf Discrete series case:}

We assume $\r$ is $St(\rho, a)$ for an irreducible unitary supercuspidal representations $\rho$ and integer $a$. Set $\r_{i} = \rho||^{(a+1)/2 - i}$ for $1 \leqslant i \leqslant a$.

\begin{enumerate}

\item Suppose $a$ is even, then

\[
L(s, \r, \wedge^{2}) = \prod_{i=1}^{a/2} L(s, \r_{i}, \wedge^{2}) L(s, \r_{i}||^{-1/2}, S^{2}),
\]

\[
L(s, \r, S^{2}) = \prod_{i=1}^{a/2} L(s, \r_{i}, S^{2}) L(s, \r_{i}||^{-1/2}, \wedge^{2}).
\]

\item Suppose $a$ is odd, then

\[
L(s, \r, \wedge^{2}) = \prod_{i=1}^{(a+1)/2} L(s, \r_{i}, \wedge^{2}) \prod_{i=1}^{(a-1)/2}L(s, \r_{i}||^{-1/2}, S^{2}),
\]

\[
L(s, \r, S^{2}) = \prod_{i=1}^{(a+1)/2} L(s, \r_{i}, S^{2}) \prod_{i=1}^{(a-1)/2}L(s, \r_{i}||^{-1/2}, \wedge^{2}).
\]

\end{enumerate}

{\bf Tempered case:}

Suppose $\r = \r_{1} \times \cdots \times \r_{l}$, where $\r_{i}$ are discrete series representations. Then

\[
L(s, \r, \wedge^{2}) = \prod_{i=1}^{l} L(s, \r_{i}, \wedge^{2}) \prod_{1 \leqslant i < j \leqslant l}L(s, \r_{i} \times \r_{j}),
\]

\[
L(s, \r, S^{2}) = \prod_{i=1}^{l} L(s, \r_{i}, S^{2}) \prod_{1 \leqslant i < j \leqslant l}L(s, \r_{i} \times \r_{j}).
\]

{\bf Non-tempered case:}

Let $\r$ be the Langlands quotient of the induced representation $\Pi = \r_{1}||^{u_{1}} \times \cdots \times \r_{l}||^{u_{l}}$ for tempered representation $\r_{i}$ and real numbers $u_{1} > \cdots > u_{l}$. Then
\[
L(s, \r, \wedge^{2}) = L(s, \Pi, \wedge^{2}) = \prod_{i=1}^{l} L(s + 2u_{i}, \r_{i}, \wedge^{2}) \prod_{1 \leqslant i < j \leqslant l}L(s + u_{i} + u_{j}, \r_{i} \times \r_{j}),
\] 

\[
L(s, \r, S^{2}) = L(s, \Pi, S^{2}) = \prod_{i=1}^{l} L(s + 2u_{i}, \r_{i}, S^{2}) \prod_{1 \leqslant i < j \leqslant l}L(s + u_{i} + u_{j}, \r_{i} \times \r_{j}).
\]

\section{Reducibility for some induced representations of GL(n)}
\label{sec: reducibility for GL(n)}

We define a {\bf segment} to be a finite length arithmetic progression of real numbers with common difference $1$ or $-1$, it is completely determined by its endpoints $x$, $y$, and hence we denote a segment by $[x, y]$ or $\{x, \cdots, y\}$. Let $F$ be a $p$-adic field and $\rho$ be a unitary irreducible supercuspidal representation of $GL(d_{\rho}, F)$. The normalized induction
\[
\rho||^{x} \times \cdots \times \rho||^{y}
\]
has a unique irreducible subrepresentation, which is denoted by $<\rho; x, \cdots, y>$ or $<x, \cdots, y>$. If $x \geqslant y$, this is called Steinberg representation; if $x < y$, this is called Speh representation. 

For any two segments $[x, y]$ and $[x', y']$ such that $(x - y)(x' - y') \geqslant 0$, we say they are {\bf linked} if as sets $[x, y] \nsubseteq [x', y']$, $[x', y'] \nsubseteq [x, y]$, and $[x, y]  \cup [x', y']$ can form a segment after imposing the same order. The following theorem is fundamental in determining the reducibility of an induced representation of $GL(n, F)$.

\begin{theorem}[Zelevinsky \cite{Zelevinsky:1980}]
\label{thm: irreducibility for induced representation of GL(n)}
For unitary irreducible supercuspidal representations $\rho, \rho'$ of general linear groups, and segments $[x, y]$, $[x', y']$ such that $(x - y)(x' - y') \geqslant 0$, 
\[
<\rho; x, \cdots, y > \times <\rho'; x', \cdots, y'>
\]
is reducible if and only if $\rho \cong \rho'$ and $[x, y], [x', y']$ are linked. In case it is reducible, it consists of the unique irreducible subrepresentations of
\[
<\rho; x, \cdots, y> \times <\rho; x', \cdots, y'> \text{ and } <\rho; x', \cdots, y'> \times <\rho; x, \cdots, y>.
\]
\end{theorem}

\begin{remark}
In fact, Zelevinsky proved this theorem only when both $x-y \geqslant 0$ and $x'-y' \geqslant 0$. Nonetheless, the Aubert involution functor on the Grothendieck group of finite length representations of $GL(n, F)$ will send 
\[
<\rho; x, \cdots, y > \times <\rho'; x', \cdots, y'> \text{ to } <\rho; y, \cdots, x > \times <\rho'; y', \cdots, x'>
\] 
up to a sign, and it preserves irreducibility (see \cite{Aubert:1995}). So one can easily extend the original result of Zelevinsky to this theorem.
\end{remark}

It is natural to ask for the notion of ``link" for two segments $[x, y]$ and $[x', y']$ such that $(x - y)(x' - y') < 0$. To do so, we need to first generalize the notion of ``segment". We define a {\bf generalized segment} to be a matrix 
\begin{align*}
  \begin{bmatrix}
   x_{11} & \cdots & x_{1n} \\
   \vdots &  & \vdots \\
   x_{m1} & \cdots & x_{mn}
  \end{bmatrix}
\end{align*}
such that each row is a decreasing (resp. increasing) segment and each column is an increasing (resp. decreasing) segment. The normalized induction
\[
\times_{i \in [1, m]} <\rho; x_{i1}, \cdots, x_{in} >
\]
has a unique irreducible subrepresentation, and we denote it by $<\rho; \{x_{ij}\}_{m \times n}>$. Moreover, 
\[
<\rho; \{x_{ij}\}_{m \times n}> \cong <\rho; \{x_{ij}\}^{T}_{m \times n}>
\] 
where $\{x_{ij}\}^{T}_{m \times n}$ is the transpose of $\{x_{ij}\}_{m \times n}$.

For any two generalized segments $\{x_{ij}\}_{m \times n}$ and $\{y_{ij}\}_{m' \times n'}$ with the same monotone properties for the rows and columns, we say they are {\bf linked} if $[x_{m1}, x_{1n}]$, $[y_{m'1}, y_{1n'}]$ are linked, and the four sides of the rectangle formed by $\{x_{ij}\}_{m \times n}$ do not have inclusive relations with the corresponding four sides of the rectangle formed by $\{y_{ij}\}_{m' \times n'}$ (e.g., $[x_{11}, x_{1n}] \nsubseteq [y_{11}, y_{1n'}] \text{ and } [x_{11}, x_{1n}] \nsupseteq [y_{11}, y_{1n'}]$, etc). It is easy to check that if $\{x_{ij}\}_{m \times n}$ and $\{y_{ij}\}_{m' \times n'}$ are linked, then $\{x_{ij}\}^{T}_{m \times n}$ and $\{y_{ij}\}^{T}_{m' \times n'}$ are also linked. So for generalized segments $\{x_{ij}\}_{m \times n}$ and $\{y_{ij}\}_{m' \times n'}$ with different monotone properties for the rows and columns, we say they are {\bf linked} if $\{x_{ij}\}^{T}_{m \times n}$ and $\{y_{ij}\}_{m' \times n'}$ are linked, or equivalently $\{x_{ij}\}_{m \times n}$ and $\{y_{ij}\}^{T}_{m' \times n'}$ are linked. One can check this notion of ``link" is equivalent to the one in \cite{MW:1989}.

\begin{example}
For any two segments $[x, y]$ and $[x', y']$ such that $(x - y)(x' - y') < 0$, we can view them as generalized segments by taking them as rows, and note they have different monotone properties. So we take
\begin{align*}
[x, y]^{T} = 
    \begin{bmatrix}
     x \\
     \vdots \\
     y
     \end{bmatrix}
\quad \text{ and } \quad 
[x', y'] =  
   \begin{bmatrix}
    x' \cdots y' 
   \end{bmatrix}.
\end{align*}
It follows that $[x, y]$ and $[x', y']$ are linked if and only if $[y, x], [x', y']$ are linked, and $x, y \notin [x', y']$ and $x', y' \notin [x, y]$.
\end{example}



The next theorem generalizes Theorem~\ref{thm: irreducibility for induced representation of GL(n)} to the case of generalized segments.

\begin{theorem}[M{\oe}glin-Waldspurger \cite{MW:1989}]
\label{thm: generalized irreducibility for induced representation of GL(n)}
For unitary irreducible supercuspidal representations $\rho, \rho'$ of general linear groups, and generalized segments $\{x_{ij}\}_{m \times n}$, $\{y_{ij}\}_{m' \times n'}$,
\[
<\rho; \{x_{ij}\}_{m \times n}> \times <\rho'; \{y_{ij}\}_{m' \times n'}>
\]
is irreducible unless $\rho \cong \rho'$ and $\{x_{ij}\}_{m \times n}, \{y_{ij}\}_{m' \times n'}$ are linked.
\end{theorem}

Let $a, b$ be integers, we define $Sp(St(\rho, a), b)$ to be the unique irreducible subrepresentation of 
\[
St(\rho, a)||^{-(b-1)/2} \times St(\rho, a)||^{-(b-3)/2} \times \cdots \times St(\rho, a)||^{(b-1)/2}.
\]
By the definition one can see $Sp(St(\rho, a), b)$ is given by the following generalized segment
\begin{align*}
  \begin{bmatrix}
   (a-b)/2 & \cdots & 1-(a+b)/2 \\
   \vdots &  & \vdots \\
   (a+b)/2-1 & \cdots & -(a-b)/2
  \end{bmatrix}
\end{align*}
The following result is a reinterpretation of Theorem~\ref{thm: generalized irreducibility for induced representation of GL(n)}.

\begin{corollary}
\label{cor: generalized irreducibility for induced representation of GL(n)}
For unitary irreducible supercuspidal representations $\rho, \rho'$ of general linear groups, and integers $a, b, a', b'$, and real number $s$,
\[
Sp(St(\rho, a), b)||^{s} \times Sp(St(\rho, a'), b')
\]
is irreducible unless $\rho \cong \rho'$, $(a + b + a' + b')/2 + s$ is an integer and 
\[
|(a - a')/2| + |(b-b')/2| < |s| \leqslant |(a + a' + b + b')/2| -1.
\]
\end{corollary}

\section{Casselman's formula and Application}
\label{sec: Casselman's formula}

Let $F$ be a $p$-adic field, and $G$ be a quasisplit connected reductive group over $F$. Let $\theta$ be an $F$-automorphism of $G$ preserving an $F$-splitting, and we assume $\theta$ has order $l$. Suppose $\r$ is an irreducible smooth representation of $G(F)$ such that $\r \cong \r^{\theta}$. Let $A_{\r}(\theta)$ be an intertwining operator between $\r$ and $\r^{\theta}$, then we can define the twisted character of $\r$ to be 
\[
f_{\com{G}}(\r) :=  trace \int_{G(F)} f(g)\r(g) dg \circ A_{\r}(\theta)
\]
for $f \in C^{\infty}_{c}(G(F))$. It follows from results of Harish-Chandra \cite{H-C:1999} in the non-twisted case and Clozel \cite{Clozel:1987} in the twisted case, that there exists a locally integrable function $\Theta^{\com{G}}_{\r}$ on $G(F)$ such that 
\[
f_{\com{G}}(\r) = \int_{G(F)} f(g)\Theta^{\com{G}}_{\r}(g) dg.
\]
We also call this function the twisted character of $\r$. If $P = MN$ is a $\theta$-stable parabolic subgroup of $G$, we denote by $\r_{N}$ the unnormalized Jacquet module of $\r$ with respect to $P$, then $\Jac_{P}\r = \r_{N} \otimes \delta_{P}^{-1/2}$, where $\delta_{P}$ is the usual modulus character. Note $\r_{N} \cong \r_{N}^{\theta}$ and $A_{\r}(\theta)$ induces an intertwining operator on $\r_{N}$.

We would like to extend $\Jac_{P}$ to the space of twisted invariant distributions on $G(F)$. Let $R(G^{\theta})$ be the space of finite linear combinations of twisted characters of $G(F)$ and 
\[
R^{*}(G^{\theta}) = \Hom(R(G^{\theta}), \C).
\] 
For any $\tau \in R(G^{\theta})$ and $f \in C^{\infty}_{c}(G(F))$, we define $f_{\com{G}}(\tau)$ by linearity. So we can get a homomorphism 
\[
C^{\infty}_{c}(G(F)) \longrightarrow R^{*}(G^{\theta})
\] 
by sending $f$ to $l_{f}(\tau): = f_{\com{G}}(\tau)$. Let us denote the image of this homomorphism by $F_{tr}(G^{\theta})$. Since $C^{\infty}_{c}(G(F))$ is equipped with a direct limit topology of finite dimensional subspaces, any linear functional on this space is continuous. Moreover, we know the twisted invariant linear functionals on $C^{\infty}_{c}(G(F))$ are supported on twisted characters (see \cite[Appendix, Theorem 1]{Kazhdan:1986} and \cite[Section 5.5]{Waldspurger:2017}), so we can identify the space $\D{I}(G^{\theta})$ of twisted invariant distributions on $G(F)$ with $\D{F}_{tr}(G^{\theta}) := \Hom(F_{tr}(G^{\theta}), \C)$. Under this identification, we have an inclusion of $R(G^{\theta})$ in $\D{F}_{tr}(G^{\theta})$, which sends $\tau$ to $L_{\tau}(l_{f}) := l_{f}(\tau)$ for any $f \in C^{\infty}_{c}(G(F))$.

The Jacquet functor induces a homomorphism 
\[
\Jac_{P}: R(G^{\theta}) \longrightarrow R(M^{\theta}),
\]
whose dual is
\[
\Jac^{*}_{P}: R^{*}(M^{\theta}) \longrightarrow R^{*}(G^{\theta})
\]
(see \cite[Lemma 4.1]{Rogawski:1988}). The key step in extending $\Jac_{P}$ to the space of twisted invariant distributions is the following lemma.
\begin{lemma}
\label{lemma: Jacquet dual}
$\Jac^{*}_{P}(F_{tr}(M^{\theta})) \subseteq F_{tr}(G^{\theta})$.
\end{lemma}
This lemma was first proved in the nontwisted case (see \cite[Proposition 3.2]{B-D-K:1986}) and was later extended to the twisted case (see \cite[Proposition 7.1]{Rogawski:1988}). As a consequence, we have a homomorphism 
\[
\Jac^{*}_{P}: F_{tr}(M^{\theta}) \longrightarrow F_{tr}(G^{\theta}),
\]
whose dual is
\[
\Jac_{P}: \D{F}_{tr}(G^{\theta}) \longrightarrow \D{F}_{tr}(M^{\theta}).
\]
By our previous identifications, this gives
\begin{align}
\label{eq: Jacquet on distribution}
\Jac_{P}: \D{I}(G^{\theta}) \longrightarrow \D{I}(M^{\theta}).
\end{align}
To see this is really an extension of $\Jac_{P}$ on $R(G^{\theta})$, we have the following proposition.

\begin{proposition}
\label{prop: extend Jacquet}
The following diagram commutes.
\begin{align*}
\xymatrix{ R(G^{\theta}) \ar[d]  \ar[r]^{\Jac_{P}} & R(M^{\theta}) \ar[d] \\
                \D{I}(G^{\theta}) \ar[r]^{\Jac_{P}}  & \D{I}(M^{\theta}).}
\end{align*}

\end{proposition}

\begin{proof}
Let $\tau \in R(G^{\theta})$ and $h \in C^{\infty}_{c}(M(F))$. By definition, we have
\begin{align*}
\Jac_{P}(L_{\tau})(l_{h}) = L_{\tau}(\Jac^{*}_{P} l_{h}) = (\Jac^{*}_{P} l_{h})(\tau) = l_{h}(\Jac_{P} \tau) = L_{(\Jac_{P}\tau)}(l_{h}).
\end{align*}
This finishes the proof.

\end{proof}

Let $R(G)^{st}$ be the space of stable finite linear combinations of characters of $G(F)$ and 
\[
R^{*}(G)^{st} = \Hom(R(G)^{st}, \C).
\] 
In the same way, we have
\begin{align}
\label{eq: stable trace}
C^{\infty}_{c}(G(F)) \longrightarrow R^{*}(G)^{st}.
\end{align}
Let us denote the image by $F_{tr}(G)^{st}$. The space $\D{SI}(G)$ of stable invariant distributions on $G(F)$ are linear functionals of the space of stable orbital integrals of $C^{\infty}_{c}(G(F))$. It follows from \cite[Theorem 6.1 and Theorem 6.2]{Arthur:1996} that one can characterize the space of stable orbital integrals as 
space of functions on certain subset of $R(G)^{st}$ through \eqref{eq: stable trace}. In particular, the stable invariant distributions are also supported on characters. So we can identify $\D{SI}(G)$ with $\D{F}_{tr}(G)^{st} := \Hom(F_{tr}(G)^{st}, \C)$. By \cite[Lemma 2.3]{Hiraga:2004}, the Jacquet functor induces a homomorphism 
\[
\Jac_{P}: R(G)^{st} \longrightarrow R(M)^{st},
\]
whose dual is
\[
\Jac^{*}_{P}: R^{*}(M)^{st} \longrightarrow R^{*}(G)^{st}.
\]
The following lemma is an analogue of Lemma~\ref{lemma: Jacquet dual}.
\begin{lemma}
$\Jac^{*}_{P}(F_{tr}(M)^{st}) \subseteq F_{tr}(G)^{st}$.
\end{lemma}

\begin{proof}
By Lemma~\ref{lemma: Jacquet dual}, we have a commutative diagram
\begin{align*}
\xymatrix{  F_{tr}(M) \ar[d]  \ar[r]^{\Jac^{*}_{P}} & F_{tr}(G) \ar[d] \\
                R^{*}(M)^{st}  \ar[r]^{\Jac^{*}_{P}} & R^{*}(G)^{st}.}
\end{align*}
Then this lemma follows from the fact that $F_{tr}(M)^{st}$ and $F_{tr}(G)^{st}$ are images of $F_{tr}(M)$ and $F_{tr}(G)$ respectively.

\end{proof}

As a consequence, we have a homomorphism 
\[
\Jac^{*}_{P}: F_{tr}(M)^{st} \longrightarrow F_{tr}(G)^{st},
\]
whose dual is
\[
\Jac_{P}: \D{F}_{tr}(G)^{st} \longrightarrow \D{F}_{tr}(M)^{st}.
\]
By our previous identifications, this gives
\begin{align}
\label{eq: Jacquet preserve stability}
\Jac_{P}: \D{SI}(G) \longrightarrow \D{SI}(M).
\end{align}
The following proposition is a direct consequence of Proposition~\ref{prop: extend Jacquet}.

\begin{proposition}
The following diagram commutes.
\begin{align*}
\xymatrix{ R(G)^{st} \ar[d]  \ar[r]^{\Jac_{P}} & R(M)^{st} \ar[d] \\
                \D{SI}(G) \ar[r]^{\Jac_{P}}  & \D{SI}(M).}
\end{align*}
\end{proposition}

For a strongly $\theta$-regular $\theta$-semisimple element $g$ in $G(F)$, let $h = \prod_{i=1}^{l}\theta^{i}(g)$, and one can associate it with a $\theta$-stable parabolic subgroup $P_{h} = M_{h}N_{h}$ by the construction in \cite{Casselman:1977}, and $g \in M_{h}(F)$. Now we can state the twisted version of Casselman's formula.

\begin{theorem}
Suppose $\r$ is an irreducible smooth representation of $G(F)$ such that $\r \cong \r^{\theta}$, and $g$ is a strongly $\theta$-regular $\theta$-semisimple element in $G(F)$. Let $h = \prod_{i=1}^{l}\theta^{i}(g)$. Then
\[
\Theta^{G^{\theta}}_{\r}(g) = \Theta^{M_{h}^{\theta}}_{\r_{N_{h}}}(g).
\]
\end{theorem}

Casselman proved this theorem only for $\theta = id$ in \cite[Theorem 5.2]{Casselman:1977}. It was generalized to the twisted case by Rogawski in \cite[Proposition 7.4]{Rogawski:1988}. What we are going to use is the following corollary of this theorem.

\begin{corollary}
\label{cor: Casselman's formula}
Let $P = MN$ be a $\theta$-stable parabolic subgroup of $G$. Suppose $\r$ is an irreducible smooth representation of $G(F)$ such that $\r \cong \r^{\theta}$, and $m$ is a strongly $\theta$-regular $\theta$-semisimple element in $G(F)$, which is also contained in $M(F)$. Then one can choose $z_{M} \in A_{M}(F)$ ($A_{M}$ is the maximal split component of the centre of $M$) with $|\alpha(z_{M})|$ sufficiently small (depending on $m$) for all roots $\alpha$ in $N$, such that
\[
\Theta_{\r}^{G^{\theta}}(z_{M} m) = \Theta^{M^{\theta}}_{\r_{N}}(z_{M} m).
\]
\end{corollary}

\begin{proof}
Let $g = z_{M}m$, and $h = \prod_{i=1}^{l}\theta^{i}(g)$. It is not hard to check from the definition of $P_{h}$ 
that $P \supseteq P_{h}$. Let $P^{M}_{h} = M \cap P_{h}$, it is the parabolic subgroup of $M$ associated with $h$, and it has Levi component $M_{h}$. Let $N^{M}_{h} = M \cap N_{h}$. Then
\[
\Theta^{G^{\theta}}_{\r}(g) = \Theta^{M_{h}^{\theta}}_{\r_{N_{h}}}(g) = \Theta^{M_{h}^{\theta}}_{(\r_{N})_{N^{M}_{h}}}(g) = \Theta^{M^{\theta}}_{\r_{N}}(g).
\]
This finishes the proof.

\end{proof}

As an application of this corollary, we are going to establish diagrams \eqref{eq: compatible with endoscopic transfer 0} and \eqref{eq: compatible with twisted endoscopic transfer 0}. First, let us recall the general setup of these diagrams. Let $H$ be a twisted endoscopic group of $G$, and we assume there is an embedding 
\[
\xi: \L{H} \rightarrow \L{G},
\]
and $\xi(\L{H}) \subseteq \Cent(s, \L{G})$ and $\D{H} \cong \Cent(s, \D{G})^{0}$ for some semisimple $s \in \D{G} \rtimes \D{\theta}$. We fix ($\D{\theta}$-stable) $\Gal{F}$-splittings $(\mathcal{B}_{H}, \mathcal{T}_{H}, \{\mathcal{X}_{\alpha_{H}}\})$ and $(\mathcal{B}_{G}, \mathcal{T}_{G}, \{\mathcal{X}_{\alpha}\})$ for $\D{H}$ and $\D{G}$ respectively. By taking certain $\D{G}$-conjugate of $\xi$, we can assume $s \in \mathcal{T}_{G} \rtimes \D{\theta}$ and $\xi(\mathcal{T}_{H}) = (\mathcal{T}_{G}^{\D{\theta}})^{0}$ and $\xi(\mathcal{B}_{H}) \subseteq \mathcal{B}_{G}$. Let $W_{H} = W(\D{H}, \mathcal{T}_{H})$ and $W_{G^{\theta}} = W(\D{G}, \mathcal{T}_{G})^{\D{\theta}}$, then $W_{H}$ can be viewed as a subgroup of $W_{G^{\theta}}$. We also view $\L{H}$ as a subgroup of $\L{G}$ through $\xi$.

We fix a standard $\theta$-stable parabolic subgroup $P = MN$ of $G$ with standard embedding $\L{P} \hookrightarrow \L{G}$. Then there exists a torus $S \subseteq (\mathcal{T}_{G}^{\D{\theta}})^{0}$ such that $\L{M} = \Cent(S, \L{G})$. Let $W_{M^{\theta}} =  W(\D{M}, \mathcal{T}_{G})^{\D{\theta}}$. We define
\[
W_{G^{\theta}}(H, M) := \{w \in W_{G^{\theta}} | \, \Cent(w(S), \L{H}) \rightarrow W_{F} \text{ surjective }\}.
\]
For any $w \in W_{G^{\theta}}(H, M)$, let us take $g \in \D{G}$ such that $\Int(g)$ induces $w$. Since $\Cent(w(S), \L{H}) \rightarrow W_{F}$ is surjective, $g \L{P} g^{-1} \cap \L{H}$ defines a parabolic subgroup of $\L{H}$ with Levi component $g \L{M} g^{-1} \cap \L{H}$. So we can choose a standard parabolic subgroup $P'_{w} = M'_{w}N'_{w}$ of $H$ with standard embedding $\L{P'_{w}} \hookrightarrow \L{H}$ such that $\L{P'_{w}}$ (resp. $\L{M'_{w}}$) is $\D{H}$-conjugate to $g \L{P} g^{-1} \cap \L{H}$ (resp. $g \L{M} g^{-1} \cap \L{H}$). In particular, $M'_{w}$ can be viewed as a twisted endoscopic group of $M$, and the embedding $\xi_{M'_{w}}: \L{M'_{w}} \rightarrow \L{M}$ is given by the following diagram:
\[
\xymatrix{ \L{P'_{w}} \ar@{^{(}->}[d] & \L{M'_{w}} \ar[l] \ar[r]^{\xi_{M'_{w}}} & \L{M} \ar[r] & \L{P} \ar@{^{(}->}[d]  \\
\L{H} \ar[r]^{\Int(h)} & \L{H} \ar[r]^{\xi} & \L{G} & \L{G} \ar[l]_{\Int(g)} 
}
\]
where $h \in \D{H}$ induces an element in $W_{H}$. Note the choice of $h$ is unique up to $\D{M}'_{w}$-conjugation, and so is $\xi_{M'_{w}}$. If we change $g$ to $h'gm$, where $h' \in \D{H}$ induces an element in $W_{H}$ and $m \in \D{M}$ induces an element in $W_{M^{\theta}}$, then we still get $P'_{w}$, but $\xi_{M'_{w}}$ changes to $\Int(m^{-1}) \circ \xi_{M'_{w}}$ up to $\D{M}'_{w}$-conjugation. To summarize, for any element $w$ in 
\[
W_{H} \backslash W_{G^{\theta}}(H, M)/W_{M^{\theta}}
\]
we can associate a standard parabolic subgroup $P'_{w} = M'_{w} N'_{w}$ of $H$ and a $\D{M}$-conjugacy class of embedding $\xi_{M'_{w}}: \L{M'_{w}} \rightarrow \L{M}$.

After this setup, we claim the following diagram commutes.
\begin{align}
\label{eq: compatible with twisted endoscopic transfer general} 
\xymatrix{\D{SI}(H) \ar[d]_{\+_{w} \Jac_{P'_{w}}}  \ar[r] & \D{I}(G^{\theta}) \ar[d]^{\Jac_{P}} \\
                \bigoplus_{w} \D{SI}(M'_{w}) \ar[r]  & \D{I}(M^{\theta}), }
\end{align}
where the sum is over $W_{H} \backslash W_{G^{\theta}}(H, M)/W_{M^{\theta}}$, and the horizontal maps correspond to the twisted spectral endoscopic transfers with respect to $\xi$ on the top and $\xi_{M'_{w}}$ on the bottom. In fact, it is enough to show the commutativity of this diagram restricting to the subspaces of finite linear combinations of (twisted) characters. 
\begin{align}
\label{eq: compatible with twisted endoscopic transfer general character case} 
\xymatrix{R(H)^{st} \ar[d]_{\+_{w} \Jac_{P'_{w}}}  \ar[r] & R(G^{\theta}) \ar[d]^{\Jac_{P}} \\
                \bigoplus_{w} R(M'_{w})^{st} \ar[r]  & R(M^{\theta}),}
\end{align}
where the horizontal maps are given by the restrictions of twisted spectral endoscopic transfers, whose existences are due to \cite{Arthur:1996} and \cite[Appendix]{MW:2016}.

\begin{lemma}
The commutativity of the diagram \eqref{eq: compatible with twisted endoscopic transfer general character case} implies that of \eqref{eq: compatible with twisted endoscopic transfer general}. 
\end{lemma}

\begin{proof}

By taking dual of the diagram \eqref{eq: compatible with twisted endoscopic transfer general character case}, we get
\begin{align*}
\xymatrix{R^{*}(H)^{st} & R^{*}(G^{\theta}) \ar[l]  \\
                \bigoplus_{w} R^{*}(M'_{w})^{st} \ar[u]^{\+_{w} \Jac^{*}_{P'_{w}}}  & R^{*}(M^{\theta}) \ar[l]   \ar[u]_{\Jac^{*}_{P}}.}
\end{align*}
By restriction, we have
\begin{align*}
\xymatrix{F_{tr}(H)^{st} & F_{tr}(G^{\theta}) \ar[l]  \\
                \bigoplus_{w} F_{tr}(M'_{w})^{st} \ar[u]^{\+_{w} \Jac^{*}_{P'_{w}}}  & F_{tr}(M^{\theta}) \ar[l]   \ar[u]_{\Jac^{*}_{P}}.}
\end{align*}
Then the diagram \eqref{eq: compatible with twisted endoscopic transfer general} is simply its dual.

\end{proof}

As a consequence of this lemma, we only need to show \eqref{eq: compatible with twisted endoscopic transfer general character case}. To apply Corollary~\ref{cor: Casselman's formula}, we need to give another description of the twisted spectral endoscopic transfer. With respect to the embedding $\xi$, there is a map from the semisimple $H(\bar{F})$-conjugacy classes of $H(\bar{F})$ to the $\theta$-twisted semisimple $G(\bar{F})$-conjugacy classes of $G(\bar{F})$ (see \cite{KottwitzShelstad:1999}). If $\Theta^{G^{\theta}}$ is a finite linear combination of twisted characters of $G(F)$ and $\Theta^{H}$ is a stable finite linear combination of characters of $H(F)$, then we say $\Theta^{H}$ transfers to $\Theta^{G^{\theta}}$ if for any strongly $\theta$-regular $\theta$-semisimple element $\gamma_{G}$ in $G(F)$
\begin{align}
\label{eq: spectral transfer}
\Theta^{G^{\theta}}(\gamma_{G}) = \sum_{\gamma_{H} \rightarrow \gamma_{G}} \frac{D_{H}(\gamma_{H})^{2}}{D_{G^{\theta}}(\gamma_{G})^{2}} \Delta_{G, H}(\gamma_{H}, \gamma_{G}) \Theta^{H}(\gamma_{H})
\end{align}
where the sum is over $H(\bar{F})$-conjugacy classes of $\gamma_{H}$ in $H(F)$ that map to the $\theta$-twisted $G(\bar{F})$-conjugacy class of $\gamma_{G}$. In this formula, $\Delta_{G, H}( \cdot, \cdot)$ is the transfer factor (see \cite{KottwitzShelstad:1999}), and it is built into the transfer map introduced in Section~\ref{sec: endoscopy}; $D_{H}(\cdot)$ and $D_{G^{\theta}}(\cdot)$ are the (twisted) Weyl discriminants.

From now on let us fix $\gamma_{G} = \gamma_{M}$ contained in $M(F)$.

\begin{lemma}
\label{lemma: counting}

\begin{enumerate}

\item
If the $H(\bar{F})$-conjugacy class $\{\gamma_{H}\}_{H(\bar{F})}$ of $\gamma_{H}$ in $H(F)$ maps to the $\theta$-twisted $G(\bar{F})$-conjugacy class $\{\gamma_{G}\}^{\theta}_{G(\bar{F})}$ of $\gamma_{G}$ with respect to $\xi$, then there exists $w \in W_{H} \backslash W_{G^{\theta}}(H, M)/W_{M^{\theta}}$ such that some $H(\bar{F})$-conjugate $\gamma_{M'_{w}}$ of $\gamma_{H}$ is contained in $M'_{w}(F)$, and the $M'_{w}(\bar{F})$-conjugacy class $\{\gamma_{M'_{w}}\}_{M'_{w}(\bar{F})}$ of $\gamma_{M'_{w}}$ maps to the $\theta$-twisted $M(\bar{F})$-conjugacy class $\{\gamma_{M}\}^{\theta}_{M(\bar{F})}$ of $\gamma_{M}$ with respect to $\xi_{M'_{w}}$. 

\item The correspondence in (1) gives a bijection between $H(\bar{F})$-conjugacy classes $\{\gamma_{H}\}_{H(\bar{F})}$ of $\gamma_{H}$ in $H(F)$ that map to $\{\gamma_{G}\}^{\theta}_{G(\bar{F})}$, and triples $(P'_{w}, \xi_{M'_{w}}, \{\gamma_{M'_{w}}\}_{M'_{w}(\bar{F})})$ indexed by $w \in W_{H} \backslash W_{G^{\theta}}(H, M)/W_{M^{\theta}}$, where $\{\gamma_{M'_{w}}\}_{M'_{w}(\bar{F})}$ maps to $\{\gamma_{M}\}^{\theta}_{M(\bar{F})}$.

\end{enumerate}

\end{lemma}

\begin{proof}

We fix a $\theta$-stable pair $(B_{M}, T_{M})$ in $M$, where $B_{M}$ is a Borel subgroup of $M$ defined over $\bar{F}$ and $T_{M} \subseteq B_{M}$ is a maximal torus defined over $F$, such that there exists $\lif{\gamma}_{M} = m^{-1}\gamma_{M} \theta(m) \in T_{M}$ for $m \in M(\bar{F})$ satisfying $N_{\theta}(\lif{\gamma}_{M}) \in (T_{M})_{\theta}(F)$, where $N_{\theta}: T_{M} \rightarrow (T_{M})_{\theta}$. Let $B = B_{M}N$ and $T = T_{M}$, then $(B, T)$ is a $\theta$-stable pair in $G$, and we fix $\lif{\gamma}_{G} = \lif{\gamma}_{M}$. 

Suppose $\{\gamma_{H}\}_{H(\bar{F})}$ maps to $\{\gamma_{G}\}^{\theta}_{G(\bar{F})}$, let $T_{H} = \Cent(\gamma_{H}, H)$, then there exists an admissible embedding $\iota: T_{H} \rightarrow T_{\theta}$ defined over $F$, where $T_{\theta}$ is the $\theta$-coinvariant group of $T$, such that it sends $\gamma_{H}$ to $N_{\theta}(\lif{\gamma}_{G})$. The $\D{\theta}$-stable $\Gal{F}$-splitting of $\D{G}$ induces a $\D{\theta}$-stable $\Gal{F}$-splitting $(\mathcal{B}_{M}, \mathcal{T}_{M}, \{\mathcal{X}_{\alpha_{M}}\})$ of $\D{M}$. Since $T = T_{M}$ is contained in $M$, we can choose isomorphism $(\D{T}^{\D{\theta}})^{0} \rightarrow (\mathcal{T}_{M}^{\D{\theta}})^{0}  = (\mathcal{T}_{G}^{\D{\theta}})^{0}$ up to $W_{M^{\theta}}$-conjugation. Then we get an element $w \in W_{G^{\theta}}$ defined by the following diagram 
\[
\xymatrix{ \D{T}_{H} \ar[d] \ar[r] & (\D{T}^{\D{\theta}})^{0} \ar[d] \ar[r] & (\mathcal{T}_{M}^{\D{\theta}})^{0} \ar@{=}[d] \\
\mathcal{T}_{H} \ar[r]^{\xi} & (\mathcal{T}_{G}^{\D{\theta}})^{0}  & (\mathcal{T}_{G}^{\D{\theta}})^{0} \ar[l]_{w}. 
}
\]
Recall $\L{M} = \Cent(S, \L{G})$ for $S \subseteq (\mathcal{T}_{G}^{\D{\theta}})^{0}$, so $S$ is $\Gal{F}$-invariant. Since $(\D{T}^{\D{\theta}})^{0} \rightarrow  (\mathcal{T}_{M}^{\D{\theta}})^{0}$ is defined by conjugation in $M$, the image of $S$ in $(\D{T}^{\D{\theta}})^{0}$ is also $\Gal{F}$-invariant. Note $T_{H} \rightarrow T_{\theta}$ is defined over $F$, so the image of $S$ in $\D{T}_{H}$ is $\Gal{F}$-invariant. It follows $\Cent(w(S), \L{H}) \rightarrow W_{F}$ is surjective, and hence $w \in W_{G^{\theta}}(H, M)$. The $\Gal{F}$-splitting of $\D{H}$ induces a $\Gal{F}$-splitting $(\mathcal{B}_{M'_{w}}, \mathcal{T}_{M'_{w}}, \{\mathcal{X}_{\alpha_{M'_{w}}}\})$ of $\D{M}'_{w}$. Then we have the following diagram
\[
\xymatrix{ \mathcal{T}_{M'_{w}} \ar[d] \ar[r]^{\xi_{M'_{w}}} & (\mathcal{T}_{M}^{\D{\theta}})^{0} \ar[d]^{w} \\
\mathcal{T}_{H} \ar[r]^{\xi} & (\mathcal{T}_{G}^{\D{\theta}})^{0}. 
}
\]
The composition
\[
\D{T}_{H} \rightarrow \mathcal{T}_{H} \rightarrow \mathcal{T}_{M'_{w}}
\]
induces a homomorphism $\eta: T_{H} \rightarrow M'_{w}$ by $H$-conjugation, which is determined up to $M'_{w}$-conjugation. Since the image $S'$ of $S$ in $\mathcal{T}_{M'_{w}}$ is $\Gal{F}$-invariant and $\L{M'_{w}} = \Cent(S', \L{H})$, one can show for any $\sigma \in \Gal{F}$, there exists $m' \in M'_{w}(\bar{F})$ such that $\sigma(\eta) \circ \eta^{-1} = \Int(m')$. By \cite[Corollary 2.2]{Kottwitz:1982}, one can choose $\eta$ to be defined over $F$. Let $T_{M'_{w}} = \eta(T_{H})$ and $\gamma_{M'_{w}} = \eta(\gamma_{H})$. It follows from the following commutative diagram
\[
\xymatrix{ \D{T}_{M'_{w}} \ar[r] \ar[d] & \mathcal{T}_{M'_{w}} \ar[d] \ar[r]^{\xi_{M'_{w}}} & (\mathcal{T}_{M}^{\D{\theta}})^{0} \ar[d]^{w} \ar[r] & (\D{T}_{M}^{\D{\theta}})^{0} \ar@{=}[d] \\
\D{T}_{H} \ar[r] & \mathcal{T}_{H} \ar[r]^{\xi} & (\mathcal{T}_{G}^{\D{\theta}})^{0} \ar[r] & (\D{T}^{\D{\theta}})^{0} 
}
\]
that
\[
\xymatrix{
T_{M'_{w}} \ar[r]^{\eta^{-1}} & T_{H} \ar[r]^{\iota \quad \quad } & T_{\theta} = (T_{M})_{\theta}
}
\] 
is an admissible embedding with respect to $\xi_{M'_{w}}$. So $\{\gamma_{M'_{w}}\}_{M'_{w}(\bar{F})}$ maps to $\{\gamma_{M}\}_{M(\bar{F})}^{\theta}$. This proves part (1) of the lemma.

To prove part (2), we first want to show the map obtained above by sending $\{\gamma_{H}\}_{H(\bar{F})}$ to the triple $(P'_{w}, \xi_{M'_{w}}, \{\gamma_{M'_{w}}\}_{M'_{w}(\bar{F})})$ is well-defined. If we fix the admissible embedding $T_{H} \rightarrow T_{\theta}$, but change $\D{T}_{H} \rightarrow \mathcal{T}_{H}$ by $W_{H}$-conjugation and $(\D{T}^{\D{\theta}})^{0} \rightarrow (\mathcal{T}_{M}^{\D{\theta}})^{0}$ by $W_{M^{\theta}}$-conjugation, one can see $w$ is well-defined in $W_{H} \backslash W_{G^{\theta}}(H, M)/W_{M^{\theta}}$. Moreover, $\{\gamma_{M'_{w}}\}_{M'_{w}(\bar{F})}$ is uniquely determined. Now let us vary the admissible embedding $\iota$ by choosing $\gamma'_{H} \in \{\gamma_{H}\}_{H(\bar{F})} \cap H(F)$ and $\iota': T'_{H} \rightarrow T_{\theta}$ defined over $F$, such that $\gamma'_{H} \in T'_{H}(F)$ is again sent to $N_{\theta}(\lif{\gamma}_{G})$. Suppose $\gamma'_{H} = h\gamma_{H}h^{-1}$ for $h \in H(\bar{F})$. We claim 
\[
\iota^{-1} \circ \iota' : T'_{H} \rightarrow T_{H}
\]
can be given by conjugation in $H$. To see this, we consider the following commutative diagram
\[
\xymatrix{ T_{H} \ar[r]  \ar[d]^{\iota} & T'_{H} \ar[r]^{\Int(h)^{-1}} \ar[d]^{\iota'} & T_{H} \ar[d]^{\iota} \\
T_{\theta} \ar@{=}[r] & T_{\theta} \ar[r]^{w'} & T_{\theta}.
}
\]
Here $w'$ can be viewed as in $W(G, T)^{\theta}$. Since $\theta$ preserves a splitting of $G$, we can choose $g \in G(\bar{F})$ such that $\theta(g) = g$ and $\Int(g) = w'$. Since $w'$ fixes $N_{\theta}(\lif{\gamma}_{G})$, there exists $t \in T(\bar{F})$ such that 
\[
g \lif{\gamma}_{G}g^{-1} = \lif{\gamma}_{G}t^{-1}\theta(t).
\]
Then
\[
(tg)\lif{\gamma}_{G}\theta(tg)^{-1} = \lif{\gamma}_{G}. 
\]
Since $\gamma_{G}$ is strongly $\theta$-regular, $tg \in T(\bar{F})$ and hence $g \in T(\bar{F})$. It follows $w' = 1$, and we have $\iota^{-1} \circ \iota' = \Int(h)^{-1}$.

As a consequence of our claim, we have
\[
\xymatrix{ 
\D{T}'_{H} \ar[d] \ar[r] & \D{T}_{H} \ar[d] \ar[r] & (\D{T}^{\D{\theta}})^{0} \ar[d] \ar[r] & (\mathcal{T}_{M}^{\D{\theta}})^{0} \ar@{=}[d] \\
\mathcal{T}_{H} \ar@{=}[r] & \mathcal{T}_{H} \ar[r]^{\xi} & (\mathcal{T}_{G}^{\D{\theta}})^{0}  & (\mathcal{T}_{G}^{\D{\theta}})^{0} \ar[l]_{w} 
}
\]
This means $w$ is unchanged in $W_{H} \backslash W_{G^{\theta}}(H, M)/W_{M^{\theta}}$. Moreover, we have 
\[
\xymatrix{ 
\D{T}_{M'_{w}} \ar@{=}[r] \ar[d] & \D{T}_{M'_{w}} \ar[r] \ar[d] & \mathcal{T}_{M'_{w}} \ar[d] \ar[r]^{\xi_{M'_{w}}} & (\mathcal{T}_{M}^{\D{\theta}})^{0} \ar[d]^{w} \ar[r] & (\D{T}_{M}^{\D{\theta}})^{0} \ar@{=}[d] \\
\D{T}'_{H} \ar[r] & \D{T}_{H} \ar[r] & \mathcal{T}_{H} \ar[r]^{\xi} & (\mathcal{T}_{G}^{\D{\theta}})^{0} \ar[r] & (\D{T}^{\D{\theta}})^{0} 
}
\]
It is easy to see that one still gets the same $\{\gamma_{M'_{w}}\}_{M'_{w}(\bar{F})}$.

Next we want to construct the inverse. For any triple $(P'_{w}, \xi_{M'_{w}}, \{\gamma_{M'_{w}}\}_{M'_{w}(\bar{F})})$, let $T_{M'_{w}} = \Cent(\gamma_{M'_{w}}, M'_{w})$. Since $\{\gamma_{M'_{w}}\}_{M'_{w}(\bar{F})}$ maps to $\{\gamma_{M}\}^{\theta}_{M(\bar{F})}$, we have an admissible embedding $T_{M'_{w}} \rightarrow (T_{M})_{\theta}$ with respect to $\xi_{M'_{w}}$, which sends $\gamma_{M'_{w}}$ to $\lif{\gamma}_{M}$. Let $\gamma_{H} = \gamma_{M'_{w}}$ and $T_{H} = T_{M'_{w}}$. It follows from the following diagram 
\[
\xymatrix{ 
\D{T}_{M'_{w}} \ar[r] \ar@{=}[d] & \mathcal{T}_{M'_{w}} \ar[d] \ar[r]^{\xi_{M'_{w}}} & (\mathcal{T}_{M}^{\D{\theta}})^{0} \ar[d]^{w} \ar[r] & (\D{T}_{M}^{\D{\theta}})^{0} \ar@{=}[d] \\
\D{T}_{H} \ar[r] & \mathcal{T}_{H} \ar[r]^{\xi} & (\mathcal{T}_{G}^{\D{\theta}})^{0} \ar[r] & (\D{T}^{\D{\theta}})^{0} 
}
\]
that $T_{H} \rightarrow T_{\theta}$ is an admissible embedding with respect to $\xi$. So $\{\gamma_{H}\}_{H(\bar{F})}$ maps to $\{\gamma_{G}\}_{G(\bar{F})}^{\theta}$. In this way, we send $(P'_{w}, \xi_{M'_{w}}, \{\gamma_{M'_{w}}\}_{M'_{w}(\bar{F})})$ to $\{\gamma_{H}\}_{H(\bar{F})}$, where $\gamma_{H} = \gamma_{M'_{w}}$. Finally, it is easy to check that this map does give the inverse.

\end{proof}

As a consequence of this lemma, we can rewrite the right hand side of \eqref{eq: spectral transfer} as 
\[
\sum_{w} \sum_{\gamma_{M'_{w}} \rightarrow \gamma_{M}} \frac{D_{H}(\gamma_{M'_{w}})^{2}}{D_{G^{\theta}}(\gamma_{M})^{2}} \Delta_{G, H}(\gamma_{M'_{w}}, \gamma_{M}) \Theta^{H}(\gamma_{M'_{w}}).
\]
So the next step is to write the summands in terms of $M$ and $M'_{w}$. First, it is easy to check from the definition of transfer factors that
\[
\frac{D_{H}(\gamma_{M'_{w}})}{D_{G^{\theta}}(\gamma_{M})} \Delta_{G, H}(\gamma_{M'_{w}}, \gamma_{M}) = \frac{D_{M'_{w}}(\gamma_{M'_{w}})}{D_{M^{\theta}}(\gamma_{M})} \Delta_{M, M'_{w}}(\gamma_{M'_{w}}, \gamma_{M})
\]
Secondly, there is a natural homomorphism from $A_{M}$ to $A_{M'_{w}}$ with respect to $\xi_{M'_{w}}$ (see \cite{KottwitzShelstad:1999}). For $z_{M} \in A_{M}(F)$, let $z_{M'_{w}}$ be its image in $A_{M'_{w}}(F)$, then the $M'_{w}(\bar{F})$-conjugacy class of $z_{M'_{w}}\gamma_{M'_{w}}$ maps to the $M(\bar{F})$-conjugacy class of $z_{M}\gamma_{M}$. Moreover, $sup|\alpha'(z_{M'_{w}})|$ for roots $\a'$ in $N'_{w}$ is less than $sup|\prod_{i=1}^{l}\theta^{i}(\alpha)(z_{M})|$ for roots $\alpha$ in $N$.

Let
\[
\Theta^{M^{\theta}}  = \sum_{i} c_{i} \cdot \Theta^{M^{\theta}}_{\r_{i}}
\]
\[
(\text{ resp. } \Theta^{M'_{w}} = \sum_{j} d_{w, j} \cdot \Theta^{M'_{w}}_{\r'_{w, j}})
\]
be the (resp. stable) finite linear combination of (twisted) characters of $M$ (resp. $M'_{w}$) obtained from $\Theta^{G^{\theta}}$ (resp. $\Theta^{H}$) by taking the unnormalized Jacquet modules. If we take $z_{M} \in A_{M}(F)$ so that $sup|\alpha(z_{M})|$ is sufficiently small for all roots $\alpha$ in $N$, then by Corollary~\ref{cor: Casselman's formula}, 
\[
\Theta^{G^{\theta}}(z_{M}\gamma_{M}) = \Theta^{M^{\theta}}(z_{M}\gamma_{M}),
\]
and
\[
\Theta^{H}(z_{M'_{w}}\gamma_{M'_{w}}) = \Theta^{M'_{w}}(z_{M'_{w}}\gamma_{M'_{w}}).
\]
Besides, it is easy to verify
\[
D_{G^{\theta}}(z_{M}\gamma_{M}) = \delta_{P}(z_{M}\gamma_{M})^{-1/2} \cdot D_{M^{\theta}}(z_{M}\gamma_{M}),
\]
and
\[
D_{H}(z_{M'_{w}}\gamma_{M'_{w}}) = \delta_{P'}(z_{M'_{w}}\gamma_{M'_{w}})^{-1/2} \cdot D_{M'_{w}}(z_{M'_{w}}\gamma_{M'_{w}}).
\]
Putting all these together, we get
\begin{align*}
&\delta_{P}(z_{M}\gamma_{M})^{-1/2} \Theta^{M^{\theta}}(z_{M}\gamma_{M}) \\
= & \sum_{w} \sum_{\gamma_{M'_{w}} \rightarrow \gamma_{M}} \frac{D_{M'_{w}}(z_{M'_{w}}\gamma_{M'_{w}})^{2}}{D_{M^{\theta}}(z_{M}\gamma_{M})^{2}} \Delta_{M, M'_{w}}(z_{M'_{w}}\gamma_{M'_{w}}, z_{M}\gamma_{M}) \delta_{P'_{w}}(z_{M'_{w}}\gamma_{M'_{w}})^{-1/2} \Theta^{M'_{w}}(z_{M'_{w}}\gamma_{M'_{w}}).
\end{align*}
Since $z_{M}, z_{M'_{w}}$ are in the centres of $M$ and $M'_{w}$ respectively, we have 
\begin{align*}
\Delta_{M, M'_{w}}(z_{M'_{w}} \gamma_{M'_{w}}, z_{M} \gamma_{M}) & = \chi'_{w}(z_{M}) \Delta_{M, M'_{w}}(\gamma_{M'_{w}}, \gamma_{M}), \\
\delta_{P}(z_{M}\gamma_{M})^{-1/2} \Theta^{M^{\theta}}_{\r_{i}}(z_{M}\gamma_{M})  & = \zeta_{M, i}(z_{M}) \delta_{P}(\gamma_{M})^{-1/2} \Theta^{M^{\theta}}_{\r_{i}}(\gamma_{M}), \\
\delta_{P'_{w}}(z_{M'_{w}}\gamma_{M'_{w}})^{-1/2} \Theta^{M'_{w}}_{\r'_{w, j}}(z_{M'_{w}}\gamma_{M'_{w}}) & = \zeta_{M'_{w}, j}(z_{M'_{w}}) \delta_{P'_{w}}(\gamma_{M'_{w}})^{-1/2} \Theta^{M'_{w}}_{\r'_{w, j}}(\gamma_{M'_{w}}), \\
D_{M^{\theta}}(z_{M}\gamma_{M}) & = D_{M^{\theta}}(\gamma_{M}), \\
D_{M'_{w}}(z_{M'_{w}}\gamma_{M'_{w}}) & = D_{M'_{w}}(\gamma_{M'_{w}}),
\end{align*}
where $\chi'_{w}, \zeta_{M, i}, \zeta_{M'_{w}, j}$ are the corresponding central characters. Hence
\begin{align*}
& \sum_{i} (c_{i} \cdot \delta_{P}(\gamma_{M})^{-1/2}  \Theta^{M^{\theta}}_{\r_{i}}(\gamma_{M})) \cdot \zeta_{M, i}(z_{M}) \\
= & \sum_{w}  \sum_{j}  \big(  d_{w, j} \sum_{\gamma_{M'_{w}} \rightarrow \gamma_{M}} \frac{D_{M'_{w}}(\gamma_{M'_{w}})^{2}}{D_{M^{\theta}}(\gamma_{M})^{2}} \Delta_{M, M'_{w}}(\gamma_{M'_{w}}, \gamma_{M})  \delta_{P'_{w}}(\gamma_{M'_{w}})^{-1/2}  \Theta^{M'_{w}}_{\r'_{w, j}}(\gamma_{M'_{w}})\big) \cdot \chi'_{w}(z_{M}) \zeta_{M'_{w}, j}(z_{M'_{w}}).
\end{align*}
Let 
\[
a_{M, i} = c_{i} \cdot \delta_{P}(\gamma_{M})^{-1/2} \Theta^{M^{\theta}}_{\r_{i}}(\gamma_{M}),
\]
and
\[
b_{w, j} = d_{w, j} \sum_{\gamma_{M'_{w}} \rightarrow \gamma_{M}} \frac{D_{M'_{w}}(\gamma_{M'_{w}})^{2}}{D_{M^{\theta}}(\gamma_{M})^{2}} \Delta_{M, M'_{w}}(\gamma_{M'_{w}}, \gamma_{M}) \delta_{P'_{w}}(\gamma_{M'_{w}})^{-1/2}  \Theta^{M'_{w}}_{\r'_{w, j}}(\gamma_{M'_{w}}).
\]
We also write $\chi_{M, i}(z_{M}) = \zeta_{M, i}(z_{M})$ and $\chi_{w, j}(z_{M}) = \chi'_{w}(z_{M}) \zeta_{M'_{w}, j}(z_{M'_{w}})$. Then we can get a short expression
\begin{align}
\label{eq: central character identity}
\sum_{i} a_{M, i} \cdot \chi_{M, i}(z_{M}) = \sum_{w} \sum_{j} b_{w, j} \cdot \chi_{w, j}(z_{M}),
\end{align}
and it suffices for us to show this holds when $z_{M} = 1$, i.e.,
\[
\sum_{i} a_{M, i} = \sum_{w} \sum_{j} b_{w, j}.
\] 
In fact, we can choose $z_{M} \in F^{\times} \hookrightarrow A_{M}(F)$ such that \eqref{eq: central character identity} holds provided $|z_{M}| < q_{F}^{-k}$ for some positive integer $k$, where $q_{F}$ is the order of the residue field of $F$. Then it is enough to have the following lemma.

\begin{lemma}
For quasicharacters $\chi_{i}$ of $F^{\times}$ and complex numbers $a_{i}$, if 
\[
\sum_{i=1}^{r}a_{i} \chi_{i} (z) = 0
\]
provided $|z|< q_{F}^{-k}$ for some positive integer $k$, then 
\[
\sum_{i = 1}^{r} a_{i} = 0.
\]
\end{lemma}

\begin{proof}
Suppose $\chi_{i}$ are distinct, we claim $a_{i} = 0$ for $1 \leqslant i \leqslant r$. It is clear that this lemma will follow from our claim. So next we will show the claim by induction on $r$. When $r = 1$, there is nothing to show. In general, let us first assume all $\chi_{i}$ are unramified. We choose $z_{0} \in F^{\times}$ such that $|z_{0}| < q_{F}^{-k}$, and denote $\chi_{i}(z_{0})$ by $C_{i}$. Then 
\[
\sum_{i}a_{i} C^{j}_{i} = \sum_{i} a_{i} \chi_{i}(z_{0}^{j}) = 0
\]
for any positive integer $j$. In particular, $\{a_{i}\}$ forms a solution of the linear system of equations defined by the matrix $\{C_{i}^{j}\}^{T}_{r \times r}$, where we let $1 \leqslant j \leqslant r$. Since $|\det(\{C_{i}^{j}\}_{r \times r})| = |\prod_{i}C_{i}| \cdot \prod_{i \leqslant j} |C_{i} - C_{j}| \neq 0$, then $a_{i}$ have to be all zero. Now suppose some $\chi_{i}$ is ramified, we can replace $\chi_{i}$ by $\chi'_{i} := \chi_{i} / \chi_{1}$ for all $i$. If $\chi'_{i}$ are all unramified, then we are back to the previous case. If $\chi'_{i_{0}}$ is ramified for some $i_{0} > 1$, then we can choose some unit of the ring of integers of $F$ such that $\chi'_{i_{0}}(u) \neq 1$. By subtracting $\chi_{1}(u)\sum_{i}a_{i}\chi_{i}(z)$ from $\sum_{i}a_{i}\chi_{i}(uz)$, we get
\[
\sum_{i > 1}a_{i}(\chi_{i}(u) - \chi_{1}(u))\chi_{i}(z) = 0
\]
provided $|z|< q_{F}^{-k}$. By induction, we have $a_{i}(\chi_{i}(u) - \chi_{1}(u)) = 0$ for $i > 1$. Since $\chi_{i_{0}}(u) - \chi_{1}(u) \neq 0$ by our assumption,  this implies $a_{i_{0}} = 0$. Hence 
\[
\sum_{i \neq i_{0}} a_{i}\chi_{i}(z) = 0
\] 
provided $|z|< q_{F}^{-k}$. By induction again, we have $a_{i} = 0$ for $i \neq i_{0}$. This finishes the proof of the claim.

\end{proof}

\bibliographystyle{amsalpha}

\bibliography{reps}

\providecommand{\bysame}{\leavevmode\hbox to3em{\hrulefill}\thinspace}
\providecommand{\MR}{\relax\ifhmode\unskip\space\fi MR }
\providecommand{\MRhref}[2]{%
  \href{http://www.ams.org/mathscinet-getitem?mr=#1}{#2}
}
\providecommand{\href}[2]{#2}
\begin{thebibliography}{JPSS83}

\bibitem[Art96]{Arthur:1996}
J.~Arthur, \emph{On local character relations}, Selecta Math. (N.S.) \textbf{2}
  (1996), no.~4, 501--579.

\bibitem[Art13]{Arthur:2013}
\bysame, \emph{The endoscopic classification of representations: orthogonal and
  symplectic groups}, Colloquium Publications, vol.~61, American Mathematical
  Society, 2013.

\bibitem[Aub95]{Aubert:1995}
A.-M. Aubert, \emph{Dualit\'e dans le groupe de {G}rothendieck de la
  cat\'egorie des repr\'esentations lisses de longueur finie d'un groupe
  r\'eductif {$p$}-adique}, Trans. Amer. Math. Soc. \textbf{347} (1995), no.~6,
  2179--2189.

\bibitem[Ban99]{Ban:1999}
D.~Ban, \emph{Parabolic induction and {J}acquet modules of representations of
  {${\rm O}(2n,F)$}}, Glas. Mat. Ser. III \textbf{34(54)} (1999), no.~2,
  147--185. \MR{1739616}

\bibitem[BDK86]{B-D-K:1986}
J.~Bernstein, P.~Deligne, and D.~Kazhdan, \emph{Trace {P}aley-{W}iener theorem
  for reductive {$p$}-adic groups}, J. Analyse Math. \textbf{47} (1986),
  180--192. \MR{874050}

\bibitem[BW00]{BorelWallach:2000}
A.~Borel and N.~Wallach, \emph{Continuous cohomology, discrete subgroups, and
  representations of reductive groups}, second ed., Mathematical Surveys and
  Monographs, vol.~67, American Mathematical Society, Providence, RI, 2000.
  \MR{1721403}

\bibitem[BZ77]{BernsteinZelevinsky:1977}
I.~N. Bernstein and A.~V. Zelevinsky, \emph{Induced representations of
  reductive {p}-adic groups. {I}}, Ann. Sci. \'Ecole Norm. Sup. (4) \textbf{10}
  (1977), no.~4, 441--472.

\bibitem[Cas77]{Casselman:1977}
W.~Casselman, \emph{Characters and {J}acquet modules}, Math. Ann. \textbf{230}
  (1977), no.~2, 101--105.

\bibitem[Clo87]{Clozel:1987}
L.~Clozel, \emph{Characters of nonconnected, reductive {$p$}-adic groups},
  Canad. J. Math. \textbf{39} (1987), no.~1, 149--167.

\bibitem[HC99]{H-C:1999}
Harish-Chandra, \emph{Admissible invariant distributions on reductive
  {$p$}-adic groups}, University Lecture Series, vol.~16, American Mathematical
  Society, Providence, RI, 1999, Preface and notes by Stephen DeBacker and Paul
  J. Sally, Jr.

\bibitem[Hen00]{Henniart:2000}
G.~Henniart, \emph{Une preuve simple des conjectures de {L}anglands pour {${\rm
  GL}(n)$} sur un corps {$p$}-adique}, Invent. Math. \textbf{139} (2000),
  no.~2, 439--455.

\bibitem[Hir04]{Hiraga:2004}
K.~Hiraga, \emph{On functoriality of {Z}elevinski involutions}, Compositio
  Math. \textbf{140} (2004), 1625--1656.

\bibitem[HT01]{HarrisTaylor:2001}
M.~Harris and R.~Taylor, \emph{The geometry and cohomology of some simple
  {S}himura varieties}, Annals of Mathematics Studies, vol. 151, Princeton
  University Press, Princeton, NJ, 2001, With an appendix by Vladimir G.
  Berkovich.

\bibitem[Jan06]{Jantzen:2006}
C.~Jantzen, \emph{Jacquet modules of induced representations for {$p$}-adic
  special orthogonal groups}, J. Algebra \textbf{305} (2006), no.~2, 802--819.

\bibitem[JPSS83]{JPSS:1983}
H.~Jacquet, I.~I. Piatetskii-Shapiro, and J.~A. Shalika, \emph{Rankin-{S}elberg
  convolutions}, Amer. J. Math. \textbf{105} (1983), no.~2, 367--464.

\bibitem[Kaz86]{Kazhdan:1986}
D.~Kazhdan, \emph{Cuspidal geometry of {$p$}-adic groups}, J. Analyse Math.
  \textbf{47} (1986), 1--36. \MR{874042}

\bibitem[Kot82]{Kottwitz:1982}
R.~E. Kottwitz, \emph{Rational conjugacy classes in reductive groups}, Duke
  Math. J. \textbf{49} (1982), no.~4, 785--806.

\bibitem[KS99]{KottwitzShelstad:1999}
R.~E. Kottwitz and D.~Shelstad, \emph{Foundations of twisted endoscopy},
  Ast\'erisque \textbf{55} (1999), no.~255, vi+190.

\bibitem[M{\oe}g02]{Moeglin:2002}
C.~M{\oe}glin, \emph{Sur la classification des s\'eries discr\`etes des groupes
  classiques {$p$}-adiques: param\`etres de {L}anglands et exhaustivit\'e}, J.
  Eur. Math. Soc. (JEMS) \textbf{4} (2002), no.~2, 143--200.

\bibitem[M{\oe}g11]{Moeglin1:2011}
\bysame, \emph{Multiplicit\'e 1 dans les paquets d'{A}rthur aux places
  {$p$}-adiques}, On certain {$L$}-functions, Clay Math. Proc., vol.~13, Amer.
  Math. Soc., Providence, RI, 2011, pp.~333--374.

\bibitem[M{\oe}g14]{Moeglin:2014}
\bysame, \emph{Paquets stables des s\'eries discr\`etes accessibles par
  endoscopie tordue; leur param\`etre de {L}anglands}, Automorphic forms and
  related geometry: assessing the legacy of {I}. {I}. {P}iatetski-{S}hapiro,
  Contemp. Math., vol. 614, Amer. Math. Soc., Providence, RI, 2014,
  pp.~295--336.

\bibitem[MT02]{MoeglinTadic:2002}
C.~M{\oe}glin and M.~Tadi{\'c}, \emph{Construction of discrete series for
  classical {$p$}-adic groups}, J. Amer. Math. Soc. \textbf{15} (2002), no.~3,
  715--786 (electronic).

\bibitem[MT11]{MuicTadic:2011}
G.~Mui{\'c} and M.~Tadi{\'c}, \emph{Unramified unitary duals for split
  classical {$p$}-adic groups; the topology and isolated representations}, On
  certain {$L$}-functions, Clay Math. Proc., vol.~13, Amer. Math. Soc.,
  Providence, RI, 2011, pp.~375--438.

\bibitem[MW89]{MW:1989}
C.~M{\oe}glin and J.-L. Waldspurger, \emph{Le spectre r\'esiduel de {${\rm
  GL}(n)$}}, Ann. Sci. \'Ecole Norm. Sup. (4) \textbf{22} (1989), no.~4,
  605--674.

\bibitem[MW06]{MW:2006}
\bysame, \emph{Sur le transfert des traces d'un groupe classique {$p$}-adique
  \`a un groupe lin\'eaire tordu}, Selecta Math. (N.S.) \textbf{12} (2006),
  no.~3-4, 433--515.

\bibitem[MW16]{MW:2016}
\bysame, \emph{Stabilisation de la formule des traces tordue}, Progress in
  Mathematics, vol. 316/317, Birkh\"auser Basel, 2016.

\bibitem[Ng{\^o}10]{Ngo:2010}
B.~C. Ng{\^o}, \emph{Le lemme fondamental pour les alg\`ebres de {L}ie}, Publ.
  Math. Inst. Hautes \'Etudes Sci. (2010), no.~111, 1--169.

\bibitem[Rog88]{Rogawski:1988}
J.~D. Rogawski, \emph{Trace {P}aley-{W}iener theorem in the twisted case},
  Trans. Amer. Math. Soc. \textbf{309} (1988), no.~1, 215--229. \MR{957068}

\bibitem[Sch13]{Scholze:2013}
P.~Scholze, \emph{The local {L}anglands correspondence for {${GL}_n$} over
  {$p$}-adic fields}, Invent. Math. \textbf{192} (2013), no.~3, 663--715.

\bibitem[Sha81]{Shahidi:1981}
F.~Shahidi, \emph{On certain {$L$}-functions}, Amer. J. Math. \textbf{103}
  (1981), no.~2, 297--355.

\bibitem[Sha92]{Shahidi:1992}
\bysame, \emph{Twisted endoscopy and reducibility of induced representations
  for {$p$}-adic groups}, Duke Math. J. \textbf{66} (1992), no.~1, 1--41.

\bibitem[Sil79]{Silberger:1979}
A.~J. Silberger, \emph{Introduction to harmonic analysis on reductive
  {$p$}-adic groups}, Mathematical Notes, vol.~23, Princeton University Press,
  Princeton, N.J., 1979, Based on lectures by Harish-Chandra at the Institute
  for Advanced Study, 1971--1973.

\bibitem[Tad86]{Tadic:1986}
M.~Tadi{\'c}, \emph{Classification of unitary representations in irreducible
  representations of general linear group (non-{A}rchimedean case)}, Ann. Sci.
  \'Ecole Norm. Sup. (4) \textbf{19} (1986), no.~3, 335--382.

\bibitem[Tad09]{Tadic1:2009}
\bysame, \emph{On reducibility and unitarizability for classical {$p$}-adic
  groups, some general results}, Canad. J. Math. \textbf{61} (2009), no.~2,
  427--450.

\bibitem[Wal95]{Waldspurger:1995}
J.-L. Waldspurger, \emph{Une formule des traces locale pour les alg\`ebres de
  {L}ie {$p$}-adiques}, J. Reine Angew. Math. \textbf{465} (1995), 41--99.

\bibitem[Wal97]{Waldspurger:1997}
\bysame, \emph{Le lemme fondamental implique le transfert}, Compositio Math.
  \textbf{105} (1997), no.~2, 153--236.

\bibitem[Wal03]{Waldspurger:2003}
\bysame, \emph{La formule de {P}lancherel pour les groupes {$p$}-adiques
  (d'apr\`es {H}arish-{C}handra)}, J. Inst. Math. Jussieu \textbf{2} (2003),
  no.~2, 235--333.

\bibitem[Wal06]{Waldspurger:2006}
\bysame, \emph{Endoscopie et changement de caract\'eristique}, J. Inst. Math.
  Jussieu \textbf{5} (2006), no.~3, 423--525.

\bibitem[Wal08]{Waldspurger:2008}
\bysame, \emph{L'endoscopie tordue n'est pas si tordue}, Mem. Amer. Math. Soc.
  \textbf{194} (2008), no.~908, x+261.

\bibitem[Wal17]{Waldspurger:2017}
\bysame, \emph{La formule des traces locale tordue}, {\`a} para{\^i}tre aux
  Memoirs of AMS (2017).

\bibitem[Zel80]{Zelevinsky:1980}
A.~V. Zelevinsky, \emph{Induced representations of reductive {p}-adic groups.
  {II}. {O}n irreducible representations of {${\rm GL}(n)$}}, Ann. Sci. \'Ecole
  Norm. Sup. (4) \textbf{13} (1980), no.~2, 165--210.

\end{thebibliography}

\end{document}